\documentclass[12pt,a4paper,reqno]{amsart}

\usepackage[top=35mm, bottom=35mm, left=30mm, right=30mm]{geometry}
 
\usepackage{array}
\usepackage{booktabs}
\usepackage{microtype}
\usepackage[T1]{fontenc}
\usepackage[hypertexnames=false,colorlinks=true,linkcolor=blue,citecolor=blue,urlcolor=blue]{hyperref}

\usepackage{amsfonts}
\usepackage{amssymb}
\usepackage{amsmath}
\usepackage{amsthm}
\usepackage{mathtools}
\usepackage{enumitem}
\usepackage{mathptmx}
\usepackage{eucal}

\newtheorem{thm}{Theorem}[section]
\newtheorem{coro}[thm]{Corollary}
\newtheorem{prop}[thm]{Proposition}
\newtheorem{lem}[thm]{Lemma}
\newtheorem{defn}[thm]{Definition}
\newtheorem{ques}[thm]{Question}
\theoremstyle{remark}
\newtheorem{rem}[thm]{Remark}
\newtheorem{ex}[thm]{Example}
\numberwithin{equation}{section}

\newtheorem{innercustomthm}{ }
\newenvironment{customthm}[1]
  {\renewcommand\theinnercustomthm{#1}\innercustomthm}
  {\endinnercustomthm}

\newcommand{\N}{\mathbb N}
\newcommand{\Z}{\mathbb Z}
\newcommand{\Q}{\mathbb Q}
\newcommand{\R}{\mathbb R}
\newcommand{\B}{\mathcal B}
\newcommand{\I}{\mathcal I}

\newcommand{\1}{\mathbf 1}
\newcommand{\e}{\mathrm e}
\newcommand{\dd}{\,\mathrm d}
\newcommand{\RME}{\mathrm{RME}}

\title[Sharp Orlicz Endpoints for Spatial-Temporal  Ergodic Averaging]{Sharp Orlicz Endpoints for Spatial-Temporal \\ Ergodic Averaging}

\author[J.~Li]{Jie Li}
\address[Jie Li]{School of Mathematics and Statistics, Jiangsu Normal University, Xuzhou, Jiangsu, 221116, P.R. China}
\email{jiel0516@mail.ustc.edu.cn}

\date{\today}

\subjclass[2020]{37A30, 42B25, 46E30}
\keywords{Ergodic averages along subsequences, spatial--temporal averaging, Orlicz endpoints, restricted maximal estimates, local $\infty$-sweeping out.}

\begin{document}

\begin{abstract}
We study the composition of temporal ergodic averaging with spatial averaging over shrinking metric balls, and determine sharp Orlicz endpoints for the corresponding  unrestricted joint limit. For consecutive Birkhoff averages normalized by $N\Lambda_q(N)$ ($q\ge 0$), the sharp Orlicz endpoint is $L\log_{q+1}L$.  
In particular, the ordinary case $q=0$ yields an $L\log L$ local joint convergence theorem under the Lebesgue differentiation property alone, while $L^1$ fails even on
the Euclidean interval, answering two questions of Young.  The same $L\log_{q+1}L$ endpoint holds for prime averages for every $q\ge 1$.
For arbitrary time sequences, $L\log_qL$ always suffices at the same normalization $N\Lambda_q(N)$ ($q\ge 1$), and this endpoint is sharp in the Orlicz
sense for fixed-base exponential sequences $k^n$ ($k\ge 2$) and for sequences with polynomial ratio separation, including $n!$.
 
The positive results rest on a local stability principle: under the ball Lebesgue differentiation property, pointwise temporal convergence lifts to the local joint limit whenever the associated temporal maximal function admits an $L^1$ majorant. The additional logarithm in the regular-time case comes from lifting restricted logarithmic maximal estimates from sets to general functions. The arbitrary-sequence theorem uses a dyadic decomposition instead. The lower bounds are local $\infty$-sweeping out constructions, obtained by weighted local sweeping out for polynomial-growth regular times, by residue constructions for polynomially ratio-separated sequences, and by digit
constructions for fixed-base exponentials. 
\end{abstract}

\maketitle

\section{Introduction}

Throughout the paper a \emph{metric probability space} means a triple $(X,d,\mu)$, where $(X,d)$ is a metric space and $\mu$ is a Borel probability measure.  
We write $\B_X$ for the  Borel $\sigma$-algebra of $X$. 
 A \emph{measure-preserving system} over $(X,d,\mu)$ is a quadruple $(X,\B_X,\mu,T)$, where $T:X\to X$ is Borel measurable and satisfies $\mu(T^{-1}A)=\mu(A)$ for every $A\in\B_X$.  

\subsection{The local observation problem and motivations}
Let $(X,\B_X,\mu,T)$ be a measure-preserving system. For $h\in L^1(\mu)$, write the spatial average
\[
        \mathcal L_rh(x)
        =
        \frac{1}{\mu(B(x,r))}\int_{B(x,r)}h\,\dd\mu.
\]
Only balls of positive measure are considered, and this convention will not be repeated. 
For the Birkhoff averages
\[
        A_Nf(x)=\frac{1}N\sum_{j=0}^{N-1}f(T^jx),
\]
the pointwise ergodic theorem describes the limiting behaviour of $A_Nf(x)$ at a fixed  point. 
However, in applications one typically observes not an isolated point but a small neighbourhood of it.  
Thus it is natural to consider the compositional behavior of the temporal average observed within a small spatial neighbourhood:
\begin{equation*}\label{eq:intro-local-observation}
        \mathcal L_r(A_Nf)(x) = \frac{1}{\mu(B(x,r))} \int_{B(x,r)} A_Nf(y)\dd \mu(y).
\end{equation*}
The local observation problem asks whether the temporal limit $N\to \infty$ and the spatial differentiation limit $r\downarrow 0$ are compatible in a joint sense, with no relation imposed between the temporal scale and the spatial radius. 

This local observation problem connects three closely related lines of research. The first pertains to the work of Assani and Young \cite{AssaniYoung22-1}, who introduced the so-called ``spatial--temporal differentiation'' problem for temporal averages integrated over shrinking measurable spatial sets. 
Young later considered related questions for locally compact group actions and other spatial families \cite{Young23-1, Young23-2}.  More recently, Young \cite{Young25} proved a ball version for $L^p$ ($p>1$) functions under a weak type $(1,1)$ Hardy--Littlewood maximal inequality for the metric balls, and asked
whether this hypothesis can be replaced by the Lebesgue differentiation property and whether the result extends to $L^1$, see \cite[Questions~9 and~10]{Young25}. 
The ordinary Birkhoff part of our \ref{thm:regular-endpoints} answers both these questions: the Lebesgue differentiation property is sufficient on $L\log L$, while $L^1$ is not sufficient even on $[0,1]$ with its Euclidean metric.

The second  concerns  the rate problem for ergodic sums along prescribed times.    
Quas and Wierdl~\cite{QuasWierdl} studied iterated-logarithmic denominators as the natural scale for universal pointwise growth of ergodic sums along arbitrary time sequences, and obtained sharp results at this scale.
These results are pointwise. Under local observation in the present paper, the same denominators lead to additional difficulties on both sides of the problem.
On the positive side, almost everywhere control of the normalized sums is not enough. To pass through the spatial averaging step,  one needs an integrable majorant for the corresponding temporal supremum. 
This is the role of the restricted logarithmic maximal estimates introduced below.
On the negative side, the pointwise examples of Quas and Wierdl~\cite{QuasWierdl} do not directly give local counterexamples, since large values along one orbit may disappear after averaging over shrinking balls.   Sharpness under local observation therefore requires bad behaviour to persist at arbitrarily small spatial scales. For fixed-base exponentials $k^n$, the cyclic-word construction of Quas and Wierdl \cite{QuasWierdl} serves as the pointwise model, but the local result requires a separate spatial realization.

The third relates the Orlicz endpoint problem for subsequence ergodic averages.   
Perturbation methods of Bellow~\cite{Bellow1989}, Reinhold-Larsson~\cite{ReinholdLarsson1994}, and Wierdl~\cite{Wierld98}
provide a flexible way of constructing subsequences with prescribed convergence and divergence behavior. Parrish~\cite{Parrish2011} adapted this method to treat endpoints in  Orlicz classes, constructing sequences which are good on a prescribed Orlicz class but sweeping out in larger spaces.  
On the negative side, strong sweeping out constructions form the standard endpoint obstruction, and particularly the strong sweeping out property for lacunary sequences and related operators was developed systematically by Akcoglu, Bellow, Jones, Losert, Reinhold--Larsson and Wierdl~\cite{ABJLRW}.
The present paper addresses the Orlicz endpoint problem under local observation. The lower bounds are in the same Orlicz endpoint spirit, but the form of failure is different: 
it is local $\infty$-sweeping out rather than a purely temporal sweeping out statement, and large values must still be seen after the temporal sums have been averaged over shrinking balls.

Prime averages provide a useful Orlicz endpoint test case for the positive theory.  Wierdl \cite{Wierdl88} established pointwise convergence of prime averages for $L^p$ ($p>1$) functions, while LaVictoire \cite{LaVictoire} proved that the primes are universally $L^1$-bad.  
Trojan  \cite{Trojan} later obtained Orlicz endpoint estimates for the prime maximal function.  Motivated by this, we introduce below a restricted logarithmic maximal estimate adapted to the iterated-logarithmic denominators.  
The restricted indicator form of Trojan's estimate implies this condition for the primes.  After logarithmic lifting from indicators to general functions, it gives the $L^1$ temporal maximal majorant needed for local observation, although the ordinary prime maximal operator is not uniform weak type $(1,1)$ over  measure-preserving systems.

\subsection{Main results}

For $q\in\N $ and $t\ge 0$, define the regularized iterated logarithms by
\[
L_1(t)=\log(\e+t),\quad L_{q+1}(t)=\log\bigl(\e+L_q(t)\bigr),
\]
and set
\[
\Lambda_0(t)=    1, \quad \Lambda_q(t)=\prod_{j=1}^{q}L_j(t).
\]
Given a finite-valued Young function $\Phi$, we write $L^\Phi$ for the corresponding Luxemburg Orlicz space. Define
\[
        \Psi_q(t)=\int_0^t (1+L_q(u))\dd  u ,
\]
and write $ L\log_qL=L^{\Psi_q}.$ Then $L\log_1L$ is the usual Zygmund class $L\log L$, and $\Psi_q(t)$ is comparable to $t(1+L_q(t))$.

Let $(X,d,\mu)$ be a metric probability space. 
We say that  $(X,d,\mu)$ satisfies the \emph{Lebesgue differentiation property} (LDP), if for every $h\in L^1(\mu)$,
\[
        \lim_{r\downarrow 0} \mathcal L_r(h)(x)=h(x), \quad  a.e. \ x\in X.
\]
Fix a nonnegative integer sequence $\mathbf a=(a_n)_{n\ge 1}$ and a positive normalizing sequence $\mathbf D=(D_N)_{N\ge 1}$. For a given
 measure-preserving system $(X,\B_X,\mu,T)$, a  Young function $\Phi$ and an $f\in L^\Phi $, define 
\[
         A_N^{\mathbf a,\mathbf D}f(x)= \frac{1}{D_N}\sum_{n=1}^N f(T^{a_n}x)
\]
for $x\in X$.  We say that the averages $ A_N^{\mathbf a,\mathbf D}f$ converge \emph{locally stably} if there is $f^* \in L^1(\mu)$ such that 
\[ 
\lim_{\substack{r\downarrow  0, N\to \infty}} \mathcal L_r( A_N^{\mathbf a,\mathbf D}f)(x)=f^*(x),  \quad  a.e. \ x\in X.
\]
Here and below, all limits involving $r\downarrow 0$ and $N\to \infty$ are understood as joint limits.
When the underlying metric probability space has the LDP, local stable convergence implies ordinary pointwise convergence to the same limit 
(Lemma~\ref{lem:local-stability-consequences}).

We say that the pair $(\mathbf a,\mathbf D)$ admits an \emph{$L^\Phi$ local $\infty$-sweeping out} counterexample on  $[0,1]$ if there exist an
invertible ergodic Lebesgue measure-preserving transformation $T$ of $[0,1]$, a nonnegative $f\in L^\Phi([0,1])$, and a measurable set
$Y\subset[0,1]$ with $\lambda(Y)>0$, such that
\begin{equation*}\label{eq:local-infty-sweeping out} 
  \limsup_{\substack{r\downarrow 0, N\to \infty}} \mathcal L_r\bigl( A_N^{\mathbf a,\mathbf D}f\bigr)(x) =\infty, \quad \text{  a.e. }x\in Y.
\end{equation*}
In fact, the lower-bound theorems in the present paper are established in the following stronger almost-full-measure form: for every fixed $\Phi$ and $0<\varepsilon<1$, the objects  $T,f,Y$ in the counterexample can be chosen so that $\lambda(Y)>1-\varepsilon$.
By the Lebesgue differentiation theorem, the interval $[0,1]$ equipped with the standard Euclidean metric and Lebesgue measure satisfies the LDP, see \cite[Chapter I, Section 1]{Stein}.  Thus, these counterexamples occur within the same class of metric probability spaces as the positive results.

Let $s\in\N$. We call $L\log_sL$ the \emph{sharp Orlicz endpoint for the local observation problem associated with $(\mathbf a,\mathbf D)$} if the following two
conditions hold:
\begin{enumerate}
\item[(a)]
For every measure-preserving system whose underlying metric probability space has the LDP, and every $f\in L\log_sL$, the averages $ A_N^{\mathbf a,\mathbf D}f$ converge locally stably.

\item[(b)]
For every  Young function $\Phi$ satisfying
\[
        \Phi(t)=o\bigl(tL_s(t)\bigr) \quad (t\to \infty),
\]
the pair $(\mathbf a,\mathbf D)$ admits a local $\infty$-sweeping out counterexample in $L^\Phi$ on the Lebesgue interval.
\end{enumerate}

For $q\in\N_0$, define
\[
        \mathbf D^{(q)}= (N\Lambda_q(N))_{N\ge 1}
\]
The present paper focuses primarily on this family of normalizations.

Our first main result addresses the consecutive and prime time sequences.
 
\begin{customthm}{\bf Theorem A} \label{thm:regular-endpoints}
{\it
\begin{enumerate}[label=\textup{(\roman*)}, widest=(ii), leftmargin=0.2cm, align=left, labelsep=3pt]
\item \label{thm:regular-endpoints:1}
For consecutive times $a_n=n-1$ and every $q\in\N_0$,  $L\log_{q+1}L$ is the sharp Orlicz endpoint for the local observation problem 
associated with $(\mathbf a,\mathbf D^{(q)})$.

The locally stable limit is $\mathbb E(f\mid\I)$ when $q=0$, and $0$ when $q\ge 1$,
where $\I$ denotes the $T$-invariant $\sigma$-algebra.

\item\label{thm:regular-endpoints:2}
For prime times $a_n=p_n$ and every $q\in\N$, $ L\log_{q+1}L$ is the sharp Orlicz endpoint for the local observation problem associated with $(\mathbf a,\mathbf D^{(q)})$, and the locally stable limit is 0.
\end{enumerate}
}
\end{customthm}

\ref{thm:regular-endpoints} rests on two general principles developed in this paper.
The positive direction is given by a general local stability principle of Theorem~\ref{thm:local-stability-principle}, which shows that under the LDP, pointwise convergence becomes locally stable once the temporal maximal function admits an integrable majorant.
To verify this condition we introduce a restricted maximal estimate of level $q$ (RME$_q$), a logarithmic analogue of weak type $(1,1)$ (Definition \ref{def:RME}).
For consecutive averages, RME$_q$ follows  from the classical weak type $(1,1)$ inequality (Proposition \ref{prop:weak-RME}).
For prime averages, it is derived from Trojan's restricted endpoint estimate (Proposition \ref{prop:polylog-RME}). A critical logarithmic lifting argument in Proposition \ref{prop:RME-lift} then yields the required maximal majorant.

The negative direction is governed by the weighted local sweeping-out principle of Theorem~\ref{thm:weighted-sweeping}, which provides a general principle for constructing local $\infty$-sweeping-out counterexamples.  For polynomial-growth regular times, it yields lower bounds matching the positive endpoint theorem
described above.  The two directions together identify the sharp Orlicz endpoints in  \ref{thm:regular-endpoints}.

\medskip
The second main result gives the universal positive theorem and the sharp Orlicz endpoint for particular sparse sequences.

\begin{customthm}{\bf Theorem B}\label{thm:universal-sparse}
{\it 
Let $q\in\N$.
\begin{enumerate}[label=\textup{(\roman*)},leftmargin=*]
\item \label{thm:universal-sparse:1}
For every sequence of nonnegative integers $(a_n)$, every measure-preserving system  $(X,\B_X,\mu,T)$  and every $f\in L\log_{q}L$,
we have
\[
        \frac{1}{N\Lambda_q(N)}\sum_{n=1}^{N} f(T^{a_n}x) \to 0 \quad\text{ a.e. } x\in X.
\]
Moreover, if the underlying metric probability space has the LDP, then this convergence is locally stable.

\item \label{thm:universal-sparse:2}
For denominator $N\Lambda_q(N)$, $ L\log_{q}L$ is the sharp Orlicz endpoint for the local observation problem, for each of the following families of time sequences:
\begin{enumerate}[label=\textup{(\alph*)},leftmargin=*]
\item \label{thm:universal-sparse:2:1} sequences satisfying
\[
        \frac{a_{n+1}}{a_n}\ge n^\eta
\]
for all sufficiently large $n$ and for some $\eta>0$, such as
$a_n=n!$;

\item \label{thm:universal-sparse:2:2} fixed-base exponential sequences $a_n=k^n$ for $k\ge 2$.
\end{enumerate}
\end{enumerate}
}
\end{customthm}

Via the local stability principle, once the relevant almost-everywhere temporal convergence is known, the spatial step in the positive parts of
~\ref{thm:regular-endpoints} and \ref{thm:universal-sparse} is reduced to proving $L^1$-integrability of the corresponding temporal maximal function.  The
local stability principle itself introduces no logarithmic loss.  The one-logarithm difference between the two endpoint scales enters at the stage
where this maximal integrability is established.

For the regular-time results, restricted logarithmic maximal estimates are lifted from indicator functions to general functions, and this lifting produces the
$L\log_{q+1}L$ scale (Proposition \ref{prop:RME-lift}).  For arbitrary time sequences, the normalization
already supplies the required logarithmic summability.  A weighted dyadic truncation argument gives the maximal $L^1$ bound directly (Lemma \ref{lem:positive-weighted-dyadic}), without passing through the restricted lifting step, and leads instead to the $L\log_qL$ scale (Theorem \ref{thm:universal-positive}).

The lower-bound theory for~\ref{thm:universal-sparse} has a different structure.  
The weighted local sweeping-out principle provides a common framework for passing from suitable finite-stage configurations to local $\infty$-sweeping out (Theorem \ref{thm:finite-stage-lower}), but it does not by itself produce the required finite configurations. 
Therefore,  a key step is to formulate a finite-stage hypothesis which separates the local dynamical realization from the combinatorics of the time sequence (Definition \ref{def:finite-stage-hypotheses}).  Polynomially ratio-separated sequences satisfy this hypothesis through residue constructions (Proposition \ref{prop:poly-ratio-stages}), whereas fixed-base exponential sequences require digit constructions adapted to the underlying base (Proposition \ref{prop:exponential-stages}).  Once these
sequence-specific finite models have been obtained, they are embedded into the weighted local sweeping-out framework.


\medskip
The paper is organized as follows. Section~\ref{sec:preliminaries} collects the preliminary materials, including the interval-realization tools and the local stability principle. Section~\ref{sec:Llogqp1} proves the sharp $L\log_{q+1}L$ Orlicz endpoint results for consecutive and prime averages, together with the polynomial-growth lower bounds. Section~\ref{sec:Llogq} proves the universal $L\log_qL$ Orlicz endpoint theorem for arbitrary time sequences and the matching lower results for polynomially ratio-separated and fixed-base exponential sequences.

\subsection*{Acknowledgments} 
The research of Jie Li is supported by NNSF of China (Grant No. 12031019).

\section{Notation and preliminary facts}\label{sec:preliminaries}

Throughout this paper,  $\N$, $\N_0$, $\Q$ and $\mathbb{R}$ represent the sets of  positive integers, nonnegative integers, rational numbers and real numbers, respectively. We use $\log$  to denote the natural logarithm. 

For two nonnegative quantities $A$ and $B$, the notation $A\lesssim B$ means that $A\le CB$ for a positive absolute
constant $C$.  If the implicit constant may depend on a fixed parameter $q$, we write $A\lesssim_q B$.  If both
$A\lesssim_q B$ and $B\lesssim_q A$ hold, we denote $A\asymp_q B$.

\subsection{Measure-theoretic conventions}\label{subsect:lebesgue}

In this paper, all equalities between measurable sets, all invariance statements, and all equalities between measurable maps are understood modulo null sets. 
 
Let $(X,\B_X,\mu)$ and $(Y,\B_Y,\nu)$ be probability spaces.  A \emph{measure isomorphism} from $(X,\B_X,\mu)$ to
$(Y,\B_Y,\nu)$ is a bi-measurable bijection
\[
        \phi:X_0 \to  Y_0
\]
between full-measure measurable sets $X_0\subset X$ and $Y_0\subset Y$, equipped with the restricted $\sigma$-algebras, such that
\[
        \phi_*(\mu|_{X_0})=\nu|_{Y_0}.
\]
If such a map exists, the two probability spaces are said to be \emph{measure-isomorphic}.

A measurable set $A$ with $\mu(A)>0$ is an \emph{atom} if every measurable $B\subset A$ has measure $0$ or $\mu(A)$.  The measure $\mu$ is
\emph{nonatomic} if it has no atoms.  A \emph{Lebesgue probability space} is a complete probability space which is measure-isomorphic to the disjoint union
of an interval with Lebesgue measure and at most countably many atoms. We shall use the following standard classification of nonatomic Lebesgue
probability spaces: every nonatomic Lebesgue probability space is measure-isomorphic to $([0,1],\B^\lambda_{[0,1]},\lambda)$, 
where $\B^\lambda_{[0,1]}$ denotes the Lebesgue completion of the Borel $\sigma$-algebra, see \cite[Theorem~17.41]{Kechris1995}.
Consequently, if $E\subset[0,1]$ is a Borel subset and $\lambda(E)>0$,
then $\left(E,\B^\lambda_{[0,1]}|_E,  \frac{\lambda|_E}{\lambda(E)}\right)$ is a nonatomic Lebesgue probability space, where $\B^\lambda_{[0,1]}|_E=\{B\cap E:B\in\B^\lambda_{[0,1]}\}$,  and hence is  measure-isomorphic to $ ([0,1],\B^\lambda_{[0,1]},\lambda)$.

Recall that a \emph{measure-preserving system} is a quadruple $(X,\B_X,\mu,T)$, where $(X,\B_X,\mu)$ is a probability
space and $T:X\to X$ is measurable with $T_*\mu=\mu$.  We say that the system is \emph{ergodic} if every $A\in\B_X$ satisfying
$T^{-1}A=A$ has $\mu(A)\in\{0,1\}$.  Two measure-preserving systems $(X,\B_X,\mu,T)$ and $(Y,\B_Y,\nu,S)$ are
\emph{isomorphic} if there are invariant full-measure sets $X_0\subset X$ and $Y_0\subset Y$, and a measure isomorphism $\phi:X_0\to  Y_0$ such that
\[
        \phi\circ T=S\circ \phi \quad\text{on } X_0.
\]
Finally, the system $(X,\B_X,\mu,T)$ is \emph{invertible} if $T$ is a measure-preserving automorphism of $(X,\B_X,\mu)$, 
and in this case we write $T^{-1}$ for its inverse.

\subsection{Orlicz classes and rearrangements}
The backgrounds for Orlicz spaces, Luxemburg norms and rearrangements are classical, see \cite{BennettSharpley} for more details.

By a \emph{Young function} we mean a function
\[
        \Phi(t)=\int_0^t\phi(u)\,\dd u, \quad t\ge 0,
\]
where $\phi:[0,\infty)\to[0,\infty)$ is finite-valued, increasing and
left-continuous, and $\Phi(t)\to \infty$ as $t\to \infty$.  Thus
$\Phi$ is finite-valued, convex and increasing, with $\Phi(0)=0$.

For a metric probability  space $(X,\B_X,\mu)$ and any measurable function $f$, we use the Luxemburg norm 
\[
   \|f\|_{\Phi}=\inf\left\{\lambda>0:\int_X \Phi(|f|/\lambda)\dd \mu\le 1\right\}.
\]
The Orlicz space $L^\Phi$ consists of all measurable $f$  such that 
\[
    \int_X \Phi(|f|/\lambda)\dd \mu<\infty
\]
for some $\lambda>0$. 

For $s\in\N$, recall the regularized iterated logarithms
\[
        L_1(t)=\log(\mathrm{e}+t),\quad L_{s+1}(t)=\log(\mathrm{e}+L_s(t)).
\]
Define the logarithmic Young function
\[
        \Psi_s(t)=\int_0^t \bigl(1+L_s(u)\bigr)\dd u.
\]
Since $L_s$ is finite and increasing, $\Psi_s$ is a Young function. We use the notation
\[
      L\log_sL=L^{\Psi_s}.
\]

The simpler expression $t(1+L_s(t))$ will be used only as an equivalent growth function.
Indeed, by monotonicity of $L_s$ and the fact that $1+L_s(t/2)\asymp_s1+L_s(t)$, one has
\begin{equation}\label{eq:Psi-comparison}
        c_s t\bigl(1+L_s(t)\bigr)\le \Psi_s(t)   \le t\bigl(1+L_s(t)\bigr),   \quad t\ge 0.
\end{equation}
for some $c_s>0$.  Thus $ L\log_1L$ is the usual Zygmund space $L\log L$, see \cite[Chapter~4, Sections~6 and~8]{BennettSharpley}.

\begin{lem} \label{lem:Psi-estimates} 
For every fixed $s\in\N$, the following hold. 
\begin{enumerate}[label=\textup{(\alph*)}] 
\item\label{lem:Psi-estimates:1}  For every $A\ge 1$ there is $C_{A,s}<\infty$ such that
 \[ 
 \Psi_s(At)\le C_{A,s}\Psi_s(t), \quad t\ge 0. 
 \] 
\item\label{lem:Psi-estimates:2}
 On a metric probability space $(X,\B_X,\mu)$, 
 \[ 
 f\in L^{\Psi_s} \ \Longleftrightarrow\ \int_X |f|\bigl(1+L_s(|f|)\bigr)\dd \mu<\infty. 
 \]
  
 \item\label{lem:Psi-estimates:3}
 For every fixed $A,\alpha>0$, there is $C_{A,\alpha,s}<\infty$ such that 
 \[ 
 L_s(A(1+t)^\alpha)\le C_{A,\alpha,s}\bigl(1+L_s(\mathrm{e}+t)\bigr), \quad t\ge 0. 
 \] 
 In particular, there is $C_s>0$ such that 
 \[ 
 L_s(t^2)\le C_s\bigl(1+L_s(\mathrm{e}+t)\bigr). 
 \] 
\end{enumerate} 
\end{lem}
                     
\begin{proof}
By \eqref{eq:Psi-comparison},
\[
        \Psi_s(At)\le At\bigl(1+L_s(At)\bigr)\le C_{A,s}t\bigl(1+L_s(t)\bigr)
        \le C'_{A,s}\Psi_s(t),
\]
where the middle estimate follows by induction from the regularized logarithms.
This proves (a).  

(b) follows immediately from (a) and \eqref{eq:Psi-comparison}: if $\int_X\Psi_s(|f|/\lambda)\dd \mu<\infty$ for some $\lambda>0$, then scaling back
to $\lambda=1$ is allowed by monotonicity when $\lambda\le 1$ and by (a) when $\lambda>1$.  
The converse is immediate from the upper bound in \eqref{eq:Psi-comparison}. 

 Finally, (c) is first the elementary estimate for
$L_1$,
\[
        \log(\mathrm{e}+A(1+t)^\alpha)\le C_{A,\alpha}(1+\log(\mathrm{e}+t)),
\]
and then induction over the iterates. 
The final estimate follows by taking $A=1$ and $\alpha=2$.
\end{proof}

Let $(X,\B_X,\mu)$ be a probability space and let $f\ge 0$ be measurable and real-valued.  Write
\[
        d_f(t)=\mu\{x:f(x)>t\}, \quad t\ge 0,
\]
and define the decreasing rearrangement by
\[
        f^*(u)=\inf\{t>0:d_f(t)\le u\}, \quad 0<u<1.
\]
The generalized-inverse relation gives, for every $t>0$ and almost every
$u\in(0,1)$,
\[
        u<d_f(t) \quad\Longleftrightarrow\quad  t<f^*(u).
\]
Consequently,
\begin{equation}\label{eq:rearrangement-section}
        \int_0^\infty\mathbf1_{\{u<d_f(t)\}}\,\dd t=f^*(u) \quad\text{for a.e. }u\in(0,1),
\end{equation}
and $f$ and $f^*$ are equimeasurable:
\[
        d_f(t)  =  |\{u\in(0,1):f^*(u)>t\} |.
\]
Thus, for every nonnegative Borel function $\Theta$,
\begin{equation}\label{eq:equimeasurability}
        \int_X\Theta(f)\,\dd\mu = \int_0^1\Theta(f^*(u))\,\dd u.
\end{equation}
In particular, the layer-cake formula yields
\[
        \int_X f\,\dd\mu  = \int_0^\infty d_f(t)\,\dd t=\int_0^1f^*(u)\,\dd u.
\]
These standard formulas of rearrangement-invariant spaces can be found in \cite[Chapter~2, Sections~1--2]{BennettSharpley}.

\begin{lem}\label{lem:dist-bound}
For every $s\in\N$, there is $C_s>0$ such that every
nonnegative real-valued measurable $f$ on a probability space satisfies
\[
        \int_0^\infty d_f(t)\left(1+L_s\left(\frac{1}{d_f(t)}\right)\right) \dd t \le C_s\left( 1+\int_Xf\bigl(1+L_s(\mathrm e+f)\bigr)\,\dd\mu\right),
\]
where the integrand on the left is interpreted as zero when $d_f(t)=0$.
\end{lem}

\begin{proof}
The part without the logarithm is $\int_Xf\,\dd\mu$.  For the remaining part,  by Tonelli's theorem  and   \eqref{eq:rearrangement-section}  we have 
\[
\begin{aligned}
        \int_0^\infty d_f(t)L_s(1/d_f(t)) \dd t
        &=\int_0^\infty\int_0^1 \mathbf1_{\{u<d_f(t)\}}L_s(1/d_f(t)) \dd u \dd t\\
        &\le \int_0^\infty\int_0^1 \mathbf1_{\{u<d_f(t)\}}L_s(1/u) \dd u \dd t\\
        &=\int_0^1f^*(u)L_s(1/u) \dd u.
\end{aligned}
\]
Split $(0,1)$ into
\[
        E_1=\{u:f^*(u)\ge u^{-1/2}\} \   \text{and }\   E_2=(0,1)\setminus E_1.
\]
On $E_1$,  we have $1/u\le (f^*(u))^2$, and then by Lemma~\ref{lem:Psi-estimates}(c),  
\[
        L_s(1/u) \le C_s\bigl(1+L_s(\mathrm e+f^*(u))\bigr).
\]
Since $f$ and $f^*$ are equimeasurable,
\[
\begin{aligned}
        \int_{E_1}f^*(u)L_s(1/u)\,\dd u
        &\le C_s\int_0^1f^*(u) \bigl(1+L_s(\mathrm e+f^*(u))\bigr)\,\dd u\\
        &=  C_s\int_Xf\bigl(1+L_s(\mathrm e+f)\bigr)\,\dd\mu.
\end{aligned}
\]
On $E_2$, it is clear that
\[
        f^*(u)L_s(1/u)\le u^{-1/2}L_s(1/u),
\]
and the right-hand side is integrable on $(0,1)$.  
Combining the estimates on $E_1$ and $E_2$, the desired upper bound follows.
\end{proof}

\begin{lem} \label{lem:lux-tail}
If $f\in L^{\Psi_s}$ and $f_K=f\1_{\{|f|\le K\}}$, then
\[
        \|f-f_K\|_{\Psi_s}\to 0 \quad \text{as } ~K\to \infty.
\]
\end{lem}

\begin{proof}
Choose $\lambda_0>0$ such that
\[
        \int_X\Psi_s(|f|/\lambda_0)\,\dd\mu<\infty.
\]
By Lemma~\ref{lem:Psi-estimates}\ref{lem:Psi-estimates:1}, for every $A\ge 1$ there is a constant $C_{A,s}\ge 1$ such that
\[
        \Psi_s(At)\le C_{A,s}\Psi_s(t),   \quad t\ge 0.
\]
Fix $\varepsilon>0$.  If $\varepsilon<\lambda_0$, apply this estimate with
$A=\lambda_0/\varepsilon$. If $\varepsilon\ge\lambda_0$, use the
monotonicity of $\Psi_s$.  In either case, we have 
\[
        \Psi_s(|f|/\varepsilon)\le C_{\lambda_0,\varepsilon,s}\Psi_s(|f|/\lambda_0) \in L^1(\mu).
\]
Observe that
\[
        |f-f_K|=|f|\mathbf1_{\{|f|>K\}}\to 0   \quad\text{a.e.},
\]
and since $\Psi_s$ is increasing, 
\[
\Psi_s(|f-f_K|/\varepsilon) \le \Psi_s(|f|/\varepsilon) \in L^1,
\]
then by the dominated convergence theorem, 
\[
\int_X \Psi_s(|f-f_K|/\varepsilon)\,\dd\mu\to 0 \quad (K\to \infty).
\]
For all large $K$ this integral is at most $1$, hence $\|f-f_K\|_{\Psi_s}\le\varepsilon$.
Since $\varepsilon$ is arbitrary, the lemma follows.
\end{proof}

\subsection{Logarithmic calculus}

Recall the iterated-logarithm denominators defined by
\[
\Lambda_0 = 1, \quad  \Lambda_q(t) = \prod_{j=1}^q L_j(t), \quad q \in \mathbb{N}.
\]
We shall repeatedly use the following elementary estimates.

\begin{lem}\label{lem:log-calc}
Let $q\in\N_0$.
\begin{enumerate}[label=\textup{(\roman*)},leftmargin=*]
\item\label{lem:log-calc:1}
There are constants $c_q,C_q>0$ such that for every real $R\ge 1$,
\[
        c_q\bigl(1+L_{q+1}(R)\bigr)\le 1+\int_1^R\frac{\dd u}{u\Lambda_q(u)}\le C_q\bigl(1+L_{q+1}(R)\bigr).
\]

\item\label{lem:log-calc:2}
There are constants $c_q,C_q>0$ such that for every integer $R\ge 1$,
\[
        c_q\bigl(1+L_{q+1}(R)\bigr)\le1+\sum_{n=1}^R\frac{1}{n\Lambda_q(n)}\le C_q\bigl(1+L_{q+1}(R)\bigr).
\]
In particular,
\[
        \sum_{n=1}^{\infty}\frac{1}{n\Lambda_q(n)} =\infty.
\]

\item\label{lem:log-calc:3}
For every $j\ge 1$, there is a constant $C_j>0$ such that
\[
        L_j(AB)\le  C_j\bigl(1+L_j(A)+L_j(B)\bigr), \quad A,B\ge 0.
\]
\end{enumerate}
\end{lem}

\begin{proof}
 
An induction on $q$ gives
\begin{equation}\label{eq:iterated-log-derivative}
        L_{q+1}'(u) \asymp_q \frac{1}{u\Lambda_q(u)}, \quad u\ge 1.
\end{equation}
Indeed, the case $q=0$ is immediate, and the induction step follows from
\[
        L_{q+1}'(u) = \frac{L_q'(u)}{\mathrm e+L_q(u)}
\]
and $\mathrm e+L_q(u)\asymp_q L_q(u)$ on $[1,\infty)$.
Integrating \eqref{eq:iterated-log-derivative} over $[1,R]$, and then absorbing the fixed value $L_{q+1}(3)$ by adding $1$, we get \ref{lem:log-calc:1} for all real $R\ge 1$.

The function $u\mapsto(u\Lambda_q(u))^{-1}$ is decreasing.  Then for every integer $R\ge 1$,
\[
        \int_1^{R+1}\frac{\dd u}{u\Lambda_q(u)}\le  \sum_{n=1}^R\frac{1}{n\Lambda_q(n)}\le\frac{1}{\Lambda_q(1)}+\int_1^R\frac{\dd u}{u\Lambda_q(u)}.
\]
Hence \ref{lem:log-calc:2} follows from \ref{lem:log-calc:1} and $1+L_{q+1}(R+1)\asymp_q1+L_{q+1}(R)$.  The divergence is immediate.

Finally, \ref{lem:log-calc:3} follows from
\[
        \log(\mathrm e+AB)\le C\bigl(1+\log(\mathrm e+A)+\log(\mathrm e+B)\bigr)
\]
and induction over the iterates.
\end{proof}

\subsection{Interval partitions and tower embeddings} \label{subsec:sweeping-tools}

This subsection presents three elementary tools for the lower-bound constructions.  Lemma~\ref{lem:fine-filling} gives fine interval partitions with prescribed measure bounds.  It is used directly in the polynomial-growth case and also serves as the main ingredient in the finite allocation Lemma~\ref{lem:finite-allocation} needed in the sparse case.
Finally, Lemma~\ref{lem:tower-embedding} realizes, within a single invertible ergodic Lebesgue measure-preserving transformation of $[0,1]$, the tower
data required by the polynomial-growth and sparse constructions.

\begin{lem}\label{lem:fine-filling}
Let $R\subset[0,1]$ be a finite union of half-open intervals,  $n\ge 2$ and  $0<\ell_i\le u_i\in\R$ for $1\le i\le n$ with  
    \[    
        \sum_{i=1}^n\ell_i\le\lambda(R)\le\sum_{i=1}^n u_i.
\]
For every $\eta>0$, there exists a partition $R=\bigsqcup_{i=1}^n H_i$ such that each   $H_i$ is a finite union of half-open intervals, 
\[
        \ell_i\le\lambda(H_i)\le u_i,
\]
and every  component interval of each $H_i$ has length at most $\eta$.
\end{lem}

\begin{proof}
Set $L=\lambda(R)$. Choose $v_i\in[\ell_i,u_i]$ such that $\sum_{i=1}^n v_i=L$.
Write $R$ as the disjoint union of its maximal half-open component
intervals,
\[
    R=\bigsqcup_{k=1}^K[a_k,b_k),\quad a_1<b_1<a_2<\cdots<a_K<b_K.
\]
Set $r_k=b_k-a_k$, $ s_0=0$ and $s_k=\sum_{m=1}^k r_m$. Then  $s_K=L$. Define $\Phi:[0,L)\to R$ by
\[
    \Phi(t)=a_k+t-s_{k-1}, \quad \forall \ t\in[s_{k-1},s_k).
\]
It is clear that $\Phi$ is a measure-preserving piecewise translation.

For fixed $\eta>0$, choose $M\in\N$ so large that
\[
    \max_{1\le i\le n}\frac{v_i}{M}<\eta.
\]
For $0\le m<M$ and $1\le i\le n$, set
\[
    I_{m,i} =\left[ \frac{mL+\sum_{j<i}v_j}{M}, \frac{mL+\sum_{j\le i}v_j}{M} \right).
\]
Since $\sum_i v_i=L$, the intervals $I_{m,i}$ partition
$[0,L)$. Put
\[
    G_i=\bigsqcup_{m=0}^{M-1}I_{m,i}, \quad \text{ and } \quad H_i=\Phi(G_i).
\]
Then the sets $G_i$ partition $R$, and
\[
    \lambda(H_i)=\lambda(G_i) =\sum_{m=0}^{M-1}\lambda(I_{m,i}) =v_i \in[\ell_i,u_i].
\]

It remains to check the lengths of the component intervals. Since $n\ge 2$ and $v_j>0$ for every $j$, distinct intervals
$I_{m,i}$ with the same label $i$ are separated by an interval of positive length. For each $k$, the restriction of $\Phi$ to
$[s_{k-1},s_k)$ is a translation. Thus every component interval of $ G_i\cap[s_{k-1},s_k)$
is contained in some $I_{m,i}$, and its image under $\Phi$ has length at most $\frac{v_i}{M}<\eta$.
For fixed $k$, translation preserves the gaps between distinct intervals. For different $k$, the images lie in different maximal
component intervals of $R$. Hence every component interval of $H_i$ has length less than $\eta$.
\end{proof}

\begin{lem}\label{lem:finite-allocation}
Let $C_1,\ldots,C_m$ be pairwise disjoint finite unions of half-open
intervals, and let $\gamma_1,\ldots,\gamma_\ell\ge 0$. If
\[
    \sum_{i=1}^m\lambda(C_i)
    \le
    \sum_{a=1}^{\ell}\gamma_a,
\]
then there are finite unions of half-open intervals $C_{i,a}$ such that
\[
    C_i=\bigsqcup_{a=1}^{\ell}C_{i,a}, \quad 1\le i\le m,
\]
and
\[
    \sum_{i=1}^m\lambda(C_{i,a})\le\gamma_a, \quad 1\le a\le\ell.
\]
\end{lem}

\begin{proof}
Let $C = \bigsqcup_{i=1}^m C_i$ and $L = \lambda(C)$. In the same spirit as Lemma \ref{lem:fine-filling}, 
concatenating the maximal half-open component intervals of $C$ we obtain a measure-preserving piecewise translation $\Theta:[0,L)\to  C.$

Set $\Gamma_0 = 0$ and $\Gamma_a = \sum_{b=1}^a \gamma_b$ for  $1 \le a \le \ell$.
By assumption $\Gamma_\ell \ge L$. For each  $1 \le a \le \ell$, define  
\[
    J_a=[\Gamma_{a-1},\Gamma_a)\cap[0,L).
\]
Then the intervals $J_a$  partition $[0, L)$, and $\lambda(J_a) \le \gamma_a$ for all $1 \le a \le \ell$. 
Define
\[
    C_{i,a}=C_i\cap\Theta(J_a).
\]
Then for each  $1 \le a \le \ell$ and each $1\le i\le m$, $C_{i,a}$ are finite unions of half-open intervals, 
\[
    C_i=\bigsqcup_{a=1}^{\ell}C_{i,a},
\]
and 
\[
    \sum_{i=1}^m\lambda(C_{i,a})=\lambda (\Theta(J_a) ) =\lambda(J_a) \le\gamma_a.
\]
\end{proof}

\begin{lem}\label{lem:tower-embedding}
Let $\ell_s\in\N_0$ and $\delta_s>0$, $s\ge 1$, and let $\{E_{s,t}:s\ge 1,\ 0\le t\le \ell_s\}$
be pairwise disjoint Borel subsets of $[0,1]$ satisfying $\lambda(E_{s,t})=\delta_s$.
Assume that $ q_0=\sum_{s=1}^\infty(\ell_s+1)\delta_s<1.$
For $0\le t<\ell_s$, let
\[
        \phi_{s,t}:E_{s,t}\to E_{s,t+1}
\]
be a measure-preserving isomorphism.  Then there exists an invertible ergodic Lebesgue measure-preserving transformation $T$ of
$[0,1]$ such that
\[
        T=\phi_{s,t}  \quad\text{a.e. on }E_{s,t}
\]
for every $s\ge 1$ and $0\le t<\ell_s$.
\end{lem}

\begin{proof}
Let $(Z,\B_Z,\nu,R)$ be the irrational rotation on the circle with Lebesgue measure. It is know that $R$ is invertible and ergodic. 
Since  $q_0=\sum_{s=1}^{\infty}(\ell_s+1)\delta_s$, we have
$D=1-q_0+\sum_{s=1}^{\infty}\delta_s>0$.
Choose a Borel partition
\[
        Z=B_0\sqcup B_1\sqcup B_2\sqcup\cdots
\]
such that
\[
        \nu(B_0)=\frac{1-q_0}{D},\quad \text{ and }\quad \nu(B_s)=\frac{\delta_s}{D}\quad (s\ge 1).
\]
For example, the sets $B_s$ may be chosen as consecutive half-open
intervals on the circle.  Define $h:Z\to\N$ by
\[
        h|_{B_0}=1,\quad \text{ and }\quad  h|_{B_s}=\ell_s+1\quad (s\ge  1).
\]
We have 
\[
        \int_Z h\dd \nu= \frac{1-q_0}{D} + \sum_{s=1}^{\infty}(\ell_s+1)\frac{\delta_s}{D}= \frac{1}D.
\]

Consider the suspension
\[
        Z^h=\{(z,t):z\in Z,\ 0\le t<h(z)\},
\]
equipped with the normalized suspension measure
\[
        \nu^h=\frac{1}{\int_Z h\dd\nu} \bigl((\nu\times \#)\!|_{Z^h}\bigr),
\]
where $\#$ denotes counting measure on $\N_0$.  Define the
discrete suspension transformation
\[
        \widetilde R(z,t)=
        \begin{cases}
        (z,t+1), & t+1<h(z),\\
        (R \, z,0), & t+1=h(z).
        \end{cases}
\]
This is the standard discrete suspension construction, see \cite[Section~2.9]{EinsiedlerWard} for instance.  
Clearly, the map $\widetilde R$ is invertible and preserves $\nu^h$.  
Moreover, $\widetilde R$ is ergodic.  Indeed, let $G\in L^2(Z^h,\nu^h)$ satisfy $G\circ \widetilde R=G$.  Then  for every $i\ge 0$,
\[
        G(z,i+1)=G(z,i) \quad\text{for a.e. } z\in\{z\in Z: h(z)>i+1\}.
\]
Since there are only countably many levels, it follows that for a.e. $z$,
\[
        G(z,0)=G(z,1)=\cdots=G(z,h(z)-1).
\]
On the top level, $G\circ \widetilde R=G$ implies
\[
        G(R\, z,0)=G(z,h(z)-1)  
\]
for a.e. $z$. Hence
\[
        G(R\, z,0)=G(z,0) 
\]
for a.e. $z$.
Thus $g(z)=G(z,0)$ is $R$-invariant.  Since $R$ is ergodic, $g$ is constant a.e. Then the column relation above 
implies that $G$ is constant a.e. on $Z^h$. Hence $\widetilde R$ is ergodic.

By the classification theorem for non-atomic Lebesgue spaces as stated in Subsection \ref{subsect:lebesgue}, any two non-atomic Lebesgue spaces are measure-isomorphic.  Hence, for each $s\ge 1$, we may choose a measure-preserving isomorphism 
\[
        \Theta_{s,0}:B_s\times\{0\}\to E_{s,0}.
\]
For $1\le t\le\ell_s$, define
\[
        \Theta_{s,t}(z,t)
        =
        \phi_{s,t-1}\circ\phi_{s,t-2}\circ\cdots\circ\phi_{s,0}
        \bigl(\Theta_{s,0}(z,0)\bigr).
\]
Also choose a measure-preserving isomorphism 
\[
        \Theta_{0,0}:B_0\times\{0\}\to [0,1]\setminus\bigsqcup_{s,t}E_{s,t}.
\]
The maps $\Theta_{0,0}$ and $\Theta_{s,t}$ together define a global  measure-preserving isomorphism   $\Theta:Z^h\to[0,1]$.
Now define $ T =\Theta\circ\widetilde R\circ\Theta^{-1}$. 
Consequently, $T$ is an invertible, ergodic, Lebesgue measure-preserving transformation of $[0,1]$.
If $y\in E_{s,t}$ with $0\le t<\ell_s$, we can write $y=\Theta_{s,t}(z,t)$ for some $(z,t)\in Z$. Then
\[
\begin{aligned}
        Ty=\Theta\widetilde R(z,t)=\Theta(z,t+1)=
        \Theta_{s,t+1}(z,t+1)= \phi_{s,t}\bigl(\Theta_{s,t}(z,t)\bigr)=\phi_{s,t}(y),
\end{aligned}
\]
for a.e. $y\in E_{s,t}$.  This finishes the proof.
\end{proof}

\subsection{Local stability principle}\label{subsec:local-stable-principle}
This subsection establishes the general local stability principle used in the positive results throughout the paper.  
It shows that, on metric probability spaces with the LDP, almost-everywhere convergence together with an integrable maximal bound
implies locally stable convergence.

\begin{defn}
We say that a metric probability space $(X,d,\mu)$ has the \emph{Lebesgue differentiation property} (LDP), if for every $h\in L^1(\mu)$,
\begin{equation}\label{LDP}
  \lim_{r\downarrow0}\mathcal L_r(h)(x)=\lim_{r\downarrow0}\frac{1}{\mu(B(x,r))}\int_{B(x,r)}h(y)\dd\mu(y)=h(x) \quad  a.e. \ x\in X.
\end{equation}
A point at which  \eqref{LDP} holds  is referred as a \emph{Lebesgue point} of $h$.
\end{defn}
 
\begin{rem}
Euclidean spaces with Lebesgue measure have the LDP for balls \cite[Chapter~I, Section~1]{Stein}.  More generally, Borel regular
doubling metric measure spaces have the LDP for balls, see \cite[Section~3.4]{HKST}.  The property is not automatic for metric Borel
probability spaces.  Preiss, Riss and Ti{\v s}er \cite{PreissRissTiser2021} constructed a Gaussian measure on a separable
Hilbert space and a Borel set of positive measure whose ball densities are uniformly zero everywhere. In particular the LDP fails in this setting.
\end{rem}

We extend the notion of local stability to arbitrary sequences of integrable functions.

\begin{defn}\label{def:local-stability}
Let $(X,d,\mu)$ be a metric probability space, and let $F_N,F\in L^1(\mu)$ with $F_N\to F$ a.e. We say that the
convergence is \emph{locally stable} if
\[
    \lim_{\substack{r\downarrow 0,\ N\to \infty}}
    \mathcal L_rF_N(x)=F(x)
\]
for a.e. $x\in X$. Equivalently, for a.e. $x$ and every $\varepsilon>0$, there exist $N_0\in\N$ and $r_0>0$ such that
\[
    |\mathcal L_rF_N(x)-F(x)|<\varepsilon
\]
whenever $N\ge N_0$ and $0<r<r_0$.
\end{defn}

\begin{lem}\label{lem:local-stability-consequences}
Let $(X,d,\mu)$ have the LDP, and let $F_N,F\in L^1(\mu)$.
\begin{enumerate}
\item\label{lem:local-stability-consequences:1}
If
\[
        \lim_{m\to \infty} \limsup_{r\downarrow 0}\sup_{N\ge m}\mathcal L_r(|F_N-F|)(x)=0
\]
for  a.e. $x$, then $F_N$ converges locally stably to $F$;

\item\label{lem:local-stability-consequences:2}
If $F_N$ converges locally stably to $F$, then
\[
        F_N(x)\to  F(x)
\]
for a.e. $x$.
\end{enumerate}

Consequently, the hypothesis in \eqref{lem:local-stability-consequences:1} implies both local stable convergence and ordinary pointwise convergence.
\end{lem}

\begin{proof}
\eqref{lem:local-stability-consequences:1}: Fix $x$ in a full-measure set on which
the hypothesis holds and the Lebesgue differentiation conclusion holds for $F$.  For $m\ge 1$, set
\[
        A_m(x)=\limsup_{r\downarrow 0}\sup_{N\ge m} \mathcal L_r(|F_N-F|)(x).
\]
By assumption $ A_m(x)\to 0$ as $m\to \infty$.

Let $\varepsilon>0$.  Choose $m\in\N$ such that
\[
        A_m(x)<\frac{\varepsilon}{4}.
\]
Then there exists $r_1>0$ such that
\[
        \sup_{0<r<r_1}\sup_{N\ge m}\mathcal L_r(|F_N-F|)(x)< \frac{\varepsilon}{2}.
\]
Since the Lebesgue differentiation conclusion holds for $F$, there exists $r_2>0$ such that
\[
        |\mathcal L_r(F)(x)-F(x)| < \frac{\varepsilon}{2}
\]
whenever $0<r<r_2$.  Therefore, for every $N\ge m$ and every $0<r<\min\{r_1,r_2\}$,
\[
        |\mathcal L_r(F_N)(x)-F(x)|\le\mathcal L_r(|F_N-F|)(x) +|\mathcal L_r(F)(x)-F(x)|  <\varepsilon.
\]
Then the locally stable convergence follows.

\eqref{lem:local-stability-consequences:2}:  Fix $x$ in a full-measure set on which the locally stable convergence holds and the Lebesgue differentiation conclusion holds simultaneously for
all $F_N$ ($N\in\N$).  
Let $\varepsilon>0$.  By locally stable convergence, there exist $m\in\N$ and $\delta>0$ such that
\[
        |\mathcal L_r(F_N)(x)-F(x)| <\varepsilon
\]
whenever $N\ge m$ and $0<r<\delta$.  Fix $N\ge m$.  Letting $r\downarrow 0$ and applying the LDP to $F_N$, we obtain
\[
        |F_N(x)-F(x)|\le\varepsilon.
\]
Thus
\[
        \sup_{N\ge m}|F_N(x)-F(x)|\le\varepsilon.
\]
Since $\varepsilon>0$ was arbitrary, it follows that $F_N(x)\to F(x)$.
\end{proof}
 
The following principle reduces local stability to pointwise convergence together with an integrable maximal majorant, yielding a general positive-direction criterion for  sharp endpoint problems.

\begin{thm}\label{thm:local-stability-principle}
Assume that $(X,d,\mu)$ has the LDP.  Let $(F_N)$ be measurable functions such that
\[
        F_N(x)\to F(x)\quad\textup{for a.e. }x, \quad F\in L^1(\mu),
\]
and assume that there is $H\in L^1(\mu)$ such that 
\[
        \sup_{N\ge 1}|F_N|\le H.
\]
Then
\[
        \lim_{m\to \infty}\limsup_{r\downarrow 0}  \sup_{N\ge m}\mathcal L_r(|F_N-F|)(x)=0
\]
for a.e. $x\in X$.  In particular, $F_N\to F$ is locally stable.
\end{thm}
    
\begin{proof}
Replacing $H$ by $H+|F|$, we may assume that for all $N\ge 1$, 
\[
        |F_N-F|\le H.
\]
Let $X'\subset X$ be a full-measure set on which $F_N\to F$.  By Egorov's theorem, for each $j\ge 1$ there is a measurable set $K_j\subset X'$ such that $\mu(K_j^c)<2^{-j}$
and $F_N\to F$ uniformly on $K_j$.  Since $\sum_{j=1}^{\infty}\mu(K_j^c)<\infty,$
by the Borel--Cantelli lemma we have
\[
        \mu\left(\liminf_{j\to \infty}K_j\right)= \mu\left(\bigcup_{J=1}^{\infty}\bigcap_{j\ge J}K_j\right)=1.
\]

For each $j\ge 1$, clearly $H\mathbf 1_{K_j^c}\in L^1(\mu)$. By the LDP, after discarding one further null set, we may assume that every
point $x$ under consideration is a Lebesgue point of every
$H\mathbf 1_{K_j^c}$.  Thus let $X_0$ be the full-measure set of all
$x\in\liminf_j K_j$ such that
\[
        \lim_{r\downarrow 0}\mathcal L_r(H\mathbf 1_{K_j^c})(x) = H(x)\mathbf 1_{K_j^c}(x)
\]
for every $j\ge 1$.

Now take $x\in X_0$. Let $\eta>0$. since $x\in\liminf_jK_j$ we may choose large $j\in\N$ with $x\in K_j$.  By uniform convergence on $K_j$, there is
$N_j\in\N$ such that
\[
        |F_N(y)-F(y)|<\frac{\eta}{2}, \quad \forall \ y\in K_j,\ \forall \ N\ge N_j.
\]
Thus for every $m\ge N_j$, we have 
\[
        \sup_{N\ge m}|F_N-F|
        \le
        \frac{\eta}{2}+H\mathbf 1_{K_j^c}.
\]
Averaging over balls of positive measure with center at $x$ yields 
\[
        \sup_{N\ge m}\mathcal L_r(|F_N-F|)(x)\le \mathcal L_r\left(\frac{\eta}{2} +H\mathbf 1_{K_j^c}\right)(x).
\]
Since $x\in K_j$,  $H(x)\mathbf 1_{K_j^c}(x)=0$.  Hence by the
choice of $X_0$, we have 
\[
        \lim_{r\downarrow 0} \mathcal L_r(H\mathbf 1_{K_j^c})(x)=0.
\]
It follows that for every $m\ge N_j$,
\[
        \limsup_{r\downarrow0}
        \sup_{N\ge m}\mathcal L_r(|F_N-F|)(x)
        \le
        \frac{\eta}{2}.
\]
Letting $m\to \infty$, and then letting $\eta\downarrow0$, we obtain
\[
        \lim_{m\to \infty}\limsup_{r\downarrow0}
        \sup_{N\ge m}\mathcal L_r(|F_N-F|)(x)=0.
\]
Since $x\in X_0$ was arbitrary, the above equality holds almost everywhere. 
Consequently, by Lemma \ref{lem:local-stability-consequences} $F_N\to F$ is locally stable.
\end{proof}

We identify $\mathbb T=\R/\Z$ with $[0,1)$, equipped with normalized Lebesgue measure $\lambda$. 

\begin{ex}\label{ex:pointwise-not-local}
Almost-everywhere convergence need not be locally stable, even for a
uniformly $L^1$-bounded sequence on  $\mathbb T=\R/\Z$.
\end{ex}

\begin{proof}
For $N\in\N$, let $\alpha_N=2^{-N}$ and $M_N=2^{3N}$. For $0\le j<M_N$, set
\[
    Q_{N,j}=\left[\frac{j}{M_N},\frac{j+1}{M_N}\right) \quad \text{ and } \quad E_{N,j}=\left[\frac{j}{M_N}, \frac{j+\alpha_N}{M_N}\right).  
\]  
Define  
\[
  E_N=\bigsqcup_{j=0}^{M_N-1}E_{N,j} \quad \text{ and } \quad  F_N=\alpha_N^{-1}\mathbf 1_{E_N}.
\]
It is clear that $\lambda(E_N)=\alpha_N$, $\|F_N\|_1=1$ and $\sum_{N=1}^{\infty}\lambda(E_N)=1$, where $\lambda$ is the Lebesgue measure.
Hence  by the Borel--Cantelli lemma,
\[
    F_N(x)\to 0 \quad\text{for a.e. }x\in\mathbb T.
\]

Let $I\subset\mathbb T$ be an interval. Put
\[
    J_N(I)=\{0\le j<M_N:Q_{N,j}\subset I\} \quad \text{and } \quad U_N(I)=\bigsqcup_{j\in J_N(I)}Q_{N,j}.
\]
Since at most two partition intervals intersect $I$ without being contained in it, we have
\[
    \lambda(U_N(I))\ge |I|-\frac{2}{M_N}.
\]
Observe that 
\[
    \int_{Q_{N,j}}F_N\,\dd\lambda=\alpha_N^{-1}\lambda(E_{N,j})= \frac{1}{M_N}.
\]
Thus
\[
    \int_I F_N\,\dd\lambda  \ge \sum_{j\in J_N(I)}\int_{Q_{N,j}}F_N\,\dd\lambda=\lambda(U_N(I))\ge |I|-\frac{2}{M_N},
\]
and consequently
\[
    \frac{1}{|I|}\int_I F_N\,\dd\lambda \ge 1-\frac{2}{M_N|I|}.
\]

Fix $x\in\mathbb T$, $m\in\N$ and $0<r<1/2$. Taking
$I=B(x,r)$, we obtain that
\[
\begin{aligned}
    \sup_{N\ge m}\mathcal L_r(F_N)(x)\ge\limsup_{N\to \infty} \left(1-\frac{4}{M_N\lambda(B(x,r))}\right) =1,
\end{aligned}
\]
and consequently 
\[
    \liminf_{r\downarrow 0} \sup_{N\ge m}\mathcal L_r(F_N)(x)\ge 1.
\]
Hence, $F_N\to 0$ a.e. but the convergence is not locally stable.
\end{proof}

\section{\texorpdfstring{The sharp $L\log_{q+1}L$ Orlicz endpoint}{The sharp iterated-logarithmic endpoint at level q+1}}\label{sec:Llogqp1}

This section establishes the sharp $L\log_{q+1}L$ Orlicz endpoint for local observation problems.
The upper endpoint bound  is based on the $\RME_q$ property of the underlying system and requires no further assumptions on the time sequence. 
In contrast, the matching lower bound is established under a polynomial-growth assumption on the time sequence. After proving the abstract sharp endpoint theorem, we specialize it to two principal cases: averages along consecutive times and prime times.

\subsection{Restricted logarithmic maximal estimates}\label{subsec:RME}
This subsection introduces a restricted integral estimate for indicator functions, through endpoint lifting of which we obtain the required integrable maximal majorant in Theorem \ref{thm:local-stability-principle}.

Let $\mathbf a=(a_n)_{n\ge 1}$ be a sequence of nonnegative integers.  For a measure-preserving system $(X,\B_X,\mu, T)$ and a measurable function $f$ on $X$, define
\[
        A_N^{\mathbf a}f(x)=\frac{1}N\sum_{n=1}^N f(T^{a_n}x), \quad   A_{\mathbf a}^*f(x)=\sup_{N\ge 1}A_N^{\mathbf a}|f|(x),
\]
and
\[
          A_{N}^{\mathbf a, q}f(x)=\frac{A_N^{\mathbf a}f(x)}{\Lambda_q(N)}, \quad   M_{\mathbf a,q}f(x)=\sup_{N\ge 1} A_{N}^{\mathbf a, q}|f|(x)
\] 
for $q\in\N_0$.

\begin{defn} \label{def:RME}
For $q\in \N_0$, we say that $\mathbf a$ satisfies \emph{restricted maximal estimate of level $q$} ($\RME_q$) on a measure-preserving system $(X,\B_X,\mu, T)$ if
there is $C=C_{\mathbf a, q, T}>0$ such that for every $E\in \B_X$,
\[
        \int_X M_{\mathbf a,q}\1_E\dd \mu\le C\,\mu(E) \left(1+L_{q+1}\left(\frac{1}{\mu(E)}\right)\right),
\]
with the right-hand side interpreted as zero if $\mu(E)=0$.  

If the constant can be chosen independently of the measure-preserving system, we say that $\mathbf a$ satisfies $\RME_q$ uniformly over measure-preserving systems.
\end{defn}

The next proposition establishes the critical one-logarithm lifting step. 
The estimate $\mathrm{RME}_q$ is applied to the superlevel sets of any nonnegative function, and integrating over these levels yields an $L^1$ bound for the maximal function.

\begin{prop}\label{prop:RME-lift}
Fix $q\in\N_0$, a measure-preserving system $(X,\B_X,\mu,T)$, and a sequence $\mathbf a=(a_n)_{n\ge 1}$ of nonnegative integers. 
 Assume that $\mathbf a$ satisfies $\RME_q$ on this system, and let $C_0$ be an admissible constant in Definition~\ref{def:RME}.  If $f\in L^{\Psi_{q+1}}$, then
\[
\begin{aligned}
        \int_X M_{\mathbf a,q}f\dd \mu &\le C_0\int_0^\infty d_f(t)\left(1+L_{q+1}\left(\frac{1}{d_f(t)}\right)\right)\dd t    \\
        &\le C_0 C_q\left(1+ \int_X |f|\bigl(1+L_{q+1}(\mathrm e+|f|)\bigr)\dd \mu\right),
\end{aligned}
\]
where $ d_f(t)=\mu(\{x\in X:|f(x)|>t\})$, and the related integrand is interpreted as $0$ when
$d_f(t)=0$.  Moreover,
\[
        \|M_{\mathbf a,q}f\|_{L^1}\le C_0 C_q\|f\|_{\Psi_{q+1}}.
\]
Here $C_q$ depends only on $q$. 
 In particular, if $\mathbf a$ satisfies $\RME_q$ uniformly over measure-preserving systems, then the above estimates
hold uniformly over measure-preserving systems.
\end{prop}

\begin{proof}
Applying the layer-cake formula to the nonnegative function $|f|$, we have for a.e. $x\in X$, 
\[
        |f(x)|=\int_0^\infty \mathbf 1_{\{|f|>t\}}(x)\dd t.
\]
Since the averages $A_N^{\mathbf a}$ are positive and linear, 
\[
        A_N^{\mathbf a}|f|(x)=\int_0^\infty A_N^{\mathbf a}\mathbf 1_{\{|f|>t\}}(x)\dd t.
\]
It follows that for a.e. $x\in X$,
\[
\begin{aligned}
     M_{\mathbf a,q}f(x)  = \sup_{N\ge 1}\frac{A_N^{\mathbf a}|f|(x)}{\Lambda_q(N)}  \le \int_0^\infty \sup_{N\ge 1} 
                            \frac{A_N^{\mathbf a}\mathbf 1_{\{|f|>t\}}(x)} {\Lambda_q(N)}\dd t  
                         =\int_0^\infty M_{\mathbf a,q}\mathbf 1_{\{|f|>t\}}(x)\dd t.
\end{aligned}
\]
Then by Tonelli's theorem, 
\[
        \int_X M_{\mathbf a,q}f\dd \mu \le\int_0^\infty \int_X M_{\mathbf a,q}\mathbf 1_{\{|f|>t\}}\dd \mu\dd t.
\]
For every $t\ge 0$, applying $\RME_q$ to the set
$\{|f|>t\}$ yields
\[
        \int_X M_{\mathbf a,q}\mathbf 1_{\{|f|>t\}}\dd \mu\le   C_0\, d_f(t) \left(1+L_{q+1}\left(\frac{1}{d_f(t)}\right)\right),
\]
with the right-hand side interpreted as $0$ if $d_f(t)=0$.  Hence
\[
\begin{aligned}
 \int_X M_{\mathbf a,q}f\dd \mu &\le C_0\int_0^\infty d_f(t) \left(1+L_{q+1}\left(\frac{1}{d_f(t)}\right)\right)\dd t\\
        &\le C_0 C_q^\prime\left(1+ \int_X |f|\bigl(1+L_{q+1}(\mathrm e+|f|)\bigr)\dd \mu\right),
\end{aligned}
\]
where the second inequality follows from Lemma~\ref{lem:dist-bound} applied to $|f|$, with  $C_q^\prime>0$ depending only on $q$.

It remains to prove the Luxemburg estimate. Let $\lambda>\|f\|_{\Psi_{q+1}}$.  Then 
\[
        \int_X \Psi_{q+1}\left(\frac{|f|}{\lambda}\right)\,d\mu\le 1.
\]
Applying \eqref{eq:Psi-comparison} to $s=q+1$,  there is $C_q^{\prime\prime}>0$ such that  
\[
        \int_X \frac{|f|}{\lambda}\left(1+L_{q+1}\left(\mathrm e+\frac{|f|}{\lambda}\right)\right) \dd\mu \le C_q^{\prime\prime}.
\]
Applying the preceding distributional estimate to $f/\lambda$, we have 
\[
        \|M_{\mathbf a,q}(f/\lambda)\|_{L^1 } \le C_0C_q^{\prime}(1+C_q^{\prime\prime}).
\]
Absorbing constants depending only on $q$, write $C_q=\max\{C_q^{\prime},C_q^{\prime}(1+C_q^{\prime\prime})\}$. 
Since $M_{\mathbf a,q}$ is homogeneous,  then
\[
        \|M_{\mathbf a,q}f\|_{L^1(\mu)} \le C_0C_q\lambda.
\]
Consequently, letting $\lambda\downarrow\|f\|_{\Psi_{q+1}}$ gives
\[
        \|M_{\mathbf a,q}f\|_{L^1(\mu)} \le C_0C_q\|f\|_{\Psi_{q+1}}.
\]
This finishes the whole proof.
\end{proof}

\begin{prop}\label{prop:RME-local}
Fix $q\in\N_0$, a metric probability space $(X,d,\mu)$ with the LDP, a measure-preserving transformation $T$, and a sequence $\mathbf a=(a_n)_{n\ge 1}$ of nonnegative integers.  Assume that $\mathbf a$ satisfies $\RME_q$ on $(X,\B_X,\mu,T)$.  
If $f\in L^{\Psi_{q+1}}$ and
\[
        A_{N}^{\mathbf a, q}f\to F\quad\text{a.e.}, \quad F\in L^1(\mu),
\]
then the convergence is locally stable.
\end{prop}

\begin{proof}
By Proposition~\ref{prop:RME-lift}, $M_{\mathbf a,q}f\in L^1(\mu)$.  Since for every $N\ge 1$, 
\[
        |A_{N}^{\mathbf a, q}f|\le M_{\mathbf a,q}f.
\]
Thus the convergence is locally stable by Theorem~\ref{thm:local-stability-principle}.
\end{proof}

We say that the maximal operator $A_{\mathbf a}^*$ is of \emph{weak type $(1,1)$} on a measure-preserving system if there exists $C_w>0$ such that
\[
        \mu\bigl(\{A_{\mathbf a}^*g>\alpha\}\bigr) \le \frac{C_w}{\alpha}\|g\|_{L^1(\mu)}
\]
for every $g\in L^1(\mu)$ and every $\alpha>0$.  
Equivalently,
\[
        \|A_{\mathbf a}^*g\|_{L^{1,\infty}(\mu)}\le C_w\|g\|_{L^1(\mu)}.
\]
If the same constant $C_w$ holds for every  measure-preserving system, we say that $A_{\mathbf a}^*$ is \emph{uniformly of weak type $(1,1)$} over such systems.

\begin{prop}\label{prop:weak-RME}
Let $(X,\B_X,\mu,T)$ be a measure-preserving system and let $\mathbf a=(a_n)_{n\ge 1}$ be a sequence of nonnegative integers.   
If $A_{\mathbf a}^*$ is of \emph{weak type $(1,1)$} on  $(X,\B_X,\mu,T)$ with weak type constant $C_w$,
then $\mathbf a$ satisfies $\RME_q$ on this system for every $q\in\N_0$.  

If $A_{\mathbf a}^*$ is uniform weak type $(1,1)$ over measure-preserving systems, then $\mathbf a$ satisfies
$\RME_q$ uniformly over such systems for every $q\in\N_0$.
\end{prop}

\begin{proof}
Fix $q\in\N_0$ and $E\in\B_X$.  Put $\delta=\mu(E)$ and $g=M_{\mathbf a,q}\mathbf 1_E$.
The case $\delta=0$ is trivial.  If $\delta \ge 1/8$, since $0\le g\le 1$, we have
\[
        \int_X g\dd\mu\le 1 \le 8 \delta,
\]
which is sufficient after enlarging the constant.  Now we assume that $0<\delta<1/8$.

Let $g^*$ be the decreasing rearrangement of $g$.  Since $0\le g\le 1$ a.e., we have $0\le g^*\le1$ a.e., and hence 
\[
        \int_0^{4\delta}g^*(s)\dd s\le 4\delta.
\]

We claim that for $4\delta<s<1$, there exists $C_{q,w}>0$  depending only on $q$ and $C_w$  such that 
\[
        g^*(s)\le C_{q,w}\frac{\delta}{s\Lambda_q(s/\delta)}.
\]
For $4\delta<s\le8\delta$, this follows from $g^*(s)\le 1$, since $s/\delta\in(4,8]$ and $\Lambda_q$ is bounded on this finite interval.  
Now let $8\delta<s<1$ and set $m=\left\lfloor\frac{s}{4\delta}\right\rfloor$.
Clearly 
\[
        \frac{s}{8\delta}\le m\le \frac{s}{4\delta}\ \text{ and } \ \Lambda_q(m)\asymp_q\Lambda_q(s/\delta).
\]
Define $S_s=\{x:\exists\,1\le n<m,\ T^{a_n}x\in E\}$. Since $T$ is measure-preserving, we have
\[
        \mu(S_s)\le m\delta\le s/4.
\]
For every fixed $N\ge m$, observe that 
\[
        \frac{A_N^{\mathbf a}\mathbf 1_E}{\Lambda_q(N)}\le \frac{A_{\mathbf a}^*\mathbf 1_E}{\Lambda_q(m)}.
\]
Set
\[
        \alpha_s=\frac{4C_w\delta}{s\Lambda_q(m)}
\]
and
\[
        L_s=\left\{\frac{A_{\mathbf a}^*\mathbf 1_E}{\Lambda_q(m)}>\alpha_s\right\}=\left\{A_{\mathbf a}^*\mathbf 1_E>\frac{4C_w\delta}{s} \right\}.
\]
Then by the weak type estimate assumption, 
\[
        \mu\left(L_s\right)\le \frac{C_w\delta}{4C_w\delta/s}=s/4.
\]
It follows that 
\[
        g(x)=M_{\mathbf a,q}\mathbf 1_E(x)\le \alpha_s, \quad \forall \ x\notin S_s\cup L_s.
\]
Consequently,
\[
        \mu(\{g>\alpha_s\})\le \mu(S_s)+\mu(L_s)\le \frac{s}{2}<s,
\]
and so $ g^*(s)\le \alpha_s$. This proves the claim.

Using the claim, we obtain that 
\[
\begin{aligned}
        \int_X g \dd\mu &=\int_0^1g^*(s)\dd s    \\
        &\le 4\delta+ C_{q,w}\delta  \int_{4\delta}^{1}\frac{\dd s}{s\Lambda_q(s/\delta)}   \\
        &= 4\delta+ C_{q,w}\delta\int_4^{1/\delta}\frac{\dd u}{u\Lambda_q(u)}   \\
        &\le C_{q,w}\delta \left(1+L_{q+1}\left(\frac{1}\delta\right)\right),
\end{aligned}
\]
where the last step follows from Lemma~\ref{lem:log-calc}.   This is exactly $\RME_q$ on the given system.  The uniform statement follows immediately
when the weak type constant $C_w$ is uniform over systems.
\end{proof}

\begin{thm}\label{thm:RME-convergence-zero}
Fix $q\in\N$, a measure-preserving system $(X,\B_X,\mu,T)$, and a sequence $\mathbf a=(a_n)_{n\ge 1}$ of nonnegative integers.  
Assume that $\mathbf a$ satisfies $\RME_q$ on this system.  Then for every $f\in L^{\Psi_{q+1}}$,
\[
        A_{N}^{\mathbf a, q}f(x)\to 0 \quad \text{for a.e. } x\in X.
\]  
If, in addition,  $(X,d,\mu)$ has the LDP, then the convergence is locally stable.
\end{thm}

\begin{proof}
For bounded $f$,
\[
        |A_{N}^{\mathbf a, q}f|
        \le \frac{\|f\|_\infty}{\Lambda_q(N)}
        \to 0,
\]
since $q\ge 1$.  For general $f\in L^{\Psi_{q+1}}$, let
\[
        f_K=f\mathbf 1_{\{|f|\le K\}}, \ \text{ and } \ h_K=f-f_K.
\]
Then $A_{N}^{\mathbf a, q}f_K\to 0$ a.e., and hence
\[
\begin{aligned}
   \limsup_{N\to \infty} |A_{N}^{\mathbf a, q}f| & \le \limsup_{N\to \infty} |A_{N}^{\mathbf a, q}f_K|+\limsup_{N\to \infty} |A_{N}^{\mathbf a, q}h_K| \le M_{\mathbf a,q}h_K   \quad\text{a.e.}
\end{aligned}
\]
Therefore  for $\varepsilon>0$, by  Chebyshev's inequality and Proposition \ref{prop:RME-lift}, 
\[
\begin{aligned}
        \mu\left(\left\{\limsup_{N\to \infty} |A_{N}^{\mathbf a, q}f|>\varepsilon\right\}\right)   \le\mu\left(\{M_{\mathbf a,q}h_K>\varepsilon\}\right)  
         \le \frac{C}{\varepsilon} \|h_K\|_{\Psi_{q+1}}.
\end{aligned}
\]
Using Lemma~\ref{lem:lux-tail}, we have  $\|h_K\|_{\Psi_{q+1}}\to 0$. Then it implies that 
\[
        \limsup_{N\to \infty} |A_{N}^{\mathbf a, q}f|=0
        \quad\text{a.e. }  
\]
Thus  $A_{N}^{\mathbf a, q}f(x)\to 0$ for a.e. $x\in X$. The local stability follows
from Proposition~\ref{prop:RME-local} with $F=0$.
\end{proof}

\subsection{Lower endpoint construction for polynomial-growth times} \label{subsec:polynomial-lower}
To obtain the lower bound in the sharp endpoint theorem, we first establish a weighted local sweeping out construction on the Lebesgue interval.

\begin{thm}\label{thm:weighted-sweeping}
Let $\mathbf a=(a_n)$ be a strictly increasing sequence of nonnegative integers. Let $k_0\in\N$ and $(\beta_k)_{k\ge k_0}$ be positive numbers such that
\[
        \sum_{k=k_0}^{\infty}\beta_k=\infty.
\]
Then for every $0<\varepsilon<1$ and $0<\gamma<1$, there exist an invertible ergodic Lebesgue measure-preserving transformation $T$ of $[0,1]$, 
a nonnegative $f\in L^1([0,1])$, and a measurable set $Y\subset[0,1]$ such that
\[
        \lambda(Y)>1-\varepsilon, \quad \int f\dd\lambda<\gamma,
\]
and
\[
        \limsup_{\substack{r\downarrow 0, k\to \infty}}\mathcal L_r\left(\beta_k\sum_{n=1}^k f(T^{a_n}\cdot)\right)(x)=\infty
\]
for a.e. $x\in Y$.
\end{thm}

\begin{proof}
Without loss of generality, all intervals below are assumed to be half-open, and at the end we discard the countable set of endpoints which occurs in the construction.

Fix $0<\varepsilon<1$ and $0<\gamma<1$.  Choose positive numbers $c_{m,s}$, $m,s\in\N$, such that
\begin{equation*} 
        \sum_{m,s}c_{m,s}<\gamma.
\end{equation*}
Fix an enumeration  of $\N^2$.  For each pair $(m,s)$, put $b_{m,s}=\max\{s,k_0\}$. 
Since every tail of the divergent series  $\sum_k\beta_k$ diverges, choose $K_{m,s}\ge b_{m,s}+1$ such that
\begin{equation}\label{eq:weighted-capacity}
        \sum_{k=b_{m,s}}^{K_{m,s}}\frac{c_{m,s}\beta_k}{4m}>2.
\end{equation}
Choose an integer $L_{m,s}>a_{K_{m,s}}$. Since the sequence $\mathbf a=(a_n)$ is strictly increasing, 
the integers $L_{m,s}-a_k$, $b_{m,s}\le k\le K_{m,s}$, are  distinct elements of $\{0,\ldots,L_{m,s}\}$.
Choose $V_{m,s}>0$ sufficiently large that
\begin{equation}\label{eq:weighted-height-condition}
        V_{m,s}\ge 4m\max_{b_{m,s}\le k\le K_{m,s}}\frac{1}{\beta_k},
\end{equation}
\begin{equation}\label{eq:weighted-count-condition}
        (K_{m,s}-b_{m,s}+1) \frac{c_{m,s}}{V_{m,s}}\le\frac{1}{2},
\end{equation}
and
\begin{equation}\label{eq:weighted-mass-condition}
        (L_{m,s}+1)\frac{c_{m,s}}{V_{m,s}}<\frac{\varepsilon}{2^{m+s+4}}.
\end{equation}
For brevity, write 
\[
        \delta_{m,s}=\frac{c_{m,s}}{V_{m,s}} \ \text{ and }\  B_{m,s,k}=\frac{c_{m,s}\beta_k}{4m}.
\]
By \eqref{eq:weighted-height-condition}, it is clear that for every $  b_{m,s}\le k\le K_{m,s}$,  $\delta_{m,s}\le B_{m,s,k}$.
 
We choose the tower levels inductively along the fixed enumeration of $\N^2$.  Suppose that the levels corresponding to the preceding pairs
have already been selected, and set
\[
        R_{m,s}=[0,1]\setminus \bigcup_{(u,v)\prec(m,s)} \bigcup_{t=0}^{L_{u,v}} E_{u,v,t}.
\]
At each finite stage $(m,s)$, after discarding finitely many endpoints, $R_{m,s}$ is a finite union
of half-open intervals.  Moreover,
\[
        \lambda(R_{m,s})> 1-\sum_{u,v\ge 1}\frac{\varepsilon}{2^{u+v+4}}=1-\frac{\varepsilon}{16}> \frac{1}{2}. 
\]
By \eqref{eq:weighted-count-condition},
\[
        \sum_{k=b_{m,s}}^{K_{m,s}}\delta_{m,s}
        \le\frac{1}{2}<\lambda(R_{m,s}),
\]
whereas \eqref{eq:weighted-capacity} gives
\[
        \sum_{k=b_{m,s}}^{K_{m,s}}B_{m,s,k}
        >2>\lambda(R_{m,s}).
\]
Lemma~\ref{lem:fine-filling}, applied to $R_{m,s}$ with mesh $2^{-m-s}$, therefore gives a partition
\[
        R_{m,s}
        =
        \bigsqcup_{k=b_{m,s}}^{K_{m,s}}H_{m,s,k}
\]
modulo endpoints, where each $H_{m,s,k}$ is a finite union of half-open intervals,
\[
      \delta_{m,s} \le\lambda(H_{m,s,k}) \le B_{m,s,k},
\]
and every component interval of every $H_{m,s,k}$ has length at most $2^{-m-s}$.

 In each component interval of $H_{m,s,k}$, take the left subinterval of
relative length
\[
        \alpha_{m,s,k}=\frac{\delta_{m,s}}{\lambda(H_{m,s,k})}.
\]
Let $C_{m,s,k}$ be the union of these selected subintervals.  Then
\[
        \lambda(C_{m,s,k})=\delta_{m,s}.
\]
For each $ b_{m,s}\le k\le K_{m,s}$, prescribe
\[
        E_{m,s,L_{m,s}-a_k}=C_{m,s,k}.
\]
The remaining levels $E_{m,s,t}$, $0\le t\le L_{m,s}$, are chosen disjointly inside the part of $R_{m,s}$ which is not yet occupied  by the
prescribed sets, each of measure $\delta_{m,s}$.  This is possible since by \eqref{eq:weighted-mass-condition},
\[
        (L_{m,s}+1)\delta_{m,s} <\varepsilon 2^{-m-s-4} <\lambda(R_{m,s}).
\]
Thus all tower levels are pairwise disjoint, and
\[
        \sum_{m,s}(L_{m,s}+1)\delta_{m,s}  <\frac{\varepsilon}{16}.
\]
    
By Lemma~\ref{lem:tower-embedding}, there is an invertible ergodic Lebesgue measure-preserving transformation $T$ of $[0,1]$ such that
\[
        T(E_{m,s,t})=E_{m,s,t+1}, \quad 0\le t<L_{m,s}.
\]
Define $ f=\sum_{m,s}V_{m,s}\mathbf 1_{E_{m,s,L_{m,s}}}$.
Since the top levels are pairwise disjoint, 
\[
        \int f\dd\lambda  =\sum_{m,s}V_{m,s}\lambda(E_{m,s,L_{m,s}}) =\sum_{m,s}c_{m,s}<\gamma. 
\]
Let $Y$ be the complement of the union of all tower levels.  Then 
\[
  \lambda(Y)> 1-\frac{\varepsilon}{16}>1-\varepsilon.
\]

Now fix $x\in Y$ outside the countable set of endpoints.  For each $(m,s)$, we have $x\in R_{m,s}$, and hence $x$ belongs to a unique component
interval $J_{m,s}(x)$ of a unique set $H_{m,s,k_{m,s}(x)}$.  Fix $m\in\N$. Since $k_{m,s}(x)\ge b_{m,s}\ge s$, we have $k_{m,s}(x)\to \infty$ as $s\to \infty$.  Put
\[
        \rho_{m,s}(x)=2|J_{m,s}(x)|.
\]
Since $|J_{m,s}(x)|\le2^{-m-s}$, then $\rho_{m,s}(x)\to 0$ as $s\to \infty$. Moreover, 
\[
        J_{m,s}(x)\subset B(x,\rho_{m,s}(x)) \ \text{ and }\ \lambda(B(x,\rho_{m,s}(x)))\le 4|J_{m,s}(x)|.
\]

Write $k = k_{m,s}(x)$. By construction, $C_{m,s,k}$ occupies a proportion $\alpha_{m,s,k}$ of every component of $H_{m,s,k}$, so $C_{m,s,k} \cap J_{m,s}(x)$ has length $\alpha_{m,s,k}|J_{m,s}(x)|$. Thus
\[
        \frac{\lambda(B(x,\rho_{m,s}(x))\cap C_{m,s,k})} {\lambda(B(x,\rho_{m,s}(x)))} \ge \frac{\alpha_{m,s,k}}4. 
\]
Since $C=E_{m,s,L_{m,s}-a_k}$, by iterating the tower relations
\[
        T^{a_k}y\in E_{m,s,L_{m,s}}
\]
for a.e. $y\in C_{m,s,k}$.  Therefore
\[
        f(T^{a_k}y)\ge V_{m,s}
\]
for a.e. $y\in C_{m,s,k}$.  
It follows that 
\[
\begin{aligned}
\mathcal L_{\rho_{m,s}(x)}\left(\beta_k\sum_{n=1}^k f(T^{a_n}\cdot)\right)(x)
&  \ge  \beta_k V_{m,s} \frac{\lambda(B(x,\rho_{m,s}(x))\cap C_{m,s,k})} {\lambda(B(x,\rho_{m,s}(x)))}   \\
&  \ge  \frac{\beta_kV_{m,s}}4 \frac{\delta_{m,s}}{\lambda(H_{m,s,k})}    \\
&  \ge \frac{\beta_kV_{m,s}}4 \frac{c_{m,s}/V_{m,s}}{c_{m,s}\beta_k/(4m)} =m. 
\end{aligned}
\]
For each fixed $m$, this estimate holds along a sequence $s\to \infty$ for which $\rho_{m,s}(x)\to 0$ and $ k_{m,s}(x)\to \infty$.
Since $m$ is arbitrary, the joint local limsup is infinite at every $x\in Y$ outside the discarded countable set.
\end{proof}

\begin{defn}\label{defn:polynomial-growth}
  We say that a sequence $0\le a_1<a_2<\cdots$ has \emph{polynomial growth} if there are constants $A,D>0$ such that
\[
        a_N\le A N^D \quad (N\ge 1).
\]
After increasing $A$, we may equivalently assume that
\[
        a_N+1\le A(N+1)^D  \quad (N\ge 1).
\]
\end{defn}

Now we specialize the preceding weighted construction to the weights
\[
        \beta_N=\frac{1}{N\Lambda_q(N)}.
\]
For sequences of polynomial growth, the tower parameters can be chosen so that the resulting sweeping out function belongs to any prescribed Orlicz
space below the endpoint determined by $tL_{q+1}(t)$.

\begin{thm}\label{thm:poly-lower}
Let $q\in\N_0$, and  $\mathbf a=(a_n)$ be a strictly increasing sequence of nonnegative integers with polynomial growth.  
Let $\Phi$ be a  Young function such that
\[
        \Phi(t)=o\bigl(tL_{q+1}(t)\bigr) \quad (t\to \infty).
\]
Then the pair $(\mathbf a, N\Lambda_q(N))$ admits an \emph{$L^\Phi$ local $\infty$-sweeping out} counterexample on $[0,1]$. 

Indeed, we can prove a stronger form: for every $0<\varepsilon<1$, there exist an invertible ergodic Lebesgue measure-preserving transformation $T$ of $[0,1]$, a nonnegative function $f\in L^\Phi([0,1])$, and a measurable set $Y\subset[0,1]$ such that
\[
        \lambda(Y)>1-\varepsilon
\]
and
\[
        \limsup_{\substack{r\downarrow 0, N\to \infty}} \mathcal L_r\left( \frac{1}{N\Lambda_q(N)}\sum_{n=1}^N f(T^{a_n}\cdot)\right)(x) =\infty
\]
for a.e. $x\in Y$.
\end{thm}

\begin{proof}
Choose $A,D>0$ such that
\begin{equation*}\label{eq:poly-growth-bound}
        a_N+1\le A(N+1)^D, \quad N\ge 1.
\end{equation*}
Throughout the proof, constants of the form $C_q$, or  $C_{q,A,D,\varepsilon}$, are allowed to differ on each line to avoid excessive notation.

For $ k\ge 1$, set
\[
        \beta_k=\frac{1}{k\Lambda_q(k)}.
\]
By Lemma~\ref{lem:log-calc}\ref{lem:log-calc:2}, we have 
\[
        \sum_{k=1}^{\infty}\beta_k=\infty.
\]

Choose positive numbers $\eta_{m,s}$, $m,s\in\N$, such that
\begin{equation}\label{eq:poly-eta}
        \sum_{m,s\ge 1}(m+1)\eta_{m,s}<\infty.
\end{equation}
Fix $q\in\N_0$. By assumption,
\[
        \frac{\Phi(t)}{tL_{q+1}(t)} \to 0, \quad t\to \infty,
\]
so for every $(m,s)$ there exists $M_{m,s}\ge 1$ such that  
\begin{equation}\label{eq:poly-Phi-control}
        \frac{\Phi(t)}t\le \eta_{m,s}L_{q+1}(t)
\end{equation}
for all $t\ge M_{m,s}$.

Fix an enumeration  of $\N^2$. For a pair $(m,s)$, put $  b_{m,s}=s$.   Choose $c_{m,s}>0$ so that
\begin{equation}\label{eq:poly-c-choice}
 c_{m,s}=\min\left\{2^{-(m+s+2)},\frac{1}{1+L_{q+1}(M_{m,s})},\frac{1}{1+L_{q+1}(b_{m,s})}\right\}.
\end{equation}
Let $K_{m,s}$ be the minimal integer $K\ge b_{m,s}$ such that
\begin{equation}\label{eq:poly-capacity}
        \frac{c_{m,s}}{4m} \sum_{k=b_{m,s}}^K \frac{1}{k\Lambda_q(k)}
        >2.
\end{equation}
Such an integer exists because every tail of $\sum_{k=1}^{\infty}\frac{1}{k\Lambda_q(k)}$ diverges.  Since the left-hand side of
\eqref{eq:poly-capacity} is smaller than $2$ when $K=b_{m,s}$, it follows that $K_{m,s}>b_{m,s}$.
By minimality,  
\[
        c_{m,s}\sum_{k=b_{m,s}}^{K_{m,s}-1} \frac{1}{k\Lambda_q(k)}\le 8m.
\]
Thus, noting that $\beta_k\le 1$, we have 
\begin{equation}\label{eq:poly-tail-bound}
        c_{m,s} \sum_{k=b_{m,s}}^{K_{m,s}} \frac{1}{k\Lambda_q(k)}\le 8m+1.
\end{equation}
Moreover, by Lemma~\ref{lem:log-calc}\ref{lem:log-calc:2} and \eqref{eq:poly-c-choice}, there is a constant $C_q>0$ such that
\[
        c_{m,s} \sum_{k=1}^{b_{m,s}-1}\frac{1}{k\Lambda_q(k)}\le C_q,
\]
where the sum is understood to be zero when $b_{m,s}=1$.
Consequently, for sufficiently large  $C_q>0$, we have
\[
        c_{m,s}\left(1+\sum_{k=1}^{K_{m,s}}\frac{1}{k\Lambda_q(k)}\right)\le C_q(m+1).
\]
Since $c_{m,s}\beta_{K_{m,s}+1}\le c_{m,s}<1$, the same estimate holds with $K_{m,s}+1$ in place of $K_{m,s}$.  
Therefore the lower estimate in Lemma~\ref{lem:log-calc}\ref{lem:log-calc:2} yields
\begin{equation}\label{eq:poly-K-log}
        c_{m,s}\bigl(1+L_{q+1}(K_{m,s}+1)\bigr) \le C_q(m+1)
\end{equation}
for large  $C_q>0$.

For fixed $0<\varepsilon<1$,  set $L_{m,s}=a_{K_{m,s}}+1$ and
\begin{equation}\label{eq:poly-V-choice}
        V_{m,s}=\max\biggl\{M_{m,s}, 4m K_{m,s}\Lambda_q(K_{m,s}), 2(K_{m,s}-b_{m,s}+1)c_{m,s}, \frac{2(L_{m,s}+1)c_{m,s}} {\varepsilon2^{-m-s-4}} \biggr\}.
\end{equation}
Since each $L_j$ is increasing, $k\longmapsto k\Lambda_q(k)$ is increasing on $[0,\infty)$.  Hence
\[
        V_{m,s}\ge 4m\max_{b_{m,s}\le k\le K_{m,s}}\frac{1}{\beta_k}.
\]
The third and fourth terms in \eqref{eq:poly-V-choice} give
\[
        (K_{m,s}-b_{m,s}+1)\frac{c_{m,s}}{V_{m,s}}\le\frac{1}{2}
\]
and
\[
        (L_{m,s}+1)\frac{c_{m,s}}{V_{m,s}}< \frac{\varepsilon}{2^{m+s+4}}.
\]
Together with \eqref{eq:poly-capacity}, these are precisely the parameter  conditions  \eqref{eq:weighted-capacity}\eqref{eq:weighted-height-condition}\eqref{eq:weighted-count-condition}\eqref{eq:weighted-mass-condition} 
required in the construction from the proof of Theorem~\ref{thm:weighted-sweeping}.

We next claim that there is $C_{q,A,D,\varepsilon}>0$ such that 
\begin{equation}\label{eq:poly-V-log}
        c_{m,s}L_{q+1}(V_{m,s})\le C_{q,A,D,\varepsilon}(m+1).
\end{equation}
It suffices to check the four terms in \eqref{eq:poly-V-choice} separately.
The first term in \eqref{eq:poly-V-choice} is controlled by \eqref{eq:poly-c-choice}.  
For the second term, first note that
\[
        \Lambda_q(K)\le C_q(K+1)^q, \quad K\ge 1.
\]
Thus, combining Lemma~\ref{lem:log-calc}\ref{lem:log-calc:3} and Lemma~\ref{lem:Psi-estimates}\ref{lem:Psi-estimates:3}, we obtain  
\[
   L_{q+1}\bigl(4m K_{m,s}\Lambda_q(K_{m,s}) \bigr) \le C_q\bigl(1+L_{q+1}(m+1)+L_{q+1}(K_{m,s}+1)\bigr)
\]
for some large $C_q>0$. After multiplication by $c_{m,s}$, one can verify that this is bounded by $C_q(m+1)$ by \eqref{eq:poly-K-log}.  
Similarly, for the third term, since $c_{m,s}\le 1$, 
\[
        2(K_{m,s}-b_{m,s}+1)c_{m,s} \le 2 K_{m,s},
\]
and therefore 
\[
   c_{m,s} L_{q+1}\bigl(2(K_{m,s}-b_{m,s}+1)c_{m,s} \bigr)\le C_q(m+1).
\]
for some large $C_q>0$.
For the last term, by $c_{m,s}<2^{-m-s}$ and
\eqref{eq:poly-growth-bound}, we get 
\[
        \frac{2(L_{m,s}+1)c_{m,s}}{\varepsilon2^{-m-s-4}}\le \frac{32}{\varepsilon}(L_{{m,s}}+1) \le  C_{A,\varepsilon}(K_{m,s}+1)^D.
\]
Using Lemma~\ref{lem:Psi-estimates}\ref{lem:Psi-estimates:3}, Lemma~\ref{lem:log-calc}\ref{lem:log-calc:3} and \eqref{eq:poly-K-log},  the required bound for this term follows. This proves the claim.

Applying the parameterized construction from the proof of Theorem~\ref{thm:weighted-sweeping}, with $k_0=1$ and $\delta_{m,s} =\frac{c_{m,s}}{V_{m,s}}$,
we obtain an invertible ergodic Lebesgue measure-preserving transformation $T$, a  measurable set $Y\subset[0,1]$ with $\lambda(Y)>1-\varepsilon$,
and pairwise disjoint tower tops such that
\[
        f=\sum_{m,s\ge 1} V_{m,s}\mathbf1_{E_{m,s,L_{m,s}}} \  \text{ and }\    \lambda(E_{m,s,L_{m,s}})=\frac{c_{m,s}}{V_{m,s}}.
\]
Since $V_{m,s}\ge M_{m,s}$,
\eqref{eq:poly-Phi-control} and \eqref{eq:poly-V-log} yield
\[
\begin{aligned}
        \int_{[0,1]}\Phi(f)\dd\lambda
        &= \sum_{m,s\ge 1} \frac{c_{m,s}}{V_{m,s}} \Phi(V_{m,s})\\
        &\le\sum_{m,s\ge 1} c_{m,s}\eta_{m,s} L_{q+1}(V_{m,s})\\
        &\le C_{q,A,D,\varepsilon}  \sum_{m,s\ge 1}\eta_{m,s}(m+1)  <\infty.
\end{aligned}
\]
Hence $f\in L^\Phi([0,1])$.

The conclusion of the weighted construction with
$\beta_k=(k\Lambda_q(k))^{-1}$ yields
\[
        \limsup_{\substack{r\downarrow 0,  k\to \infty}} \ \mathcal L_r\left(\frac{1}{k\Lambda_q(k)}\sum_{n=1}^k f(T^{a_n}\cdot) \right)(x)=\infty
\]
for a.e. $x\in Y$.  This proves the required assertion.
\end{proof}

The following abstract endpoint statement is an immediate consequence of Proposition~\ref{prop:RME-local} and Theorem~\ref{thm:poly-lower}.

\begin{coro}\label{cor:abstract-sharp}
Let $q\in\N_0$, and let $\mathbf a=(a_n)$ be a strictly increasing sequence of nonnegative integers with polynomial growth. 
If the following two conditions hold: 
\begin{enumerate}
  \item [(a)]
  for every metric probability space $(X,d,\mu)$ with the LDP and every measure-preserving transformation $T$, 
         the sequence $\mathbf a$ satisfies $\RME_q$ on $(X,\B_X,\mu,T)$;
\item [(b)]
  the normalized temporal averages associated with $(\mathbf a,\mathbf D^{(q)})$ converge almost everywhere for every $f\in L\log_{q+1}L$.
\end{enumerate} 
Then $L\log_{q+1}L$ is the sharp Orlicz endpoint for the corresponding local theorem.  More precisely:
\begin{enumerate}[label=\textup{(\roman*)},leftmargin=*]
\item on every metric probability space with the LDP, the normalized
averages are locally stable for every $f\in L\log_{q+1}L$;

\item for every  Young function $\Phi$ satisfying
\[
        \Phi(t)=o\bigl(tL_{q+1}(t)\bigr) \quad(t\to \infty),
\]
the pair $(\mathbf a,\mathbf D^{(q)})$ admits an $L^\Phi$ local $\infty$-sweeping out counterexample on $[0,1]$.
\end{enumerate}
\end{coro}

\subsection{Birkhoff averages along consecutive times}\label{subsec:birkhoff}
We now apply Corollary~\ref{cor:abstract-sharp} to the consecutive sequence. 
For a measure-preserving system $(X,\B_X,\mu,T)$, define
\[
        A_Nf=\frac{1}N\sum_{n=0}^{N-1}f\circ T^n  \ \text{ and }\  B_N^{\, q}f=\frac{A_Nf}{\Lambda_q(N)}.
\]

\begin{prop}\label{prop:birkhoff-RME}
Let $q\in\N_0$.  There exists a constant $C_q>0$, independent of the
measure-preserving system, such that
\[
        \int_X \sup_{N\ge 1}\frac{A_N\1_E}{\Lambda_q(N)}\dd \mu \le C_q\mu(E) \left(1+L_{q+1}\left(\frac{1}{\mu(E)}\right)\right)
\]
for every $E\in\B_X$, where the right-hand side is interpreted as zero if $\mu(E)=0$.  
\end{prop}

\begin{proof}
By the maximal ergodic theorem \cite[Theorem~2.24]{EinsiedlerWard}, for every $g\in L^1(\mu)$ and every $\alpha>0$,
\[
        \mu\left(\left\{x\in X: \sup_{N\ge 1}A_N|g|(x)>\alpha\right\}\right)\le \frac{\|g\|_1}{\alpha}.
\]
The weak-type constant $1$ is independent of the system.  The assertion now follows from Proposition~\ref{prop:weak-RME}.
\end{proof}

Now we are ready to prove \ref{thm:regular-endpoints}\ref{thm:regular-endpoints:1}.

\begin{proof}[Proof of \ref{thm:regular-endpoints}\ref{thm:regular-endpoints:1}]
By Proposition~\ref{prop:birkhoff-RME},  the consecutive sequence satisfies $\RME_q$, $q\in \N_0$, uniformly over all measure-preserving systems.
For $q=0$,  Birkhoff's pointwise ergodic theorem yields that for every $f\in L^{\Psi_1} =L\log L \subset L^1$,
\[
        B_N^{(0)}f=A_Nf \to  \mathbb E(f\mid\I) \quad \text{ a.e.},
\]
where $\I$ denotes the $T$-invariant $\sigma$-algebra.  
For $q\ge 1$, using Theorem \ref{thm:RME-convergence-zero}, we have 
\[
B_N^{\, q}f \to 0  \quad \text{ a.e. }
\]
for all $f\in L^{\Psi_{q+1}}= L\log_{q+1}L$.
Thus the sharp  $ L\log_{q+1}L $ endpoint conclusion follows from Corollary~\ref{cor:abstract-sharp}, applied to the polynomial growth sequence $a_n=n-1$.
\end{proof}

\begin{rem}\label{rem:young-questions}
For $q=0$,  \ref{thm:regular-endpoints}\ref{thm:regular-endpoints:1} shows that on every metric probability space satisfying the LDP, 
the Birkhoff averages converge locally stably for every $f\in L\log L$.  More precisely,
\[
        \lim_{\substack{r\downarrow 0, N\to \infty}}  \mathcal L_r(A_Nf)(x)=\mathbb E(f\mid\I)(x)
\]
for a.e. $x$.
In particular, since $\mu(X)=1$, $L^p \subset L\log L $ for every $p>1$, so the same conclusion holds for every $f\in L^p(\mu)$, $p>1$.
On the other hand, \ref{thm:regular-endpoints}\ref{thm:regular-endpoints:1}  also provides a counterexample on the Lebesgue interval showing that the conclusion
does not hold for all $f\in L^1$.
Consequently,   Question~9 of \cite{Young25} has an affirmative answer under the LDP alone, in fact at the sharp $L\log L$ Orlicz endpoint, while Question~10 of \cite{Young25} has a negative answer.
\end{rem}

In the usual pointwise strong sweeping out formulation, both the upper and lower limits are prescribed, typically as $1$ and $0$.  
Our local sweeping out notion only requires the joint upper limit to be infinite, so it is natural to ask about the corresponding lower limit.  
The next proposition shows that for nonnegative $L^1$-functions, it is determined by the Birkhoff limit. 

\begin{prop}\label{prop:birkhoff-joint-liminf}
Let $(X,d,\mu)$ be a metric probability space with the LDP, $T\colon X\to X$ be measure-preserving, and $f\in L^1(\mu)$ be nonnegative.  Then
\[
        \liminf_{\substack{r\downarrow 0,  N\to \infty}}\mathcal L_r(A_Nf)(x)=\mathbb E(f\mid\I)(x)
\]
for a.e. $x$, where $\I$ denotes the $T$-invariant $\sigma$-algebra.

In particular,  if $T$ is additionally ergodic and $f\not\equiv0$ in $L^1(\mu)$, then
\[
        \liminf_{\substack{r\downarrow 0, N\to \infty}}\mathcal L_r(A_Nf)(x)=\int_X f \dd\mu>0
\]
for a.e. $x$.
\end{prop}

\begin{proof}
For $M\in \N$, define $ f_M(x)=\min\{f(x),M\}$. The functions $f_M$ are bounded and increase pointwise to $f$, and $0\le A_Nf_M\le M$ for every $N\in\N$.  
Hence by  Birkhoff's pointwise ergodic theorem and Theorem~\ref{thm:local-stability-principle}, we have 
\[
        \lim_{\substack{r\downarrow 0, N\to \infty}}\mathcal L_r(A_Nf_M)(x) = \mathbb E(f_M\mid\I)(x)  \quad \text{ a.e.}
\]
for every $M$.  Since $0\le f_M\le f$, the positivity of $A_N$ yields 
\[
        \liminf_{\substack{r\downarrow 0, N\to \infty}} \mathcal L_r(A_Nf)(x) \ge \mathbb E(f_M\mid\I)(x) \quad \text{ a.e.}
\]
for every $M$.
Letting $M\to \infty$ and using monotone convergence for conditional expectations, we have
\[
        \liminf_{\substack{r\downarrow 0, N\to \infty}} \mathcal L_r(A_Nf)(x) \ge \mathbb E(f\mid\I)(x) \quad \text{ a.e.}
\]

For the reverse inequality, fix a point $x$ in the full-measure set on which Birkhoff's pointwise ergodic theorem holds for $f$ and 
the Lebesgue differentiation conclusion holds for every $A_Nf$.
For each $N\in \N$, choose $0<r_N<N^{-1}$ such that
\[
        \bigl|\mathcal L_{r_N}(A_Nf)(x)-A_Nf(x)\bigr|<N^{-1}.
\]
Then $r_N\to 0$, and
\[
        \mathcal L_{r_N}(A_Nf)(x) \to  \mathbb E(f\mid\I)(x).
\]
Therefore, 
\[
        \liminf_{\substack{r\downarrow 0, N\to \infty}} \mathcal L_r(A_Nf)(x) \le \mathbb E(f\mid\I)(x).
\]
Combining the two inequalities proves the asserted identity.

Finally, if $T$ is ergodic, $f\ge 0$ and $f\not\equiv 0$, then
\[
        \mathbb E(f\mid\I)=\int_X f\dd\mu >0 \quad\text{a.e.}
\]
The whole proof follows.
\end{proof}

\subsection{Averages along prime Times}\label{subsec:prime-endpoint}

In this subsection, we apply Corollary~\ref{cor:abstract-sharp} to the prime sequence and prove \ref{thm:regular-endpoints}\ref{thm:regular-endpoints:2}.
Starting from Trojan's restricted endpoint estimate, we establish the uniform $\RME_q$ property for every $q\ge 1$, deduce the sharp endpoint, and
show that uniform $\RME_q$ is strictly weaker than uniform weak type $(1,1)$. 
Throughout this subsection, positive constants are allowed to take different values on different lines to avoid excessive notation.

Let $\mathbf p=(p_n)_{n\ge 1}$ be the increasing enumeration of the primes.
For $R\ge 2$, set
\[
        \mathcal P_R f(x)= \frac1{\pi(R)}\sum_{\substack{p\le R,\  p\ {\rm prime}}}f(T^p x),
\]
where $\pi(R)$ is the number of primes not exceeding $R$, and define
\[
        P^*f(x) = \sup_{N\ge 1}\frac1N\sum_{n=1}^N |f|(T^{p_n}x).
\]
Since $\pi(R)=N$ for $p_N\le R<p_{N+1}$,
\[
        P^*f=\sup_{R\ge 2}\mathcal P_R|f|.
\]

\begin{prop}\label{prop:trojan-prime-estimate}
There exists a constant $C>0$  such that for every measure-preserving system $(X,\B_X,\mu,T)$, every $E\in\B_X$ and every $0<\alpha<1$, we have
\[
        \mu\left\{ x\in X: \sup_{R\ge 2}\mathcal P_R\1_E(x)>\alpha \right\}\le C\frac{\mu(E)}{\alpha} \log^2\left(\frac{\e}{\alpha}\right).
\]
\end{prop}

\begin{proof}
This estimate comes from \cite[Proposition~7.3]{Trojan}.  The formulation there requires only that $T$ be measure-preserving, and its proof shows that the constant is independent of the underlying system.
\end{proof}

\begin{prop}\label{prop:trojan-strong-prime}
There exists a constant $C>0$ such that for every measure-preserving system $(X,\B_X,\mu,T)$ and every $E\in\B_X$, we have
\[
        \int_X P^*\1_E \dd\mu \le C\mu(E)\left(1+L_1\left(\frac{1}{\mu(E)}\right)^3\right),
\]
where the right-hand side is understood as $0$ when $\mu(E)=0$.
\end{prop}

\begin{proof}
First assume that $\mu(E)=0$.  For every $n\ge 1$, by measure preservation of $T$ we have
\[
        \mu\left(\left\{x\in X:\1_E(T^{p_n}x)=1 \right\}\right)=\mu(T^{-p_n}E)=\mu(E)= 0.
\]
It follows that 
\[
        \sup_{n\ge 1}\1_E(T^{p_n}x)=0 \quad \text{ a.e. } \ x.
\]
Hence $ P^*\1_E=0$ a.e., and the estimate follows.

Now assume that $0<\mu(E)\le 1$.  Since $0\le P^*\1_E\le 1$, by the layer-cake formula
and Proposition~\ref{prop:trojan-prime-estimate} we have 
\[
\begin{aligned}
        \int_XP^*\1_E \dd\mu &= \int_0^1\mu\left(\{x\in X: P^*\1_E(x)>\alpha\} \right)\dd\alpha\\
        &\le \mu(E)+ C^\prime\mu(E)\int_{\mu(E)}^1 \alpha^{-1}\log^2(\e/\alpha) \dd\alpha\\
        &=\mu(E)+\frac{C^\prime\mu(E)}{3}\left[\log^3\left(\frac{\mathrm e}{\mu(E)}\right)-1\right]                                                             \\
        &\le C \mu(E)\left(1+L_1(\mu(E)^{-1})^3\right)
\end{aligned}
\]
for some constants $C, C^\prime>0$, independent of both the underlying system and the measurable set $E$. 
\end{proof}

\begin{lem}\label{lem:polylog-cutoff}
Fix $A>0$ and $q\in\N$.  For $R\ge 2$, put
\[
        B_R=1+L_1(R)^A  \ \text{ and }\  m_R=\left\lceil \exp(B_R)\right\rceil.
\]
Then
\[
        \Lambda_q(m_R)\ge B_R
\]
and there is a constant $C_{A,q}>0$ such that 
\[
        1+L_{q+1}(m_R)
        \le
        C_{A,q}\bigl(1+L_{q+1}(R)\bigr).
\]
\end{lem}

\begin{proof}
Since $m_R\ge \exp(B_R)$, we have $L_1(m_R)\ge B_R$, and hence $\Lambda_q(m_R)\ge B_R$.
On the other hand, $m_R\le \exp(B_R)+1\le 2\exp(B_R)$, 
and hence
\[
        L_1(m_R)\le B_R+C= 1+L_1(R)^A+C\le C_A\bigl(1+L_1(R)\bigr)^A
\]
for some constant $C_A>0$.

Assume first that $L_1(R)\ge\mathrm e$  and put $v=L_1(R)-\mathrm e$. Then $v\ge 0$ and
\[
        1+L_1(R)= 1+\mathrm e+v\le (\mathrm e+1)(1+v).
\]
Applying Lemma~\ref{lem:Psi-estimates}\ref{lem:Psi-estimates:3} with $s=q$, $\alpha=A$ and $t=v$, we obtain that 
\[
\begin{aligned}
        L_{q+1}(m_R)&=L_q\bigl(L_1(m_R)\bigr)\\
        &\le L_q\bigl(C_A(1+L_1(R))^A\bigr)\\
        &\le L_q\bigl(C_A(\mathrm e+1)^A(1+v)^A\bigr)\\
        &\le C_{A,q}\bigl(1+L_q(\mathrm e+v)\bigr)\\
        &= C_{A,q}\bigl(1+L_q(L_1(R))\bigr)\\
        &= C_{A,q}\bigl(1+L_{q+1}(R)\bigr)
\end{aligned}
\]
for some constant $C_{A,q}>0$.

If $L_1(R)<\mathrm e$, we have
\[
        L_1(m_R) \le C_A(1+\mathrm e)^A,
\]
and consequently
\[
        L_{q+1}(m_R)= L_q\bigl(L_1(m_R)\bigr) \le C_{A,q}
\]
for some constant $C_{A,q}>0$. 
Since $1+L_{q+1}(R)\ge 1$, the same estimate follows after enlarging
$C_{A,q}$. This completes the proof.
\end{proof}

\begin{prop}\label{prop:polylog-RME}
Let $\mathbf a=(a_n)_{n\ge 1}$ be a sequence of nonnegative integers.
Suppose that there are $A>0$ and $C_A>0$ such that for every measure-preserving system  $(X,\B_X,\mu, T)$  and every measurable set $E\in\B_X$,
\[
        \int_X A^*_{\mathbf a}\1_E \dd\mu =\int_X \sup_{N\ge 1}\frac1N\sum_{n=1}^N \1_E(T^{a_n}x) \dd\mu\le C_A\mu(E) \left(1+L_1\left(\frac{1}{\mu(E)}\right)^A\right),
\]
where the right-hand side is interpreted as $0$ when $\mu(E)=0$.
Then  $\mathbf a$ satisfies $\RME_q$ uniformly over all measure-preserving systems, i.e., for every $q\in\N$, there is $C_{A,q}>0$, independent of the system, such that
\[
       \int_X M_{\mathbf a,q}\1_E\dd \mu= \int_X \sup_{N\ge 1} \frac1{N\Lambda_q(N)}\sum_{n=1}^N\1_E(T^{a_n}x) \dd\mu(x) \le C_{A,q}\mu(E) \left(1+L_{q+1}\left(\frac{1}{\mu(E)}\right)\right),
\]
where the right-hand side is interpreted as $0$ when $\mu(E)=0$. 
\end{prop}

\begin{proof}
Fix $q\in\N$, and write
\[
        G_q(x)=\sup_{N\ge 1} \frac1{N\Lambda_q(N)} \sum_{n=1}^N\1_E(T^{a_n}x).
\]
If $\mu(E)=0$, then each set $T^{-a_n}E$ has measure zero, so $G_q=0$ a.e.  If $\mu(E)>1/2$, then $0\le G_q\le1$, and hence the desired estimate follows after enlarging the constant. 

We now assume that $ 0<\mu(E)\le 1/2$. Set $R=\mu(E)^{-1}$, and choose $B_R=1+L_1(R)^A $ and $m_R=\left\lceil \exp(B_R)\right\rceil$ as in Lemma~\ref{lem:polylog-cutoff}. Let
\[
        G_q^{<m_R}=\sup_{1\le N<m_R}\frac1{N\Lambda_q(N)}\sum_{n=1}^N\1_E\circ T^{a_n}
\]
and
\[
        G_q^{\ge m_R}=\sup_{N\ge m_R}\frac1{N\Lambda_q(N)}\sum_{n=1}^N\1_E\circ T^{a_n}.
\]
Then $G_q\le G_q^{<m_R}+G_q^{\ge m_R}$.
For $1\le n\le N$, the monotonicity of the functions $L_j$ implies that $t\mapsto t\Lambda_q(t)$ is increasing.  Hence for any nonnegative numbers $b_n$ and $N<m_R$, we have 
\[
        \frac1{N\Lambda_q(N)}\sum_{n=1}^N b_n \le \sum_{n=1}^N \frac{b_n}{n\Lambda_q(n)} \le \sum_{n<m_R} \frac{b_n}{n\Lambda_q(n)}.
\]
Applying this with $b_n=\1_E\circ T^{a_n}$  and using measure preservation, we obtain that
\[
\begin{aligned}
        \int_XG_q^{<m_R}\dd\mu&\le\sum_{n<m_R}\frac1{n\Lambda_q(n)} \int_X\1_E(T^{a_n}x)\,\dd\mu(x)\\
        &=\mu(E)\sum_{n<m_R}\frac1{n\Lambda_q(n)}\\
        &\le C_q\mu(E)\bigl(1+L_{q+1}(m_R)\bigr)\\
        &\le C_{A,q}\mu(E) \bigl(1+L_{q+1}(\mu(E)^{-1})\bigr) 
\end{aligned}
\]
for some constants $C_q, C_{A,q}$, where the last two inequalities follow from Lemma~\ref{lem:log-calc} and Lemma~\ref{lem:polylog-cutoff} respectively.

For $N\ge m_R$, monotonicity of $\Lambda_q$ yields
\[
        \frac1{N\Lambda_q(N)}\sum_{n=1}^N\1_E(T^{a_n}x) \le \frac1{\Lambda_q(m_R)} \frac1N\sum_{n=1}^N\1_E(T^{a_n}x)\le \frac{A^*_{\mathbf a}\1_E(x)}{\Lambda_q(m_R)}.
\]
Thus
\[
        G_q^{\ge m_R} \le \frac{A^*_{\mathbf a}\1_E}{\Lambda_q(m_R)}.
\]
Using the hypothesis and Lemma~\ref{lem:polylog-cutoff}, we obtain that 
\[
\begin{aligned}
        \int_X G_q^{\ge m_R}\dd\mu  \le\frac{C_A\mu(E)\bigl(1+L_1(R)^A\bigr)} {\Lambda_q(m_R)} =\frac{C_A\mu(E) B_R}{\Lambda_q(m_R)}\le C_A\mu(E).
\end{aligned}
\]
Combining the two parts we get 
\[
        \int_XG_q\,\dd\mu \le C_{A,q}\mu(E) \bigl(1+L_{q+1}(\mu(E)^{-1})\bigr) 
\]
for some constant  $C_{A,q}>0$, which is precisely $\RME_q$.
\end{proof}

\begin{thm}\label{thm:prime-RME}
For every $q\in\N$, the prime sequence satisfies $\RME_q$ uniformly over all measure-preserving systems.
\end{thm}

\begin{proof}
It follows directly from Proposition \ref{prop:polylog-RME} with $\mathbf a=\mathbf p$.  Proposition~\ref{prop:trojan-strong-prime}
provides its hypothesis with $A=3$.
\end{proof}

\begin{prop}\label{prop:prime-not-weak}
The prime maximal operator $P^*$ is not uniformly of weak type $(1,1)$ over measure-preserving systems.
\end{prop}

\begin{proof}
Define
\[
    \mathcal M_{\mathbf p}g =\sup_{N\ge 1}\left|\frac1N\sum_{n=1}^N g\circ T^{p_n}\right|.
\]
By the triangle inequality, we have 
\[
    \mathcal M_{\mathbf p}g \le \sup_{N\ge 1} \frac1N\sum_{n=1}^N |g|\circ T^{p_n} = P^*g.
\]
Therefore, a uniform weak type $(1,1)$ estimate for $P^*$ would imply the same estimate, with the same constant,  for $\mathcal M_{\mathbf p}$. 
This contradicts \cite[Theorem~1.1]{LaVictoire}, which declares that for every $C>0$ there exist a measure-preserving system and a $g\in L^1$ such that
\[
    \|\mathcal M_{\mathbf p}g\|_{L^{1,\infty}}> C\|g\|_{L^1}.
\]
\end{proof}

\begin{coro}\label{cor:prime-RME-strict}
For every fixed $q\in\N$, the uniform $\RME_q$  is strictly weaker than the uniform weak type $(1,1)$.
\end{coro}

\begin{proof}
By Proposition~\ref{prop:weak-RME}, the uniform weak type $(1,1)$ implies the uniform $\RME_q$. 
On the other hand, Theorem~\ref{thm:prime-RME} shows that the prime sequence satisfies the uniform $\RME_q$. However, Proposition~\ref{prop:prime-not-weak} shows that its maximal operator does not satisfy any uniform weak type $(1,1)$ estimate.
\end{proof}

Now we are ready to prove \ref{thm:regular-endpoints}\ref{thm:regular-endpoints:2}.

\begin{proof}[Proof of \ref{thm:regular-endpoints}\ref{thm:regular-endpoints:2}]
Combining Theorem~\ref{thm:prime-RME} and Theorem~\ref{thm:RME-convergence-zero}, we obtain  almost everywhere
convergence to $0$, and local stability under the LDP for every $f\in L\log_{q+1}L$ and every $q\in\N$.  

For the lower endpoint bound, by \cite[Theorem~3]{RosserSchoenfeld1962} we have
\[
        p_n< n\left(\log n+\log\log n-\frac12\right), \quad n\ge 20.
\]
Note that
\[
        \log n+\log\log n-\frac12<n, \quad n\ge 20.
\]
Consequently
\[
        p_n<n^2,  \quad n\ge 20.
\]
Now set $ A=\max\{1,\,\max_{1\le n<20}\frac{p_n}{n^2} \}$. Then
\[
        p_n\le A n^2, \quad n\ge 1.
\]
That is, the sequence of primes has polynomial growth.  By Theorem~\ref{thm:poly-lower} or Corollary \ref{cor:abstract-sharp}, we immediately conclude that $ L\log_{q+1}L$ is the sharp Orlicz endpoint in the local observation sense.
\end{proof}

To end this subsection, we give a positive endpoint result for the ordinary prime normalization.  
Trojan proved in \cite[Theorem 7.4]{Trojan} that 
\[
    P^*: L(\log L)^2(\log\log L)\to  L^{1,\infty}
\]
is bounded and consequently obtained the almost-everywhere convergence of $\mathcal P_R f$ as $R\to \infty$, or equivalently of $A_N^{\mathbf p}f$ as $N\to \infty$,  for every $f\in L(\log L)^2(\log\log L)$. However, this estimate does not by itself provide the $L^1$-integrable majorant required by Theorem~\ref{thm:local-stability-principle}.
In the following proposition, we prove a strong $L^1$ estimate for the prime maximal operator on the smaller space $L(\log L)^3$, which in turn yields locally stable convergence on that space. 

Define
\[
        \Theta_3(t)= \int_0^t\left(1+L_1(u)^3\right) \dd u.
\]
It is clear that $\Theta_3(t)\asymp t\bigl(1+L_1(t)^3\bigr)$. Hence the Orlicz space $L^{\Theta_3}$ coincides with the usual $L(\log L)^3$ space , with equivalent Luxemburg norms.

\begin{prop}\label{prop:prime-q0-positive}
Let $\mathbf p=(p_n)_{n\ge 1}$ be the sequence of primes. 
There exists a constant $C>0$ such that for every measure-preserving system $(X,\B_X,\mu,T)$ and every $f\in L^{\Theta_3}$, we have
\[
        \|P^*f\|_{L^1} \le C\|f\|_{\Theta_3}.
\]
More precisely,
\[
\begin{aligned}
        \int_X P^*f \dd\mu
        &\le C\int_0^\infty d_f(t) \left( 1+L_1\left(\frac1{d_f(t)}\right)^3\right) \dd t \\
        &\le C\left(1+ \int_X |f|\bigl(1+L_1(\mathrm e+|f|)^3\bigr) \dd\mu \right),
\end{aligned}
\]
where $d_f(t)=\mu(\{|f|>t\})$, and the integrand is interpreted as zero when $d_f(t)=0$.

Consequently, if the underlying metric probability space has the LDP, then the ordinary prime averages
\[
        A_N^{\mathbf p}f =\frac1N\sum_{n=1}^N f\circ T^{p_n} 
\]
converge locally stably for every $f\in L^{\Theta_3}$.
\end{prop}

\begin{proof}
Using the same argument as in Proposition~\ref{prop:RME-lift}, the layer-cake formula, positivity, Tonelli's theorem, and Proposition~\ref{prop:trojan-strong-prime} yield 
\[
        \int_X P^*f \dd\mu \le C\int_0^\infty d_f(t) \left(1+L_1\left(\frac1{d_f(t)}\right)^3\right)\dd t 
\]
for some constant $C>0$.
The same rearrangement argument as in Lemma~\ref{lem:dist-bound}, with
$L_1^3$ in place of $L_s$, yields
\[
        \int_0^\infty  d_f(t) \left( 1+L_1\left(\frac1{d_f(t)}\right)^3 \right) \dd t 
        \le C\left(1+\int_X |f|\bigl(1+L_1(\mathrm e+|f|)^3\bigr)\,\dd\mu \right).
\]
Indeed, the proof there is unchanged, except that on the first part $E_1$ of the rearrangement decomposition we use
\[
        L_1(t^2)^3\le C\bigl(1+L_1(\mathrm e+t)^3\bigr).
\]

Let $\lambda>\|f\|_{\Theta_3}$.  Since
\[
        \int_X\Theta_3\left(\frac{|f|}{\lambda}\right) \dd\mu  \le 1
\]
and $\Theta_3(t)\asymp t(1+L_1(t)^3)$, applying the preceding estimate to $f/\lambda$, it follows that 
\[
        \|P^*(f/\lambda)\|_{L^1}\le C.
\]
Note that $ P^*f=\lambda P^*\left(\frac{f}{\lambda}\right)$, and hence  $ \|P^*f\|_{L^1} \le C\lambda$.
Letting $\lambda\downarrow\|f\|_{\Theta_3}$ we have
\[
        \|P^*f\|_{L^1}\le C\|f\|_{\Theta_3}.
\]

It remains to prove the locally stable convergence.  Since
\[
        L^{\Theta_3}=L(\log L)^3\subset L(\log L)^2(\log\log L),
\]
by \cite[Theorem 7.4]{Trojan}, the prime averages $A^{\mathbf p}_Nf$ converge a.e. for every $f\in L(\log L)^3$.
Moreover, if $f^*$ denotes the resulting pointwise limit, then 
\[
        |f^*| \le P^*f  \ \text{ and } \sup_{N\ge 1} \left|A_N^{\mathbf p} f\right|\le P^*f\in L^1(\mu).
\]
Thus, by Theorem~\ref{thm:local-stability-principle}, $A_N^{\mathbf p}f$ converges locally stably to $f^*$.
\end{proof}

From the proof of \ref{thm:regular-endpoints}\ref{thm:regular-endpoints:2}, the prime sequence has polynomial growth. By Theorem~\ref{thm:poly-lower} with $q=0$, we get an $L^\Phi$ local $\infty$-sweeping out counterexample on the Lebesgue interval whenever $\Phi(t)=o(tL_1(t)) \  (t\to \infty)$.
In contrast, Proposition~\ref{prop:prime-q0-positive} provides locally stable convergence only on the sufficient class $L(\log L)^3$.
Therefore we pose the following natural endpoint question.

\begin{ques}\label{ques:prime-q0}
Do the ordinary prime averages converge locally stably for every $f\in L\log L$ on every measure-preserving system whose underlying metric probability space has the LDP?
\end{ques}

\section{\texorpdfstring{The sharp $L\log_{q}L$ Orlicz endpoint}{The sharp iterated-logarithmic endpoint at level q}} \label{sec:Llogq}

This section establishes the sharp $L\log_{q}L$ Orlicz endpoint for local observation problems.
The positive theorem in this section is valid for every time sequence and follows from a dyadic decomposition.  
The lower examples are built in two steps: first we produce finite orbit-name witnesses, and then place those finite stages on successively finer interval levels.  The three finite-stage assumptions needed for the second step are stated below and used only in the assembly argument.

\subsection{The universal positive endpoint}

In this subsection we prove \ref{thm:universal-sparse}\ref{thm:universal-sparse:1}. The result holds uniformly over all
time sequences. The main ingredient is a strong $L^1$ estimate for the positive maximal operator associated with the normalized averages. 
We first present two elementary estimates for the dyadic weights and then prove a general truncation lemma.

For $q\in\N$ and integers $t\ge 2$, put
\[
        \omega_q(t)=\frac{1}{t\Lambda_{q-1}(t)}.
\]

\begin{lem}\label{lem:iterated-sum}
For every $q\in\N$, there are constants $c_q,C_q>0$ such that for every integer $T\ge 3$,
\[
        c_q\bigl(1+L_q(T)\bigr) \le \sum_{2\le t\le T}\omega_q(t) \le C_q\bigl(1+L_q(T)\bigr), 
\]
and for every integer $t\ge 2$,
\[
        c_q\omega_q(t) \le\frac{1}{\Lambda_q(2^t)} \le C_q\omega_q(t).
\]
\end{lem}

\begin{proof}
The first estimate follows from Lemma~\ref{lem:log-calc}\ref{lem:log-calc:2}, applied with $q-1$, after changing the constants to account for the finitely many initial
terms.

For fixed $q$,  we have 
\[
        L_1(2^t)\asymp t  \ \text{ and }\  L_{r+1}(2^t)\asymp L_r(t), \  1\le r\le q-1.
\]
Hence
\[
        \Lambda_q(2^t) \asymp_q t\Lambda_{q-1}(t),
\]
which proves the second estimate.
\end{proof}

The following lemma is the main maximal estimate used in this subsection.

\begin{lem} \label{lem:positive-weighted-dyadic}
Let $w_t>0$, $t\ge 2$, satisfy $w_t^{-1}\ge t$,  and put
\[
        W_w(u)=\sum_{2\le t<u}w_t, \quad u\ge 0.
\]
Let $(X,\B_X,\mu, T)$ be measure-preserving system, $(a_n)$ be any sequence of nonnegative integers and define
\[
        D_tg =\frac{1}{2^t}\sum_{n=1}^{2^t}g\circ T^{a_n}.
\]
Then for every measurable $f$, we have 
\[
        \left\|\sup_{t\ge 2} w_tD_t|f| \right\|_1 \le  1+\int_X |f| W_w(|f|)\dd\mu.
\]
\end{lem}

\begin{proof}
For each $t\ge 2$, decompose 
\[
        |f|=d_t+r_t,
\]
where
\[
        d_t=|f|\1_{\{|f|\le w_t^{-1}\}}  \text{ and } r_t= |f|\1_{\{|f|>w_t^{-1}\}}.
\]
Since $0\le d_t\le w_t^{-1}$, positivity yields
\[
        w_t D_t d_t\le 1.
\]
Since all terms are nonnegative, by Tonelli's theorem, and invariance of $\mu$, we have 
\[
 \left\|\sup_{t\ge 2}w_tD_tr_t\right\|_1\le \sum_{t\ge 2}w_t \int_X D_t r_t \dd\mu  = \int_X |f(x)|\sum_{\{t\ge 2: w_t^{-1}<|f(x)|\} } w_t \dd\mu(x).
\]
Since $w_t^{-1}\ge t$, the condition $w_t^{-1}<|f(x)|$ implies $t<|f(x)|$.  Hence
\[
        \sum_{\{t\ge 2: w_t^{-1}<|f(x)|\}}w_t \le  W_w(|f(x)|).
\]
Combining the estimates for $d_t$ and $r_t$ proves the assertion.
\end{proof}

For $q\in\N$ and a measurable function $f$, define
\[
        \mathcal M_{\mathbf a,q}f= \sup_{N\ge 2} \frac{1}{N\Lambda_q(N)} \sum_{n=1}^{N}|f|\circ T^{a_n}.
\]

\begin{thm}\label{thm:universal-positive}
Let $q\in\N$. There is a constant $C_q>0$,  depending only on $q$,  such that for every sequence of nonnegative
integers $(a_n)$, every measure-preserving system $(X,\B_X,\mu,T)$ and every $f\in L\log_{q}L$, we have  
\[
        \left\|\sup_{t\ge 2} \omega_q(t) \frac{1}{2^t}\sum_{n=1}^{2^t}|f|(T^{a_n}\cdot)\right\|_1 \le C_q\left( 1+\int_X|f|\bigl(1+L_q(|f|)\bigr)\dd\mu \right).
\]
Moreover,
\begin{equation*}\label{eq:Luxemburg-estimate}
        \|\mathcal M_{\mathbf a,q}f\|_1 \le C_q\|f\|_{\Psi_q}.
\end{equation*}
\end{thm}

\begin{proof}
Throughout the proof, $C_q$ denotes a positive constant depending only on $q$, whose value may change from line to line.

By Lemma~\ref{lem:iterated-sum}, after increasing $C_q$ if necessary,
\[
        W_{\omega_q}(u)=\sum_{2\le t<u}\omega_q(t)\le C_q\bigl(1+L_q(u)\bigr),\quad u\ge0.
\]
Let
\[
        D_tg =\frac{1}{2^t}\sum_{n=1}^{2^t}g\circ T^{a_n}.
\]
Since
\[
        \omega_q(t)^{-1}=t\Lambda_{q-1}(t) \ge t,
\]
by Lemma~\ref{lem:positive-weighted-dyadic}, we have 
\[
        \left\|\sup_{t\ge 2}\omega_q(t)D_t|f|\right\|_1\le C_q\left(1+\int_X|f|\bigl(1+L_q(|f|)\bigr)\dd\mu \right).
\]

We next pass from dyadic to arbitrary values of $N$.   Let $2^{t-1}\le N\le 2^t$ with $t\ge 2$. 
 Since $\Lambda_q$ is increasing,
\[
        N\Lambda_q(N)\ge 2^{t-1}\Lambda_q(2^{t-1}).
\]
Moreover, there is a constant $K_q\ge 1$, depending only on $q$, such that
\[
        \Lambda_q(2^t) \le  K_q\Lambda_q(2^{t-1}), \quad t\ge 2.
\]
It follows that
\[
        N\Lambda_q(N) \ge \frac{1}{2K_q}\,2^t\Lambda_q(2^t).
\]
This together with Lemma~\ref{lem:iterated-sum} yields
\[
\begin{aligned}
 \frac{1}{N\Lambda_q(N)} \sum_{n=1}^N |f|\circ T^{a_n}  \le 2K_q \frac{1}{2^t\Lambda_q(2^t)} \sum_{n=1}^{2^t}|f|\circ T^{a_n}  \le C_q \omega_q(t)D_t|f|.
\end{aligned}
\]
Consequently, 
\[
        \mathcal M_{\mathbf a,q}f \le C_q\sup_{t\ge 2}\omega_q(t)D_t|f|.
\]
Combining this estimate with the dyadic maximal estimate above, we have 
\begin{equation*}\label{eq:nonhomogeneous-maximal}
        \|\mathcal M_{\mathbf a,q}f\|_1 \le C_q\left(1+\int_X|f|\bigl(1+L_q(|f|)\bigr)\dd\mu \right).
\end{equation*}

It remains to prove the homogeneous Luxemburg estimate.  Let $\lambda>0$ satisfy
\[
        \int_X \Psi_q\left(\frac{|f|}{\lambda}\right) \dd\mu \le 1.
\]
By \eqref{eq:Psi-comparison}, there is a constant $A_q>0$, depending only on $q$, such that
\[
        \int_X \frac{|f|}{\lambda}\left( 1+L_q\left(\frac{|f|}{\lambda}\right)\right)\dd\mu  \le A_q.
\]
Applying the previous maimal estimate  to $f/\lambda$, we obtain
\[
        \left\|\mathcal M_{\mathbf a,q}\left(\frac{f}{\lambda}\right) \right\|_1\le C_q.
\]
Since $\mathcal M_{\mathbf a,q}$ is positively homogeneous, it follows that
\[
        \|\mathcal M_{\mathbf a,q}f\|_1 \le C_q\lambda.
\]
Taking the infimum over all admissible $\lambda$ proves the desired estimates.
\end{proof}

As a consequence, we can now prove \ref{thm:universal-sparse}\ref{thm:universal-sparse:1}.

\begin{proof}[Proof of \ref{thm:universal-sparse}\ref{thm:universal-sparse:1}.]
For $N\ge 2$, write
\[
        A_N^{\mathbf a, q}f= \frac{1}{N\Lambda_q(N)} \sum_{n=1}^{N}f\circ T^{a_n}.
\]
When $f$ is bounded measurable,  we have 
\[
        \sup_{x\in X}|A_N^{\mathbf a, q}f(x)| \le \frac{\|f\|_\infty}{\Lambda_q(N)} \to 0.
\]
For general $f\in L\log_{q}L$, put $f_K=f\1_{\{|f|\le K\}}$. Observe that for every measurable $h$,
\[
        \sup_{N\ge 2}|A_N^{\mathbf a, q}h| \le \mathcal M_{\mathbf a,q}h.
\]
Hence, by \eqref{eq:Luxemburg-estimate} and  Lemma~\ref{lem:lux-tail}, we have 
\[
 \left\| \sup_{N\ge 2}|A_N^{\mathbf a, q}(f-f_K)| \right\|_1 \le \|\mathcal M_{\mathbf a,q}(f-f_K)\|_1  \le C_q\|f-f_K\|_{\Psi_q} \to 0.
\]
Since $f_K$ is bounded, $A_N^{\mathbf a, q}f_K\to 0$ uniformly. Hence 
\[
\begin{aligned}
 \limsup_{N\to \infty}|A_N^{\mathbf a, q}f|
 &\le \limsup_{N\to \infty}|A_N^{\mathbf a, q}(f-f_K)| +\limsup_{N\to \infty}|A_N^{\mathbf a, q}f_K|      \\
 &\le \sup_{N\ge 2}|A_N^{\mathbf a, q}(f-f_K)|
\end{aligned}
\]
pointwise.  Thus, for every $\varepsilon>0$, by Chebyshev's inequality, 
\[
\begin{aligned}
 \mu\left\{\limsup_{N\to \infty}|A_N^{\mathbf a, q}f|>\varepsilon\right\}
 &\le \mu\left\{ \sup_{N\ge 2}|A_N^{\mathbf a, q}(f-f_K)|>\varepsilon\right\}   \\
 &\le \frac{1}{\varepsilon} \left\| \sup_{N\ge 2}|A_N^{\mathbf a, q}(f-f_K)| \right\|_1   \\
 &\le \frac{C_q}{\varepsilon}\|f-f_K\|_{\Psi_q}.
\end{aligned}
\]
Letting $K\to \infty$, we obtain
\[
        \mu\left\{  \limsup_{N\to \infty}|A_N^{\mathbf a, q}f|>\varepsilon\right\} =0
\]
for every $\varepsilon>0$.  This implies that
\[
        A_N^{\mathbf a, q}f(x)\to 0  \quad \text{a.e.}~x\in X.
\]

Finally, \eqref{eq:Luxemburg-estimate} shows that $\mathcal M_{\mathbf a,q}f\in L^1(\mu)$  and
\[
        |A_N^{\mathbf a, q}f| \le \mathcal M_{\mathbf a,q}f, \quad  N\ge 2
\]
for every $f\in L\log_{q}L$.
Thus Theorem~\ref{thm:local-stability-principle} implies that the convergence is locally stable whenever the underlying metric probability space has the LDP.
\end{proof}

\subsection{Interval realization and assembly of finite stages} \label{subsec:finite-stages}

We begin with the finite-stage assumptions used in the interval construction.  
At each stage,  the maximum of finitely many normalized averages is large on a set of uniformly positive measure. 
Each stage is then realized by finitely many disjoint interval levels, and the levels arising from all stages are subsequently embedded into a single ergodic transformation.
The total mass of the labelled levels yields the required Orlicz estimate, while a sufficiently fine partition of the available remaining set yields the local lower bound.

\begin{defn}\label{def:finite-stage-hypotheses}
Fix $q\in\N$ and  a sequence  $\mathbf a=(a_n)$ of nonnegative integers.  A \emph{finite stage} is a tuple
\[
        \mathfrak s= (\Omega,\B_\Omega,\nu,S,E,H,A,\mathcal K),
\]
where $\Omega$ is a metric space, $\B_\Omega$ is its Borel $\sigma$-algebra, $(\Omega,\B_\Omega,\nu,S)$ is an invertible  measure-preserving system, $E\in\B_\Omega$ with $\nu(E)>0$, $H,A>0$, and $\mathcal K\subset\N$ is finite and non-empty.

Set
\[
        \mathcal G(\mathfrak s)=\left\{\omega\in\Omega:\max_{K\in\mathcal K} \frac{H}{K\Lambda_q(K)}\sum_{n=1}^{K}\mathbf 1_E(S^{a_n}\omega)>A\right\}.
\]
Let $\Phi$ be a  Young function.  A sequence
\[
        \mathfrak s_j=(\Omega_j,\B_{\Omega_j},\nu_j,S_j,  E_j,H_j,A_j,\mathcal K_j), \quad j\in\N,
\]
of finite stages is called \emph{$(\mathbf a,q,\Phi)$-admissible} if
\begin{equation}\label{eq:finite-stage-escape}
        A_j\to \infty, \quad  \min\mathcal K_j\to \infty,
\end{equation}
and there is a constant $0<c_0\le1$ such that
\begin{equation}\label{eq:finite-stage-success}
        \nu_j\bigl(\mathcal G(\mathfrak s_j)\bigr)\ge c_0, \quad \forall \ j \in\N,
\end{equation}
and
\begin{equation}\label{eq:finite-stage-summability}
        \sum_{j=1}^{\infty} 2^{-j}\nu_j(E_j) \Phi(2^j H_j)<\infty.
\end{equation}
\end{defn}

The space $\Omega_j$ need not be finite, and the term ``finite stage'' refers to the finite set $\mathcal K_j$ and the finite orbit information used at the $j$-th stage.

The next lemma converts one finite stage into finitely many interval levels. All intervals in its proof are assumed to be half-open.

\begin{lem}\label{lem:finite-stage-interval}
Let $q\in\N$ and $\mathbf a=(a_n)$ be   a  sequence of nonnegative integers.  
Let $ \mathfrak s= (\Omega,\B_\Omega,\nu,S,E,H,A,\mathcal K)$ be a finite stage.  Let $0<c\le 1$, and assume that there are pairwise
disjoint Borel sets $G_K$, $K\in\mathcal K$, such that
\begin{equation}\label{eq:finite-stage-GK}
        G_K \subset\left\{ \omega\in \Omega : \frac{H}{K\Lambda_q(K)}  \sum_{n=1}^{K}\mathbf 1_E(S^{a_n}\omega)>A \right\}
\end{equation}
and
\begin{equation}\label{eq:finite-stage-GK-mass}
        \sum_{K\in\mathcal K}\nu(G_K)\ge c.
\end{equation}
Let $R\subset[0,1]$ be a finite union of half-open intervals with
$\lambda(R)>1/2$, and let
\[
        0<\eta<\frac{\lambda(R)}4, \quad  0<\theta\le\frac{c\eta}{64},  \quad \rho>0.
\]
Then there are an integer $h\ge 1$, finite sets $\mathcal P$ and $\I$, a map $\kappa:\I\to\mathcal K$, numbers
$\delta_P>0$, $P\in\mathcal P$, pairwise disjoint finite unions of
half-open intervals
\[
        I_{P,u}\subset R, \quad P\in\mathcal P,\quad 0\le u<h,
\]
labelled numbers $\varepsilon_{P,u}\in\{0,1\}$, and finite unions of half-open intervals $V_i,C_i\subset R$, $i\in\I$, with the following
properties:
\begin{enumerate}[label=\textup{(\alph*)},leftmargin=*]
\item 
\begin{equation}\label{eq:finite-stage-level-widths}
        \lambda(I_{P,u})=\delta_P, \quad P\in \mathcal P,\quad 0\le u<h,
\end{equation}
and
\begin{equation}\label{eq:finite-stage-total-level-mass}
        h\sum_{P\in\mathcal P}\delta_P=\eta.
\end{equation}

\item  
\begin{equation}\label{eq:finite-stage-labelled-mass}
        \sum_{P\in\mathcal P}\sum_{u=0}^{h-1} \varepsilon_{P,u}\lambda(I_{P,u})= \eta\nu(E).
\end{equation}

\item 
\[
        R=\bigsqcup_{i\in\I}V_i.
\]
Every component interval of every $V_i$ has length at most $\rho$, and
for every component interval $J$ of $V_i$,
\begin{equation}\label{eq:finite-stage-density}
        C_i\subset V_i,  \quad \lambda(C_i\cap J)=\theta\lambda(J).
\end{equation}

\item Suppose that $T$ is an invertible Lebesgue measure-preserving transformation of $[0,1]$ satisfying
\begin{equation}\label{eq:finite-stage-tower-relations}
        T(I_{P,u})=I_{P,u+1}  \quad\text{modulo null sets}, \quad P\in\mathcal P,\ 0\le u<h-1.
\end{equation}
Define
\begin{equation}\label{eq:finite-stage-F}
        F =H\sum_{P\in\mathcal P}\sum_{\substack{0\le u<h,\,\varepsilon_{P,u}=1}}\mathbf 1_{I_{P,u}}.
\end{equation}
Then  for every $i\in\I$,
\begin{equation}\label{eq:finite-stage-lower-on-Ci}
        \frac{1}{\kappa(i)\Lambda_q(\kappa(i))} \sum_{n=1}^{\kappa(i)}F(T^{a_n}y)>A
\end{equation}
holds for a.e. $y\in C_i$.
\end{enumerate}
\end{lem}

\begin{proof}
Define 
\[
        K_*=\max\mathcal K, \quad L=\max_{1\le n\le K_*}a_n, \quad M=L+1  \quad \text{and } \quad h=M+L+1.
\]
Let $\mathcal A$ be the finite Boolean algebra  generated by $S^{-u}E$, $0\le u<h$, and by $S^{-t}G_K$, $0\le t<M$, $K\in\mathcal K$.  
Denote by $\mathcal P$ the collection of atoms of $\mathcal A$ having positive $\nu$-measure.  
These atoms cover $\Omega$ up to a null set, and hence
\begin{equation*}\label{eq:finite-stage-atoms-sum}
        \sum_{P\in\mathcal P}\nu(P)=1.
\end{equation*}
Let  $\delta=\frac{\eta}{h}$. For each $P\in\mathcal P$, set  $\delta_P=\delta\nu(P)$.  For $0\le u<h$, define
\begin{equation*}\label{eq:finite-stage-labels}
        \varepsilon_{P,u}= \begin{cases}
        1, & P\subset S^{-u}E,\\
        0, & P\subset \Omega\setminus S^{-u}E.
        \end{cases}
\end{equation*}
For $K\in\mathcal K$, put
\begin{equation*}\label{eq:finite-stage-mK}
\begin{aligned}
        m_K  =\sum_{t=0}^{M-1}\sum_{\substack{P\in\mathcal P,\,P\subset S^{-t}G_K}} \delta_P =M\delta \nu(G_K).
\end{aligned}
\end{equation*}
Indeed, $S^{-t}G_K\in\mathcal A$, so it is a union of atoms of $\mathcal A$, and omitting the atoms of zero measure does not change its
measure. In particular, $m_K>0$ if and only if $\nu(G_K)>0$.
Since $M/h=1/2$, condition \eqref{eq:finite-stage-GK-mass} yields 
\begin{equation*}\label{eq:finite-stage-total-mK}
        \sum_{K\in\mathcal K}m_K = M\delta\sum_{K\in\mathcal K}\nu(G_K) \ge \frac{c\eta}{2}.
\end{equation*}
Define 
\[
        \I = \{(K,\ell):K\in\mathcal K,\ m_K>0,\ \ell\in\{0,1\}\}  \ \text{ and }\ \kappa(K,\ell)=K.
\]
The two copies are used only to ensure that $\#\I\ge 2$,  as required by Lemma~\ref{lem:fine-filling}.
For $i=(K,\ell)$, set $b_i=m_K/2$, and define
\[
        \alpha =  \frac{\theta\lambda(R)}{\sum_{i\in\I}b_i}, \quad  c_i=\alpha b_i.
\]
By the choice of  $\theta$, it follows that    
\[
        0<\alpha \le\frac{2\theta\lambda(R)}{c\eta}  \le\frac1{32}.
\]
Consequently,
\begin{equation*}\label{eq:finite-stage-ci}
        0<c_i<b_i  \  \text{ and } \   \sum_{i\in\I}c_i=\theta\lambda(R).
\end{equation*}

Applying Lemma~\ref{lem:fine-filling} to $R$, with $\ell_i=u_i= {c_i}/{\theta}$, $i\in\I$, and mesh $\rho$, 
we obtain a partition  
\[
R=\bigsqcup_{i\in\I}\,V_i
\]
such that every $V_i$ is a finite union of half-open intervals, $ \lambda(V_i)= {c_i}/{\theta}$, 
and every component interval of every $V_i$ has length at most $\rho$. In each component interval $J$ of $V_i$, take its left subinterval of
relative length $\theta$, and let $C_i$ be the union of the selected subintervals.  Then
\begin{equation*} 
        \lambda(C_i)=c_i   \  \text{ and } \     \lambda(C_i\cap J)=\theta\lambda(J),
\end{equation*}
which proves \eqref{eq:finite-stage-density}.

For $K\in\mathcal K$ with $m_K>0$, let $\Gamma_K=\{(t,P):0\le t<M,\ P\in\mathcal P,\ P\subset S^{-t}G_K\}$. We immediately have
\begin{equation*} 
        \sum_{\substack{i\in\I,\, \kappa(i)=K}}\lambda(C_i)= \alpha m_K< m_K = \sum_{(t,P)\in\Gamma_K}\delta_P.
\end{equation*}
Then, applying Lemma~\ref{lem:finite-allocation} simultaneously to the family $(C_i)_{\kappa(i)=K}$ and $(\delta_P)_{(t,P)\in\Gamma_K}$, we obtain finite unions of half-open intervals $C_{i,t,P}$ such that
\begin{equation*}\label{eq:finite-stage-Ci-partition}
        C_i =\bigsqcup_{(t,P)\in\Gamma_K}C_{i,t,P},
\end{equation*}
and
\begin{equation*}\label{eq:finite-stage-allocation-bound}
        \sum_{\substack{i\in\I,\, \kappa(i)=K}}\lambda(C_{i,t,P}) \le\delta_P  
\end{equation*}
for every $(t,P)\in\Gamma_K$.
For $P\in\mathcal P$ and $0\le u<M$, define
\[
        \widetilde C_{P,u} = \bigsqcup_{\substack{i\in\I,\,(u,P)\in\Gamma_{\kappa(i)}}} C_{i,u,P},
\]
and put $\widetilde C_{P,u}=\emptyset$ for $M\le u<h$.
Since the sets $G_K$ are pairwise disjoint, for each fixed pair $(P,u)$ there is at most one $K\in\mathcal K$ such that $P\subset S^{-u}G_K$. 
Therefore, if such a $K$ exists, then
\[
        \widetilde C_{P,u}  =  C_{(K,0),u,P}\sqcup  C_{(K,1),u,P}.
\]
Otherwise $\widetilde C_{P,u}=\emptyset$. In particular, $\widetilde C_{P,u}$ is a finite union of half-open intervals, and 
\begin{equation*} 
        \lambda(\widetilde C_{P,u})\le\delta_P.
\end{equation*}

The sets $\widetilde C_{P,u}$ are pairwise disjoint. Indeed, the pieces  $C_{i,t,P}$  belonging to a fixed $C_i$ are disjoint, while the sets $C_i$ themselves are contained in the pairwise disjoint sets $V_i$. Moreover, we have
\begin{equation*}\label{eq:finite-stage-prescribed-total}
        \sum_{P\in\mathcal P}\sum_{u=0}^{h-1} \lambda(\widetilde C_{P,u})  =  \sum_{i\in\I}\lambda(C_i)  =  \theta\lambda(R).
\end{equation*}

Put
\[
        d_{P,u}=\delta_P-\lambda(\widetilde C_{P,u}).
\]
Since $\sum_{P\in\mathcal P}\nu(P)=1$, we obtain 
\[
\begin{aligned}
        \sum_{P\in\mathcal P}\sum_{u=0}^{h-1}d_{P,u}
        &= \eta-\theta\lambda(R)  \\
        &< \lambda(R)-\theta\lambda(R)    \\
        &= \lambda\left( R\setminus \bigsqcup_{P\in\mathcal P}\bigsqcup_{u=0}^{h-1}\widetilde C_{P,u} \right).
\end{aligned}
\]
By successively cutting half-open intervals from the available set, we may choose pairwise disjoint finite unions of half-open intervals
\[
        D_{P,u}\subset R\setminus\bigsqcup_{P'\in\mathcal P}\bigsqcup_{v=0}^{h-1} \widetilde C_{P',v}
\]
such that $\lambda(D_{P,u})=d_{P,u}$.  Set
\[
        I_{P,u} = \widetilde C_{P,u}\sqcup D_{P,u}.
\]
Then the sets $I_{P,u}$ are pairwise disjoint, $\lambda(I_{P,u})=\delta_P$, and 
\[
         h\sum_{P\in\mathcal P}\delta_P =\eta \sum_{P\in\mathcal P}\nu(P)= \eta,
\]
This proves \eqref{eq:finite-stage-level-widths} and \eqref{eq:finite-stage-total-level-mass}.
Moreover, using the invariance of $\nu$, we obtain 
\[
\begin{aligned}
        \sum_{P\in\mathcal P}\sum_{u=0}^{h-1} \varepsilon_{P,u}\lambda(I_{P,u})
        &= \delta \sum_{u=0}^{h-1} \sum_{P\in\mathcal P,\, P\subset S^{-u}E} \nu(P) \\
        &=\delta\sum_{u=0}^{h-1}\nu(S^{-u}E) = h\delta\nu(E) =\eta\nu(E),
\end{aligned}
\]
proving \eqref{eq:finite-stage-labelled-mass}.

It remains to prove \eqref{eq:finite-stage-lower-on-Ci}.  Fix $i\in\I$, put $K=\kappa(i)$, and consider $(t,P)\in\Gamma_K$. For $y\in C_{i,t,P}$, we have 
\[
        C_{i,t,P} \subset \widetilde C_{P,t}\subset I_{P,t}\ \text{ and } \ P\subset S^{-t}G_K.
\]
For $1\le n\le K$, it is clear that  
\[
        t+a_n \le M-1+L <h.
\]
Therefore, by conditions \eqref{eq:finite-stage-tower-relations}  and \eqref{eq:finite-stage-F}, we obtain that, outside a null subset of $C_{i,t,P}$,
\[
        T^{a_n}y\in I_{P,t+a_n}, \quad F(T^{a_n}y)=H\varepsilon_{P,t+a_n}.
\]
For every $\omega\in P$, \eqref{eq:finite-stage-labels} implies 
\[
        \varepsilon_{P,t+a_n} =\mathbf 1_E(S^{t+a_n}\omega).
\]
Since $P\subset S^{-t}G_K$, condition \eqref{eq:finite-stage-GK} yields
\[
        \frac{1}{K\Lambda_q(K)} \sum_{n=1}^{K}F(T^{a_n}y) = \frac{H}{K\Lambda_q(K)} \sum_{n=1}^{K} \mathbf 1_E(S^{t+a_n}\omega)>A
\]
for a.e. $y\in C_{i,t,P}$.  Since $\{C_{i,t,P}:(t,P)\in\Gamma_K\}$ is a finite partition of $C_i$, the same inequality holds for a.e. $y\in C_i$, proving \eqref{eq:finite-stage-lower-on-Ci}.
\end{proof}

\begin{thm}\label{thm:finite-stage-lower}
Let $q\in\N$,  $\mathbf a=(a_n)$ be  a sequence of nonnegative integers, and  $\Phi$ be a finite-valued
Young function.  Suppose that there is an $(\mathbf a,q,\Phi)$-admissible sequence of finite stages.  
Then, $(\mathbf a,\mathbf D^{(q)})$ admits an  $L^\Phi$ local $\infty$-sweeping out  counterexample on  $[0,1]$. More precisely,
for every $0<\varepsilon<1$, there are an invertible ergodic Lebesgue
measure-preserving transformation $T$ of $[0,1]$, a nonnegative
$f\in L^\Phi([0,1])$, and a measurable set $Y\subset[0,1]$ such that
\[
        \lambda(Y)>1-\varepsilon,
\]
and
\begin{equation}\label{eq:finite-stage-local-divergence}
        \limsup_{r\downarrow 0,\,N\to \infty} \mathcal L_r\left(\frac{1}{N\Lambda_q(N)} \sum_{n=1}^{N}f(T^{a_n}\cdot) \right)(x) = \infty
\end{equation}
for a.e. $x\in Y$.
\end{thm}

\begin{proof}
Let $ \{\mathfrak s_j = (\Omega_j,\B_{\Omega_j},\nu_j,S_j,  E_j,H_j,A_j,\mathcal K_j)\}_{j\in\N}$ be an $(\mathbf a,q,\Phi)$-admissible sequence, 
and let $c_0$ be as in \eqref{eq:finite-stage-success}.  Since $A_j\to \infty$, we choose a strictly increasing sequence $(l_j)$ such that $A_{l_j}\ge j^4$ for all $j\in\N$.
After replacing $\mathfrak s_j$ by $\mathfrak s_{l_j}$ and relabelling, we may assume that 
\begin{equation*}\label{eq:finite-stage-A-growth}
         A_j\ge j^4, \quad j\in\N.
\end{equation*}
Indeed, it suffices to check that condition \eqref{eq:finite-stage-summability} is preserved.  For each $j\in\N$, since $l_j\ge j$, by convexity
of $\Phi$ and $\Phi(0)=0$, we have
\[
        2^{-j}\Phi(2^jH_{l_j})\le 2^{-l_j}\Phi(2^{l_j}H_{l_j}).
\]
Consequently,
\[
\sum_{j=1}^{\infty} 2^{-j}\nu_{l_j}(E_{l_j})\Phi(2^jH_{l_j}) \le \sum_{m=1}^{\infty} 2^{-m}\nu_m(E_m)\Phi(2^mH_m) <\infty.
\]
Thus \eqref{eq:finite-stage-summability} remains valid after relabelling.

Let $0<\varepsilon<1$. Choose $c_\varepsilon>0$ sufficiently small, and define $\eta_j=\frac{c_\varepsilon}{j^2}$, such that
\begin{equation*}\label{eq:finite-stage-eta}
        0<\eta_j\le 1  \ \text{ and } \  \sum_{j=1}^{\infty}\eta_j  < \min\left\{\varepsilon,\frac1{16}\right\}.
\end{equation*} 
By the assumption of $A_j$, we have $\eta_jA_j\to \infty$. Set $\theta_j=\frac{c_0\eta_j}{64}$. Write $\mathcal K_j = \{K_{j,1}<\dots<K_{j,N_j}\}$. For each $1\le\ell\le N_j$, put
\[
        B_{j,K_{j,\ell}} =\left\{\omega\in \Omega_j: \frac{H_j}{K_{j,\ell}\Lambda_q(K_{j,\ell})} \sum_{n=1}^{K_{j,\ell}}\mathbf 1_{E_j}(S_j^{a_n}\omega)>A_j \right\},
\]
and
\[
        G_{j,K_{j,\ell}}=  B_{j,K_{j,\ell}}\setminus \bigcup_{r<\ell}B_{j,K_{j,r}},
\]
where the union is understood to be empty when $\ell=1$.
Then the sets $G_{j,K}$, $K\in\mathcal K_j$, are pairwise disjoint,  each satisfies the property \eqref{eq:finite-stage-GK}, and
\begin{equation*}\label{eq:finite-stage-disjoint-success}
        \sum_{K\in\mathcal K_j}\nu_j(G_{j,K})=\nu_j\bigl(\mathcal G(\mathfrak s_j)\bigr)\ge c_0.
\end{equation*}

We now choose the interval levels inductively.  Suppose that the levels for the stages preceding $j$ have already been selected, 
and let $R_j$ denote the complement of their union in $[0,1]$.  At this finite stage $j$, after discarding finitely
many endpoints, $R_j$ is a finite union of half-open intervals, and
\[
        \lambda(R_j) > 1-\sum_{i=1}^{\infty}\eta_i > \frac{15}{16}.
\]

Since $\eta_j<1/16<\lambda(R_j)/4$, applying Lemma~\ref{lem:finite-stage-interval} with $R=R_j$, $c=c_0$, $\eta=\eta_j$,  $\theta=\theta_j$ and $\rho=2^{-j}$,
we obtain an integer $h_j$, finite sets $\mathcal P_j,\I_j$, a map $\kappa_j:\I_j\to\mathcal K_j$, levels $I_{j,P,u}$, numbers $\delta_{j,P}>0$, $P\in \mathcal P_j$, labels $\varepsilon_{j,P,u}\in \{0,1\}$, and sets $V_{j,i},C_{j,i}$ satisfying properties (a)-(d) of the lemma. In particular, we have  
\[
\delta_{j,P}=\lambda(I_{j,P,u}), \quad P\in\mathcal P_j,\quad 0\le u<h_j,
\]
and  
\[
\sum_{P\in\mathcal P_j}h_j\delta_{j,P}=\eta_j.
\]
Consequently,
\[
        \sum_{j=1}^{\infty} \sum_{P\in\mathcal P_j}h_j\delta_{j,P} = \sum_{j=1}^{\infty}\eta_j <1.
\] 
For each $j,P$ and $0\le u<h_j-1$, choose a measure-preserving isomorphism
\[
        \phi_{j,P,u}:I_{j,P,u}\to  I_{j,P,u+1}.
\]
Applying Lemma~\ref{lem:tower-embedding} to the countable family of columns indexed by the pairs $(j,P)$ with $\ell_{j,P}=h_j-1$, 
we obtain an invertible ergodic Lebesgue measure-preserving transformation $T$ of $[0,1]$ such that
\[
        T=\phi_{j,P,u} \quad\text{a.e. on }I_{j,P,u},  \quad 0\le u<h_j-1.
\]
Define
\[
        F_j=H_j\sum_{P\in\mathcal P_j}\sum_{\substack{0\le u<h_j,\,\varepsilon_{j,P,u}=1}} \mathbf 1_{I_{j,P,u}} \ \text{ and }
        \  f=\sum_{j=1}^{\infty}F_j.
\]
Clearly, the supports of the functions $F_j$ are pairwise disjoint.  By \eqref{eq:finite-stage-labelled-mass}, $\eta_j\le 1$ and convexity of $\Phi$, we have
\[
\begin{aligned}
        \int_{[0,1]}\Phi(f)\dd\lambda
        &=  \sum_{j=1}^{\infty} \eta_j\nu_j(E_j)\Phi(H_j)    \\
        &\le\sum_{j=1}^{\infty} 2^{-j}\nu_j(E_j)\Phi(2^jH_j)  <\infty.
\end{aligned}
\]
That is, $f\in L^\Phi([0,1])$.

Denote
\[
        Y= [0,1]\setminus \bigsqcup_{j=1}^{\infty}\bigsqcup_{P\in\mathcal P_j}  \bigsqcup_{u=0}^{h_j-1}I_{j,P,u}.
\]
Then
\[
        \lambda(Y)= 1-\sum_{j=1}^{\infty}\eta_j >  1-\varepsilon.
\]
Fix $x\in Y$ outside the countable set of endpoints arising in the construction.  For each $j$, the point $x$ belongs to a unique component interval $J_j(x)$ of a unique set $V_{j,i_j(x)}$.  Put
\[
        K_j(x)=\kappa_j(i_j(x)) \ \text{ and } \ r_j(x)=2|J_j(x)|.
\]
Since $|J_j(x)|\le2^{-j}$,  $r_j(x)\to 0$.  Moreover,
\[
        J_j(x)\subset B(x,r_j(x)) \ \text{ and } \  \lambda(B(x,r_j(x)))\le4|J_j(x)|.
\]
From \eqref{eq:finite-stage-density}, it follows that
\[
        \lambda(C_{j,i_j(x)}\cap J_j(x))
        =
        \theta_j|J_j(x)|.
\]
Applying Lemma~\ref{lem:finite-stage-interval} with the transformation $T$, by \eqref{eq:finite-stage-lower-on-Ci} we obtain
\[
        \frac{1}{K_j(x)\Lambda_q(K_j(x))}  \sum_{n=1}^{K_j(x)} F_j(T^{a_n}y)>A_j
\]
for a.e. $y\in C_{j,i_j(x)}$.  Since $0\le F_j\le f$, we have
\[
    \mathcal L_{r_j(x)} \left(\frac{1}{K_j(x)\Lambda_q(K_j(x))}\sum_{n=1}^{K_j(x)}f(T^{a_n}\cdot)\right)(x)  
    \ge A_j  \frac{\lambda(C_{j,i_j(x)}\cap J_j(x))}{\lambda(B(x,r_j(x)))} 
    \ge \frac{\theta_jA_j}{4}.
\]
Observe that
\[
        \frac{\theta_jA_j}{4}=  \frac{c_0\eta_jA_j}{256}\to \infty \ \text{ and } \ K_j(x)\ge\min\mathcal K_j\to \infty.
\]
Thus \eqref{eq:finite-stage-local-divergence} holds at every such $x$, and hence for almost every $x\in Y$.
\end{proof}

\begin{coro}\label{cor:finite-stage-sharp}
Let $q\in\N$ and let $\mathbf a=(a_n)$ be  a sequence of nonnegative integers.  Assume that  for every  Young function $\Phi$ satisfying
\[
        \Phi(t)=o\bigl(tL_q(t)\bigr) \quad (t\to \infty),
\]
there is an $(\mathbf a,q,\Phi)$-admissible sequence of finite stages.
Then the endpoint $ L\log_{q}L$ is sharp for the local observation problem associated with $(\mathbf a,\mathbf D^{(q)})$.
\end{coro}

\begin{proof}
The positive  assertion follows from Theorem~\ref{thm:universal-positive}, and the lower assertion follows from Theorem~\ref{thm:finite-stage-lower}.
\end{proof}
 
\subsection{Residue blocks under polynomial ratio separation}\label{subsec:poly-ratio-blocks}

The purpose of this subsection is to verify the finite-stage assumptions for sequences with polynomial ratio separation.  
The residue construction used in this subsection is adapted from Quas and Wierdl \cite{QuasWierdl}.  
In the proof of \cite[Corollary~4.13]{QuasWierdl}, the condition
\[
        \frac{a_{n+1}/a_n}{n^\varepsilon}\to \infty
\]
is used to realize any prescribed finite residue pattern on a non-degenerate interval of parameters.  The covering of the residue classes which converts
such residue prescriptions into large maximal averages is carried out in the proof of \cite[Lemma~4.11]{QuasWierdl}.  The polynomial ratio
separation assumed below implies the displayed condition for every $0<\varepsilon<\eta$.

Here we isolate the finite quantitative form needed for the later interval construction.  At each stage we select finitely many averaging lengths and
an interval of admissible parameters, and obtain both a quantitative support bound and a uniform lower bound for the set on which one of the selected
averages is large.  These estimates are not stated separately in the cited argument, so we give the residue construction in full.

For $t\ge 2$, set
\[
        D_t=\{2^{t-1}+1,\ldots,2^t\}.
\]
If $U\subset\{2,3,\ldots\}$ is finite, put
\[
        I_U=\bigsqcup_{t\in U}D_t.
\]

The next lemma is the finite nesting step used in the proof of \cite[Corollary~4.13]{QuasWierdl}.  We provide the detailed proof because the
interval form will be convenient below.

\begin{lem}\label{lem:lacunary-residue-nesting}
Let $Q\ge 1$,  and let $0<b_1<\cdots<b_L$ be real numbers such that 
\[
        \frac{b_{i+1}}{b_i}>2Q, \quad  1\le i<L.
\]
For every $r_1,\ldots,r_L\in\{0,\ldots,Q-1\}$, the set of $\alpha>0$ such that
\[
        \lfloor\alpha b_i\rfloor\equiv r_i\pmod Q, \quad  1\le i\le L,
\]
contains a non-degenerate compact interval.
\end{lem}

\begin{proof}
Write $\alpha=Q\beta$.  For $1\le i\le L$, let
\[
        U_i=\left\{\beta \in\R: \beta b_i\pmod1\in \left[\frac{r_i}{Q},\frac{r_i+1}{Q}\right) \right\}.
\]
Each component of $U_i$ is a half-open interval of length $1/(Q b_i)$, and the left endpoints of consecutive components are separated by  $1/b_i$.

We can construct nested components $J_1\supset J_2\supset\dots\supset J_L $ with 
\[
     |J_i|=\frac1{Qb_i} \ \text{ and } \ J_i \subset \bigcap_{\ell=1}^i U_\ell.
\]
First, choose a component $J_1\subset (0,\infty)$ of $U_1$.  Suppose that $J_i$, with $1\le i<L$, has been chosen.  Since
\[
        \frac1{b_{i+1}}+\frac1{Qb_{i+1}}=\frac{Q+1}{Qb_{i+1}} < \frac1{Qb_i} = |J_i|,
\]
the interval $J_i$ contains a complete component of $U_{i+1}$.  Choose one such component as $J_{i+1}$.

For $\beta\in J_L$, write
\[
        \beta b_i=m_i+\theta_i, \quad m_i\in\Z_+, \quad \frac{r_i}{Q}\le\theta_i<\frac{r_i+1}{Q}.
\]
Then
\[
        \lfloor Q\beta b_i\rfloor=Qm_i+r_i,
\]
and hence
\[
        \lfloor\alpha b_i\rfloor\equiv r_i\pmod Q.
\]
Any non-degenerate compact interval contained in the interior of $QJ_L$ meets the required property.
\end{proof}

The residue allocation and circle-rotation argument in the next lemma are the finite ingredients used in the proof of \cite[Lemma~4.11]{QuasWierdl}.  
We give the precise quantitative form needed later.

\begin{lem}\label{lem:finite-QW-rotation}
Let $\mathbf a=(a_n)_{n\ge 1}$ be a sequence of positive integers, $U\subset\{2,3,\ldots\}$ be finite and nonempty, $N, M\ge 1$, and $Q=NM \ge 2$. 
 Assume that positive numbers $(w_t)_{t\in U}$ satisfy
\[
        \sum_{t\in U}w_t>2M,  \quad   \frac1N<w_t<\frac{2^t}{4N},  \quad t\in U.
\]
Assume further that every map
\[
        \mathbf r: I_U\to \Z/Q\Z
\]
is realized by some $\alpha>0$, in the sense that
\[
        \lfloor\alpha a_n\rfloor\equiv r(n)\pmod Q,
        \quad  n\in I_U.
\]
Then there exists $\alpha>0$ such that, with $E=[0,2/Q)\subset\mathbb T$,  for the rotation
\[
        T_\alpha x=x+\frac{\alpha}{Q}\pmod1,
\]
one has
\[
        \max_{t\in U}\frac{8Nw_t}{2^t}\sum_{n=1}^{2^t}\1_E(T_\alpha^{a_n}x)>1 
\]
holds for every $x\in\mathbb T$.
\end{lem}

\begin{proof}
For $t\in U$, set $ m_t=\min\{Q,\lfloor Nw_t\rfloor\}$. Since $N w_t>1$, we have $m_t\ge 1$.  Hence if $m_t<Q$, then 
\[
        m_t=\lfloor Nw_t\rfloor\ge \frac{Nw_t}{2}.
\]
Consequently, either $m_t=Q$ for some $t$, or
\[
        \sum_{t\in U}m_t \ge  \frac N2\sum_{t\in U}w_t > NM = Q.
\]
In both cases, we get 
\[
        \sum_{t\in U}m_t\ge Q.
\]

Write $U=\{t_1<\dots<t_s\}$, set $c_0=0$, and define  recursively 
\[
        R_{t_i}=\{c_{i-1},c_{i-1}+1,\ldots,c_{i-1}+m_{t_i}-1\}  \pmod Q,\quad  c_i=c_{i-1}+m_{t_i}
\]
for $1\le i\le s$. Then $\#R_t=m_t$ and
\[
        \bigcup_{t\in U}R_t=\Z/Q\Z.
\]

For $t\in U$, prescribe residues on $D_t=\{2^{t-1}+1,\ldots,2^t\}  $ so that every $r\in R_t$ occurs at least $\left\lfloor\frac{2^{t-1}}{m_t}\right\rfloor$ times.  This is possible since $m_t\left\lfloor\frac{2^{t-1}}{m_t}\right\rfloor \le 2^{t-1}=\#D_t  $. Complete the prescription arbitrarily on the remaining indices of $D_t$.
Since $m_t\le Nw_t<2^{t-2}$, we have $2^{t-1}/m_t>2$, and therefore
\[
        \left\lfloor\frac{2^{t-1}}{m_t}\right\rfloor  \ge  \frac{2^{t-2}}{m_t}  \ge  \frac{2^{t-2}}{Nw_t}.
\]
Doing this for all $t\in U$ defines a map $\mathbf r:I_U\to\Z/Q\Z$. By hypothesis, choose $\alpha>0$ such that  
\[
        \lfloor\alpha a_n\rfloor\equiv \mathbf r(n)\pmod Q, \quad n\in I_U.
\] 

Let $E=[0,2/Q)\subset\mathbb T$.  Fix $x\in\mathbb T$. Choose $0\le r<Q$ such that 
\[
        x\in  \left[\frac{Q-r}{Q},\frac{Q-r+1}{Q}\right) \pmod1.
\]
 Since the sets $R_t$, $t\in U$, cover $\Z/Q\Z$,  there exists
$t\in U$ such that $r\in R_t$. By construction, there are at
least $\left\lfloor\frac{2^{t-1}}{m_t}\right\rfloor$  indices $n\in D_t$ for which $\mathbf r(n)=r$. 
For each such $n$, write
\[
        \alpha a_n=k_nQ+r+\theta_n, \quad k_n\in\Z,\quad 0\le\theta_n<1,
\]
and 
\[
        x=\frac {Q-r}{Q}+\frac{\eta}{Q}\pmod 1, \quad 0\le\eta<1.
\]
Then we have
\[
        T_\alpha^{a_n}x=\frac{\eta+\theta_n}{Q}\pmod 1.
\]
Since $0\le\eta+\theta_n<2\le Q$, it follows that
\[
        T_\alpha^{a_n}x\in[0,2/Q)=E.
\]
Consequently,
\[
        \frac{8Nw_t}{2^t}\sum_{n=1}^{2^t}\1_E(T_\alpha^{a_n}x)\ge \frac{8Nw_t}{2^t}\frac{2^{t-2}}{Nw_t} =2>1.
\]
Since $x\in\mathbb T$ was arbitrary, the desired conclusion follows.
\end{proof}

\begin{lem}\label{lem:iterated-log-inversion}
Fix $q\in\N$ and $C\ge 1$. Let $(u_m)$ and $(v_m)$ be positive sequences such that $u_m\to \infty$, $ v_m\ge u_m$ and $ L_q(v_m)\le C L_q(u_m)$ for all sufficiently large $m\in\N$. Then 
\[
       \log v_m=o(u_m).
\]
\end{lem}

\begin{proof}
For $r\ge0$, write $E_r=\exp^{\circ r}$, with $E_0$ the identity.
Since
\[
        L_{j-1}(t)\le \exp(L_j(t)), \quad j\ge 2,
\]
we have
\[
        L_1(v_m)\le E_{q-1}\bigl(L_q(v_m)\bigr) \le E_{q-1}\bigl(C L_q(u_m)\bigr).
\]

We claim that: for every fixed $k\in\N$, $A>0$, and $\varepsilon>0$,
\[
        E_{k-1}\bigl(A L_k(u)\bigr)\le u^\varepsilon
\]
holds for all sufficiently large $u$. This follows by induction on $k$. The case $k=1$ is immediate from $L_1(u)=O(\log u)=o(u^\varepsilon)$. 
Now suppose that the assertion holds for $k-1$, where $k\ge 2$. Since $L_k(u)=L_{k-1}(L_1(u))$, we have
\[
        E_{k-2}\bigl(A L_k(u)\bigr)\le L_1(u)^{1/2}
\]
for all sufficiently large $u$, and hence for any $\varepsilon>0$,
\[
        E_{k-1}\bigl(A L_k(u)\bigr) \le \exp(L_1(u)^{1/2})\le u^\varepsilon,
\]
because $L_1(u)^{1/2}=o(\log u)$.

Taking $k=q$, $A=C$, and $\varepsilon=1/2$, together with $v_m\ge u_m\to \infty$,  we obtain
\[
        0\le \log v_m\le L_1(v_m)\le u_m^{1/2}=o(u_m)
\]
for all sufficiently large $m$. 
\end{proof}

\begin{lem}\label{lem:block-parameter-growth}
Fix $q\in \N$.  For each sufficiently large integer $M$, choose $t_0=t_0(M)\in \N$ such that
\[
        M\le L_q(t_0)\le 2M,
\]
and let $t_1=t_1(M)$ be the least integer $t\ge t_0$ satisfying
\[
        \sum_{s=t_0}^{t}\frac1{\Lambda_q(2^s)} > 2M.
\]
Then we have
\[
        \log\!\bigl(\Lambda_q(2^{t_1})M\bigr)=o(t_0) \quad (M\to \infty).
\]
Consequently,
\[
        \Lambda_q(2^{t_1})M=o(2^{t_0}),  \quad  M=o(2^{t_0}),
\]
and, for all sufficiently large $M$,
\begin{equation}\label{eq:block-slack}
        2\Lambda_q(2^{t_1})M+M \le 2^{t_0-4}.
\end{equation}
\end{lem}

\begin{proof}
Since $L_q$ is increasing, $L_q(t)\to \infty$ and $L_q(t+1)-L_q(t)\to 0$, the choice of $t_0$ is possible for all sufficiently large $M$.  
The existence of $t_1$ follows from Lemma~\ref{lem:iterated-sum}.

By minimality of $t_1$,
\[
        \sum_{s=t_0}^{t_1-1}\frac1{\Lambda_q(2^s)}\le 2M.
\]
Hence, again by Lemma~\ref{lem:iterated-sum}, we have
\[
\begin{aligned}
        1+L_q(t_1-1)
        &\lesssim_q \sum_{s=2}^{t_1-1}\frac1{\Lambda_q(2^s)}  \\
        &\le \sum_{s=2}^{t_0-1}\frac1{\Lambda_q(2^s)}+2M \lesssim_q 1+L_q(t_0)+M \lesssim_q M.
\end{aligned}
\]
As $M\to \infty$, we have $t_0(M)\to \infty$, and hence $t_1(M)\to \infty$. Thus $L_q(t_1)-L_q(t_1-1)=o(1)$. It follows that
\[
        L_q(t_1)\lesssim_q M\lesssim L_q(t_0).
\]
Thus, by Lemma~\ref{lem:iterated-log-inversion}, we have
\[
        \log t_1=o(t_0).
\]

For fixed $q$,
\[
        \Lambda_q(2^t)\lesssim_q (1+t)^q,
\]
and hence
\[
        \log\Lambda_q(2^{t_1}) \lesssim_q 1+\log(1+t_1) =o(t_0).
\]
Moreover,
\[
        M\le L_q(t_0)\le L_1(t_0)=\log(\mathrm e+t_0),
\]
so $\log M=o(t_0)$.  It follows that
\[
        \log\!\bigl(\Lambda_q(2^{t_1})M\bigr)=o(t_0).
\]
Since $2^{t_0}=\exp((\log2)t_0)$, the remaining assertions immediately follow.
\end{proof}

\begin{defn}\label{def:poly-ratio}
A sequence $\mathbf a=(a_n)_{n\ge 1}$ of positive integers has \emph{polynomial ratio separation} if there are $\eta>0$ and $n_*\in\N$ such that
\[
        \frac{a_{n+1}}{a_n}\ge n^\eta 
\]
for all $ n\ge n_*$.
\end{defn}

\begin{prop}\label{prop:poly-ratio-stages}
Let $q\in\N$, and let $\mathbf a=(a_n)$ have polynomial ratio separation.  For every  Young function $\Phi$ satisfying
\[
        \Phi(t)=o\bigl(tL_q(t)\bigr)\quad (t\to \infty),
\]
there exists an $(\mathbf a,q,\Phi)$-admissible sequence of finite stages.
\end{prop}

\begin{proof}
Set
\[
        \rho_\Phi(s)= \frac{\Phi(s)}{s(1+L_q(s))}.
\] 
By assumption, $\rho_\Phi(s)\to 0$ as $s\to \infty$. Since $\mathbf a$ has polynomial ratio separation, choose $\eta>0$ and $n_*\in\N$ as in Definition~\ref{def:poly-ratio}. Fix a positive summable sequence $(\varepsilon_j)_{j\ge 1}$.

In view of Definition~\ref{def:finite-stage-hypotheses}, it is enough to
construct, for every $j\in\N$, a finite stage
\[
        \mathfrak s_j= \bigl(  \mathbb T,\B_{\mathbb T},\lambda,   T_j,E_j,H_j,A_j,\mathcal K_j  \bigr)
\]
such that
\[
        A_j=j,\quad  \min\mathcal K_j\ge 2^j, \quad  \mathcal G(\mathfrak s_j)=\mathbb T,
\]
and
\begin{equation}\label{eq:poly-stage-cost-target}
        2^{-j}\lambda(E_j)\Phi(2^jH_j)  \le\varepsilon_j.
\end{equation}

Fix $j\in\N$.  For every sufficiently large integer $M$, let $t_0=t_0(M)$ and $t_1=t_1(M)$ be given by Lemma~\ref{lem:block-parameter-growth}, and put
\[
        N=\left\lceil2\Lambda_q(2^{t_1})\right\rceil,   \quad   U=\{t_0,\ldots,t_1\},  \quad Q=NM,
\]
and
\[
        w_t=\frac1{\Lambda_q(2^t)}, \quad t\in U.
\]
Since $\Lambda_q(2^{t_1})\ge 1$, we have
\[
        Q  \le  \bigl(2\Lambda_q(2^{t_1})+1\bigr)M  \le 3\Lambda_q(2^{t_1})M.
\]
Then by Lemma~\ref{lem:block-parameter-growth}, 
\[
        \log Q=o(t_0) \quad (M\to \infty).
\]
Since $t_0(M)\to \infty$, after increasing $M$ we may assume that
\begin{equation}\label{eq:poly-stage-large-M}
        t_0\ge  \max\{j,3\}, \quad  2^{t_0-1}\ge n_*,  \quad  2^{\eta(t_0-1)}>2Q,
\end{equation}
and that \eqref{eq:block-slack} holds.  

We first verify the residue-realization assumption in Lemma~\ref{lem:finite-QW-rotation}.  Since $U$ is an  integer interval, 
\[
        I_U = \{2^{t_0-1}+1,\ldots,2^{t_1}\}.
\]
Let $\mathbf r:I_U \to \Z/Q\Z$ be arbitrary.  If $n,n+1\in I_U$, we have $n>2^{t_0-1}\ge n_*$, and hence
\[
        \frac{a_{n+1}}{a_n}\ge n^\eta> 2^{\eta(t_0-1)} > 2Q.
\]
Thus, applying Lemma~\ref{lem:lacunary-residue-nesting} to the finite sequence $(a_n)_{n\in I_U}$ , we obtain a realization of the prescribed
residues.

We next check the weight assumptions.  For $t\in U$, it is clear that $N\ge 2\Lambda_q(2^{t_1})  \ge 2\Lambda_q(2^t)$, and hence
\[ 
\frac1N<w_t.
\]
Moreover, \eqref{eq:block-slack} implies 
\[
        Q \le\bigl(2\Lambda_q(2^{t_1})+1\bigr)M  \le  2^{t_0-4}.
\]
Since $N\le Q$ and $t\ge t_0$, we have 
\[
        \frac{2^t}{4N} \ge \frac{2^{t_0}}{4\cdot2^{t_0-4}}  = 4.
\]
On the other hand, $\Lambda_q(2^t)\ge 1$, so $w_t\le 1$.  Hence
\[
        w_t<\frac{2^t}{4N}, \quad t\in U.
\]
Finally, by the definition of $t_1$, we get 
\[
        \sum_{t\in U}w_t>2M.
\]
Therefore, all the assumptions of Lemma~\ref{lem:finite-QW-rotation} are satisfied for every sufficiently large $M$.

We now estimate the prospective Orlicz cost.  By the minimality of $t_1$, we have
\[
        \sum_{t=t_0}^{t_1-1} \frac1{\Lambda_q(2^t)} \le 2M.
\]
Since $\Lambda_q(2^{t_1})\ge 1$, Lemma~\ref{lem:iterated-sum} and the assumption  $M\le L_q(t_0)\le 2M$ given by the choice of $t_0$,   yields
\[
\begin{aligned}
        1+L_q(t_1)
        &\lesssim_q 1+\sum_{t=2}^{t_1}\frac1{\Lambda_q(2^t)}    \\
        &\le 1+\sum_{t=2}^{t_0-1}\frac1{\Lambda_q(2^t)}+2M+\frac1{\Lambda_q(2^{t_1})}    \\
        &\lesssim_q 1+L_q(t_0)+M \lesssim_q M.
\end{aligned}
\]
Furthermore, we have 
\[
        N \le 2\Lambda_q(2^{t_1})+1 \lesssim_q(1+t_1)^q,
\]
which follows from the fact $L_\ell(2^{t_1})\le L_1(2^{t_1})\lesssim1+t_1$ for $1\le\ell\le q$.
Put $s_{j,M}=2^{j+3}jN$. Then $s_{j,M}\lesssim_{j,q}(1+t_1)^q$. By Lemma~\ref{lem:Psi-estimates}(c) and Lemma~\ref{lem:log-calc}\ref{lem:log-calc:3}, we obtain
\[
        1+L_q(s_{j,M}) \lesssim_{j,q} 1+L_q(\mathrm e+t_1)\lesssim_q 1+L_q(t_1) \lesssim_q M.
\]
Consequently,
\[
        2^{-j}\frac{2}{NM}\Phi(s_{j,M})=\frac{16j}{M}\, \rho_\Phi(s_{j,M}) \bigl(1+L_q(s_{j,M})\bigr) \lesssim_{j,q}\rho_\Phi(s_{j,M}).
\]
since $t_1(M)\ge t_0(M)\to \infty$ and $\Lambda_q(2^{t_1(M)})\to \infty$, it follows that $N(M)=\left\lceil2\Lambda_q(2^{t_1})\right\rceil\to \infty$, and hence for the fixed index $j$, $s_{j,M}\to \infty$. Consequently, 
\[
        2^{-j}\frac{2}{NM}\Phi(s_{j,M}) \to 0  \quad (M\to \infty).
\]

Choose $M_j$ so large that \eqref{eq:poly-stage-large-M}, all the weight
and residue requirements above, and \eqref{eq:block-slack} hold, and also
\[
        2^{-j}\frac{2}{N_jM_j}   \Phi(2^{j+3}jN_j)   \le\varepsilon_j,
\]
where $t_{0,j}=t_0(M_j)$, $t_{1,j}=t_1(M_j)$, $N_j=\left\lceil2\Lambda_q(2^{t_{1,j}})\right\rceil$, $U_j=\{t_{0,j},\ldots,t_{1,j}\}$ and $Q_j=N_jM_j$.
Apply Lemma~\ref{lem:finite-QW-rotation} with
\[
        w_{j,t}=\frac1{\Lambda_q(2^t)},\quad t\in U_j.
\]
Let $\alpha_j>0$ be supplied by Lemma~\ref{lem:finite-QW-rotation} and set
\[
        T_jx=x+\frac{\alpha_j}{Q_j}\pmod1,   \quad    E_j=[0,2/Q_j),
\]
and
\[
        H_j=8jN_j, \quad A_j=j, \quad\mathcal K_j=\{2^t:t\in U_j\}.
\]
Moreover, Lemma~\ref{lem:finite-QW-rotation} shows
\[
        \max_{t\in U_j} \frac{8N_j}{2^t\Lambda_q(2^t)} \sum_{n=1}^{2^t} \1_{E_j}(T_j^{a_n}x)>1 \quad \text{ for every } \ x\in\mathbb T.
\]
Equivalently we have
\[
        \max_{K\in\mathcal K_j} \frac{H_j}{K\Lambda_q(K)}\sum_{n=1}^{K} \1_{E_j}(T_j^{a_n}x)>A_j\quad \text{ for every } \ x\in\mathbb T.
\]

Define $\mathfrak s_j= \bigl(  \mathbb T,\B_{\mathbb T},\lambda, T_j,E_j,H_j,A_j,\mathcal K_j \bigr)$. 
Since $T_j$ is an invertible Lebesgue measure-preserving rotation, $\mathfrak s_j$ is a finite stage.  Furthermore,
\[
        A_j=j\to \infty, \ \text{ and } \  \min\mathcal K_j =  2^{t_{0,j}} \ge 2^j \to \infty.
\]
Thus \eqref{eq:finite-stage-escape} holds.  The preceding pointwise estimate gives $\mathcal G(\mathfrak s_j)=\mathbb T$, 
proving \eqref{eq:finite-stage-success} with $c_0=1$.
Finally, since $E_j=[0,2/Q_j)$, $Q_j=N_jM_j$, and $H_j=8jN_j$, the choice of $M_j$ yields
\[
        \sum_{j=1}^{\infty} 2^{-j}\lambda(E_j)\Phi(2^jH_j) = \sum_{j=1}^{\infty} 2^{-j}\frac{2}{N_jM_j}  \Phi(2^{j+3}jN_j)
        \le  \sum_{j=1}^{\infty}\varepsilon_j <\infty,
\]
proving \eqref{eq:finite-stage-summability}. Hence $(\mathfrak s_j)_{j\ge 1}$ is an $(\mathbf a,q,\Phi)$-admissible sequence of finite stages.
\end{proof}

\subsection{Exponential blocks and finite amplification}\label{subsec:exponential-blocks}

We verify the finite-stage assumptions for the sequence $\mathbf a=(k^n)_{n\ge 1}$, where $k\ge 2$.  The proof of
\cite[Theorem~6.1]{QuasWierdl} uses primitive cyclic classes of base-$k$ words and, at the final step, compares the scales $y$ and
$2y$.  We use these two ideas in a finite quantitative form.  More precisely, the construction below gives a testing set of controlled measure
and a lower bound for the measure of the associated large-value set at finitely many averaging lengths.  A finite product then turns this lower bound into a fixed positive constant.  We include the details so that the verification of \eqref{eq:finite-stage-escape}--\eqref{eq:finite-stage-summability}
is self-contained and uses the normalization required here.

For $y\in\N$, set
\[
        n_0(y)=\left\lfloor2y\log_k2\right\rfloor.
\]
Fix $D_k\ge 2^{12}$.  If $w=(w_t)_{t\ge 2}$ is a positive sequence with
$w_t\to 0$, define
\[
        m_y(w)  = \sum_{\substack{t\ge 2\\2^{-y}<w_t<2^{t-y}}}w_t
\]
and
\[
        l_y(w)= \sum_{\substack{t\ge 2\\ D_kn_0(y)2^{-y}<w_t<2^{t-y}}}w_t.
\]

The next lemma is a modified form of the $y$-versus-$2y$ argument used at the end of the proof of \cite[Theorem~6.1]{QuasWierdl}.  
The underlying scale-splitting idea is the same, but the estimate needed here has a different normalization: in particular, the additional lower cutoff
$D_kn_0(y)2^{-y}$ is imposed by the finite repeated-word construction below.  We therefore present the required form separately.

\begin{lem}\label{lem:QW-mass-split}
There is $C_k>0$ such that, for every positive sequence $w_t\to 0$ and every sufficiently large $y$,
\[
        l_y(w)+l_{2y}(w) \ge m_y(w)-C_ky^2\,2^{-y}.
\]
\end{lem}

\begin{proof}
Since $n_0(2y)=O_k(y)$, we may choose $y_0=y_0(k)$ so large that
\[
        D_kn_0(2y)2^{-2y}<2^{-y}
\]
whenever $y\ge y_0$.  Fix such a $y$, and let $\mathcal R_y$ denote the set of indices whose terms are counted in $m_y(w)$ but in neither $l_y(w)$ nor $l_{2y}(w)$.

If $t\in\mathcal R_y$, then $2^{-y}<w_t<2^{t-y}$. Since the upper condition in the definition of $l_y(w)$ is already satisfied, the fact that $w_t$ is not counted in $l_y(w)$ yields
\[
        w_t\le D_kn_0(y)2^{-y}.
\]
On the other hand, $w_t>2^{-y}>D_kn_0(2y)2^{-2y}$, so $w_t$ satisfies the lower condition in the definition of $l_{2y}(w)$.  Since $w_t$  is not counted in $l_{2y}(w)$,
the upper condition there must fail, and hence $w_t\ge 2^{t-2y}$. Therefore
\[
        \mathcal R_y  \subset\left\{  t\ge 2: 2^{t-2y}\le D_kn_0(y)2^{-y} \right\}.
\]
Every index in the set on the right  satisfies
\[
        t\le y+\log_2\!\bigl(D_kn_0(y)\bigr).
\]
Since $n_0(y)=O_k(y)$, it follows that $\#\mathcal R_y=O_k(y)$. By the definition of $\mathcal R_y$ and the positivity of the terms,
\[
        m_y(w)\le l_y(w)+l_{2y}(w)  +\sum_{t\in\mathcal R_y}w_t.
\]
Moreover, $w_t\le D_kn_0(y)2^{-y}$ for every $t\in\mathcal R_y$.  Consequently, we have 
\[
        \sum_{t\in\mathcal R_y}w_t \le D_kn_0(y)2^{-y}\,\#\mathcal R_y  \le C_ky^2\,2^{-y}.
\]
The desired estimate follows.
\end{proof}

A word of length $m$ over $\{0,\ldots,k-1\}$ is called \emph{primitive} if its orbit under cyclic permutation has cardinality $m$. 
A primitive cyclic class, also called a primitive necklace, is such an orbit. We use the standard terminology from combinatorics on words, see \cite[Chapter~1]{Lothaire}.

The cyclic-word technique in the following lemma is adapted from the proof of \cite[Theorem~6.1]{QuasWierdl}.  In contrast with that argument, we use ordinary base-$k$ cylinders, confine the digit prescription to finitely many dyadic blocks, and present separate quantitative bounds for the testing set and the associated large-value set.

\begin{lem}\label{lem:exp-small-block}
Let $k\ge 2$, and assume that $D_k\ge 2^{12}$.  There are constants $c_k,C_k>0$ such that the following holds.
Let $y$ be sufficiently large,  $U\subset\{2,3,\ldots\}$ be finite and nonempty, and $(w_t)_{t\in U}$ satisfy
\[
        D_kn_0(y)2^{-y}<w_t<2^{t-y}, \quad t\in U.
\]
Then there exist $\alpha\in[0,1)$ and a Borel set $F_y\subset\mathbb T$ such that
\[
        \lambda(F_y)\le C_k2^{-2y},
\]
and, for the rotation 
\[
        T_\alpha x=x-\alpha\pmod1,
\]
the set
\[
        G_y(\alpha)=\left\{x\in\mathbb T: \max_{t\in U}\frac{2^yw_t}{2^t} \sum_{n=1}^{2^t} \1_{F_y}(T_\alpha^{k^n}x)>1 \right\}
\]
satisfies
\[
        \lambda\bigl(G_y(\alpha)\bigr)\ge  c_k\min\left\{ 1,\, 2^{-y}\sum_{t\in U}w_t \right\}.
\]
\end{lem}

\begin{proof}
Write $ n_0=n_0(y)$. By the definition of $n_0$, we have 
\[
        2^{-2y}\le k^{-n_0}< k2^{-2y}.
\]

For a word $W=w_1\cdots w_{n_0} \in\{0,\ldots,k-1\}^{n_0}$, set
\[
        I_W =\left[ \sum_{j=1}^{n_0}w_jk^{-j},\, \sum_{j=1}^{n_0}w_jk^{-j}+k^{-n_0} \right) \subset \mathbb T.
\]
Put
\[
        F_y = \left\{z\in\mathbb T: \|z\|_{\mathbb T}<k^{-n_0} \right\}, 
\]
where $\|z\|_{\mathbb T} = \min_{m\in\Z}|z-m|$. For all sufficiently large $y$, we have
\[
        \lambda(F_y)=2k^{-n_0}\le 2k 2^{-2y}.
\]
For $\alpha\in [0,1)$, if the canonical base-$k$ expansion of $\alpha$ has the word $W$ in positions $n+1,\ldots,n+n_0$, then $\{k^n\alpha\}\in I_W$.
Consequently, for every $x\in I_W$, 
\begin{equation}\label{eq:exp-cylinder-hit}
        x-k^n\alpha\in F_y\pmod1.
\end{equation}

Let $\mathcal W_{n_0}^{\mathrm{prim}}$ denote the set of primitive words of length $n_0$. If $W$ is nonprimitive, then its least cyclic
period is a proper divisor $d$ of $n_0$. In particular, $d\le n_0/2$, and $W$ is determined by its first $d$ symbols. Consequently,
\[
\#\left(\{0,\ldots,k-1\}^{n_0}\setminus\mathcal W_{n_0}^{\mathrm{prim}}\right)\le \sum_{\substack{d\mid n_0\\ d<n_0}} k^d 
\le\sum_{d=1}^{\lfloor n_0/2\rfloor}k^d\le \frac{k}{k-1} k^{n_0/2}.
\]
Hence
\[
\lambda\left(\bigcup_{W\in\mathcal W_{n_0}^{\mathrm{prim}}}I_W \right) = k^{-n_0}\#\mathcal W_{n_0}^{\mathrm{prim}}
\ge 1-\frac{k}{k-1} k^{-n_0/2}.
\]
Since $n_0=n_0(y)\to \infty$ as $y\to \infty$, it follows that for all sufficiently large $y$,
\begin{equation}\label{eq:exp-primitive-mass}
\lambda\left( \bigcup_{W\in\mathcal W_{n_0}^{\mathrm{prim}}}I_W \right) \ge \frac12.
\end{equation}

For $t\in U$, set
\[
        \xi_t= \frac{2^{t-y+1}}{w_t}, \quad  R_t=\lceil\xi_t\rceil, \quad   L_t=n_0(R_t+1).
\]
Since $w_t<2^{t-y}$, we have $\xi_t>2$, and hence
\[
        R_t<\xi_t+1<\frac32\xi_t, \quad     R_t+1<2\xi_t.
\]
Using the lower bound on $w_t$, we obtain
\[
        L_t<2n_0\xi_t = \frac{4n_0\,2^{t-y}}{w_t} <   \frac4{D_k}\,2^t < 2^{t-4}.
\]

Define $ Q_t =\left\lfloor \frac{2^{t-1}}{L_t}\right\rfloor$. Since the quantity inside the floor is greater than $8$, we have
\[
        Q_t \ge \frac{2^{t-2}}{L_t} >  \frac{2^{t-3}}{n_0\xi_t} =\frac{2^{y-4}w_t}{n_0}.
\]
Thus, together with $k^{-n_0}\ge 2^{-2y}$, it follows that 
\begin{equation}\label{eq:exp-new-class-mass}
        Q_tn_0k^{-n_0} > 2^{-y-4}w_t.
\end{equation}

Process the elements of $U$ in increasing order. For distinct $t$, use the disjoint digit blocks $D_t= \{2^{t-1}+1,\ldots,2^t\}$. 
At stage $t$, partition an initial subblock of $D_t$ into $Q_t$ consecutive segments of length $L_t$.  This is possible because $Q_tL_t\le2^{t-1}=|D_t|$.
If at least $Q_t$ primitive cyclic classes remain unused, assign distinct unused classes to these segments.  
Otherwise, assign all remaining classes and stop the construction.  For a class assigned to a segment, choose a representative $W$ and fill that segment with
$R_t+1$ consecutive copies of $W$. Suppose that such a segment occupies the digit positions $s+1,\ldots,s+L_t$. 
Let $\sigma^uW$ denote the cyclic shift of $W$ by $u$ places, $0\le u<n_0$.  For
\[
        n=s+u+\ell n_0, \quad   0\le\ell<R_t,
\]
the digits in positions $n+1,\ldots,n+n_0$ form the word $\sigma^u W$. Note that 
\[
\begin{aligned}
        n+n_0 &\le  s+(n_0-1)+(R_t-1)n_0+n_0  \\
        &=  s+(R_t+1)n_0-1<  s+L_t.
\end{aligned}
\]
Thus the entire digit window $\sigma^uW$ lies in the prescribed segment.  Moreover,
since the segment is contained in $D_t$, we have 
\[
        s\ge 2^{t-1},   \quad   n+n_0\le2^t,
\]
and in particular $1\le n<2^t$.

Set every digit not prescribed above equal to zero, and let
\[
        \alpha = \sum_{j=1}^{\infty}a_jk^{-j}
\]
be the number represented by the resulting base-$k$ expansion.  Since only finitely many digits are prescribed, the resulting expansion is eventually zero and is therefore the canonical base-$k$ expansion of $\alpha$.
By \eqref{eq:exp-cylinder-hit}, for every $x\in I_{\sigma^uW}$ and every $0\le\ell<R_t$, $\1_{F_y}(x-k^n\alpha)=1$.  The corresponding $R_t$ integers $n$ are distinct and belong to $\{1,\ldots,2^t\}$.  Hence
\[
        \sum_{n=1}^{2^t}  \1_{F_y}(x-k^n\alpha) \ge R_t.
\]
Since $R_t\ge\xi_t$,
\[
        \frac{2^yw_t}{2^t}\sum_{n=1}^{2^t} \1_{F_y}(x-k^n\alpha)\ge \frac{2^yw_tR_t}{2^t} \ge 2>1.
\]
Thus every cylinder belonging to a selected primitive cyclic class is contained in $G_y(\alpha)$.

If the primitive classes are exhausted, by \eqref{eq:exp-primitive-mass} we have
\[
        \lambda\bigl(G_y(\alpha)\bigr)
        \ge\frac12.
\]
Otherwise, exactly $Q_t$ new classes are selected at every stage
$t\in U$.  Since distinct primitive classes consist of disjoint collections of
words, the corresponding cylinders are pairwise disjoint. Then by \eqref{eq:exp-new-class-mass} we have
\[
        \lambda\bigl(G_y(\alpha)\bigr)\ge \sum_{t\in U}Q_tn_0k^{-n_0}  >  2^{-y-4}\sum_{t\in U}w_t.
\]
Therefore, the assertion holds with
\[
        C_k=2k, \quad c_k=2^{-4}.
\]
In fact, the latter constant is independent of $k$.
\end{proof}

\begin{lem}\label{lem:exp-weight-mass}
Fix $q\in\N$. Let $w=(w_t)_{t\ge 2}$ satisfy $w_t=\frac{1}{\Lambda_q(2^t)}$.
Then there exists a strictly increasing sequence of integers $y_j\to \infty$ such that
\[
        l_{y_j}(w)   \gtrsim_{k,q} 1+L_q(2^{y_j}).
\]
Moreover, define 
\[
       U_y(w)= \left\{ t\ge 2: D_k n_0(y)2^{-y}<w_t<2^{t-y} \right\} \ \text{ and } \  l_y(w)=\sum_{t\in U_y(w)}w_t.
\]
It holds that  $ \min U_{y_j}(w)\to  \infty$.
\end{lem}

\begin{proof}
By Lemma~\ref{lem:iterated-sum}, we have  $ w_t \asymp_q \frac{1}{t\Lambda_{q-1}(t)}$. For sufficiently large $y$, set $b_y=\left\lfloor\frac{2^y}{y^2}\right\rfloor$.
We first observe that
\[
        2^{-y}<w_t<2^{t-y}, \quad 3y\le t\le b_y.
\]
Indeed, since $\Lambda_q\ge 1$, we have $w_t\le1$, and hence
\[
        w_t\le1<2^{t-y}
\]
whenever $t\ge3y$. For the lower bound, note that
\[
        b_y\Lambda_{q-1}(b_y)=o_q(2^y).
\]
For $q=1$, this follows immediately from $b_y\le2^y/y^2$.  For $q\ge 2$,
\[
        \Lambda_{q-1}(b_y)\asymp_q  y\Lambda_{q-2}(y)   = y^{1+o_q(1)},
\]
and therefore
\[
        b_y\Lambda_{q-1}(b_y) \le 2^y y^{-1+o_q(1)} = o_q(2^y).
\]
Since $t\mapsto t\Lambda_{q-1}(t)$ is  increasing, we have 
\[
        \inf_{3y\le t\le b_y}2^yw_t \gtrsim_q \frac{2^y} {b_y\Lambda_{q-1}(b_y)} \to  \infty.
\]
Thus $w_t>2^{-y}$ throughout this range  $3y\le t\le b_y$ for all sufficiently large $y$.

Consequently, together with Lemma~\ref{lem:iterated-sum}, it follows that 
\[
        m_y(w) \ge\sum_{3y\le t\le b_y}w_t   \gtrsim_q \sum_{3y\le t\le b_y} \frac{1}{t\Lambda_{q-1}(t)}  \gtrsim_q  L_q(b_y)-L_q(3y).
\]
If $q=1$, then
\[
        L_1(b_y)-L_1(3y)= y\log2-3\log y+O(1).
\]
If $q\ge 2$, then
\[
        L_q(b_y)= L_q(2^y)+O_q(1), \quad L_q(3y) = o_q\bigl(L_q(2^y)\bigr).
\]
This implies that
\[
        m_y(w)\gtrsim_q 1+L_q(2^y)
\]
for all sufficiently large $y$.

Since $y^2 2^{-y}= o\bigl(1+L_q(2^y)\bigr)$, Together with Lemma~\ref{lem:QW-mass-split}, we have 
\[
        l_y(w)+l_{2y}(w)\gtrsim_{k,q}  1+L_q(2^y).
\]
Moreover, note that
\[
        1+L_q(2^{2y})\asymp_q 1+L_q(2^y).
\]
Hence, for every sufficiently large $y$, there exists
$z_y\in\{y,2y\}$ such that
\[
        l_{z_y}(w)\gtrsim_{k,q} 1+L_q(2^{z_y}).
\]
Taking $y_j=z_{3^j}$ for all sufficiently large $j$, we obtain a strictly increasing sequence satisfying the required estimate.

It remains to prove $ \min U_{y_j}(w)\to  \infty$.  Fix $T\ge 2$. Since $w_t=\Lambda_q(2^t)^{-1}$ is decreasing,  $w_t\ge w_T>0$ for $2\le t\le T$.
As $y_j\to \infty$, we have $2^{T-y_j}<w_T$ for all sufficiently large $j$.  Thus, for every $2\le t\le T$,
\[
        2^{t-y_j} \le 2^{T-y_j}<w_T\le w_t,
\]
so the upper inequality in the definition of $U_{y_j}(w)$ fails.  Therefore
\[
        U_{y_j}(w)\cap\{2,\ldots,T\} =\emptyset
\]
for all sufficiently large $j$. By the preceding lower bound, $l_{y_j}(w)>0$. Since
\[
        l_{y_j}(w)= \sum_{t\in U_{y_j}(w)}w_t
\]
and all the summands are positive, the set $U_{y_j}(w)$ is non-empty.  Since $T$ was arbitrary,  $\min U_{y_j}(w)\to \infty$.
\end{proof}

\begin{lem}\label{lem:independent-amplification}
Fix $q\in\N$, a sequence $\mathbf a=(a_n)$, a number $h>0$, a finite nonempty set $\mathcal K\subset \N$, and $0<p\le 1$.
For $1\le\rho\le R$, let $ (\Omega_\rho,\B_\rho,\nu_\rho,S_\rho)$ be measure-preserving systems and $E_\rho\in\B_\rho$.  Suppose that
\[
        \nu_\rho\left\{ \omega_\rho\in \Omega_\rho: \max_{K\in\mathcal K} \frac{h}{K\Lambda_q(K)} \sum_{n=1}^K \1_{E_\rho}(S_\rho^{a_n}\omega_\rho)>1 \right\} \ge p
\]
for every $\1\le \rho\le R$.
Let $ (\Omega^{(R)},\nu^{(R)},S^{(R)}) =  \prod_{\rho=1}^R(\Omega_\rho,\nu_\rho,S_\rho)$ and $\pi_\rho:\Omega^{(R)}\to\Omega_\rho$ be the coordinate maps. Set
$E^{(R)} =\bigcup_{\rho=1}^R\pi_\rho^{-1}(E_\rho)$. 
Then
\[
        \nu^{(R)}(E^{(R)})   \le \sum_{\rho=1}^R\nu_\rho(E_\rho),
\]
and
\[
        \nu^{(R)}\left\{\omega\in \Omega^{(R)}:\max_{K\in\mathcal K}\frac{h}{K\Lambda_q(K)}\sum_{n=1}^K \1_{E^{(R)}}((S^{(R)})^{a_n}\omega)>1 \right\} \ge 1-(1-p)^R.
\]
In particular, if $R=\lceil p^{-1}\rceil$, the last quantity is at least $1-\mathrm e^{-1}$.
\end{lem}

\begin{proof}
The estimate for $E^{(R)}$ follows from the union bound.  Let $G_\rho$ denote the large-value event in the $\rho$-th coordinate.
Product independence implies 
\[
        \nu^{(R)}  \left(\bigcup_{\rho=1}^R\pi_\rho^{-1}(G_\rho) \right) = 1-\prod_{\rho=1}^R\bigl(1-\nu_\rho(G_\rho)\bigr) \ge  1-(1-p)^R.
\]
If $\omega\in\pi_\rho^{-1}(G_\rho)$, we have 
\[
        \1_{E^{(R)}}((S^{(R)})^{a_n}\omega) \ge \1_{E_\rho}(S_\rho^{a_n}\omega_\rho), \quad n\ge 1.
\]
Hence the product large-value event contains $\bigcup_{\rho=1}^R\pi_\rho^{-1}(G_\rho)$.  Finally, we have
\[
        (1-p)^R\le\exp(-pR)\le\mathrm e^{-1}
\]
when $R=\lceil p^{-1}\rceil$.
\end{proof}

\begin{prop}\label{prop:exponential-stages}
Let $q\in\N$ and $\mathbf a=(k^n)_{n\ge 1}$, where $k\ge 2$.  For every  Young function $\Phi$ satisfying
\[
        \Phi(t)=o\bigl(tL_q(t)\bigr) \quad (t\to \infty),
\]
there exists an $(\mathbf a,q,\Phi)$-admissible sequence of finite stages.
\end{prop}

\begin{proof}
Set
\[
        \rho_\Phi(s)=\frac{\Phi(s)} {s\bigl(1+L_q(s)\bigr)}\ \text{ and } \ w_t=\frac{1}{\Lambda_q(2^t)}.
\]
Then $\rho_\Phi(s)\to 0$ as $s\to \infty$.
By Lemma~\ref{lem:exp-weight-mass}, there exists an unbounded set $\mathcal Y\subset \N$ such that
\[
        l_y(w)\gtrsim_{k,q} 1+L_q(2^y), \quad y\in\mathcal Y,
\]
and $\min   U_{y}(w)\to \infty $ as $y\to\infty$ along $\mathcal Y$, where $ U_y(w)= \left\{ t\ge 2: D_k n_0(y)2^{-y}<w_t<2^{t-y} \right\}$.  
For $y\in\mathcal Y$, write $U_y=U_y(w)$, then 
\begin{equation}\label{equation:ly-lower-bound}
   \sum_{t\in U_y}w_t= l_y(w)\gtrsim_{k,q} 1+L_q(2^y).
\end{equation}
Since $w_t\to 0$, the lower inequality in the definition of $U_y$
fails for all sufficiently large $t$.  Hence $U_y$ is finite, and it is nonempty by \eqref{equation:ly-lower-bound}.

After removing finitely many elements from $\mathcal Y$, we may apply Lemma~\ref{lem:exp-small-block} for every $y\in\mathcal Y$.  Thus we obtain $\alpha_y\in[0,1)$ and a Borel set $F_y\subset\mathbb T$ such that
\[
        \lambda(F_y)\le C_k 2^{-2y}
\]
and for the rotation $T_{\alpha_y} x=x-{\alpha_y} \pmod1$, 
\[
    \lambda\left\{ x\in\mathbb T: \max_{t\in U_y} \frac{2^yw_t}{2^t} \sum_{n=1}^{2^t} \1_{F_y}(T_{\alpha_y}^{k^n}x)>1 \right\}
         \ge c_k \min\left\{ 1,\, 2^{-y}\sum_{t\in U_y}w_t \right\}
\]
for constants $c_k, C_k>0$. By \eqref{equation:ly-lower-bound}, there is $c'_{k,q}>0$ such that
\[
        2^{-y}\sum_{t\in U_y}w_t\ge  c'_{k,q}2^{-y}\bigl(1+L_q(2^y)\bigr).
\]
Since $2^{-y}\bigl(1+L_q(2^y)\bigr)\to 0$, we may further discard finitely many elements of $\mathcal Y$ and
assume that
\[
        c'_{k,q}2^{-y}\bigl(1+L_q(2^y)\bigr)\le1.
\]
Set $c_{k,q}=\min\{c_k,1\}c'_{k,q}$.
It follows that the preceding large-value set has measure at least
\[
        p_y  = c_{k,q}2^{-y}\bigl(1+L_q(2^y)\bigr)
\]
with  $0<p_y\le 1$. Put $R_y=\left\lceil p_y^{-1}\right\rceil$. Then
\[
        R_y \le 2 p_y^{-1}  \lesssim_{k,q} \frac{2^y}{1+L_q(2^y)}.
\]
Define
\[
\Omega_y=\mathbb T^{R_y}, \quad \B_y=\B_{\mathbb T}^{\otimes R_y},  \quad  \nu_y=\lambda^{R_y}, \ \ \text{ and } \ \  S_y=T_{\alpha_y}^{\times R_y}.
\]
Equipped with the product metric, $\Omega_y$ is a metric space and $\B_y$ is its Borel $\sigma$-algebra.  
Moreover, $S_y$ is an invertible $\nu_y$-preserving transformation.   

Let $\pi_\rho:\Omega_y\to\mathbb T$ denote the coordinate maps. Set
\[
        E_y  = \bigcup_{\rho=1}^{R_y}\pi_\rho^{-1}(F_y), \ \text{ and } \ \mathcal K_y=\{2^t:t\in U_y\}.
\]
Applying Lemma~\ref{lem:independent-amplification} with $h=2^y$, with
the same system and testing set in each coordinate, yields 
\[
   \nu_y\left\{ \omega\in\Omega_y: \max_{K\in\mathcal K_y} \frac{2^y}{K\Lambda_q(K)} \sum_{n=1}^{K} \1_{E_y}(S_y^{k^n}\omega)>1 \right\}
       \ge 1-\mathrm e^{-1},
\]
and 
\[
        \nu_y(E_y) \le  R_y\lambda(F_y) \lesssim_{k,q} \frac{2^{-y}}{1+L_q(2^y)}.
\]
In particular, $\nu_y(E_y)>0$. Otherwise measure preservation would force every finite sum in the preceding maximal expression to vanish
almost everywhere, contradicting the positive measure of its large-value set.

Fix $j\in\N$.  For $y\in\mathcal Y$, set $H_{j,y}=j2^y$ and  $A_{j,y}=j$, and consider
\[
        \mathfrak s_{j,y} =\bigl(  \Omega_y,\B_y,\nu_y,S_y, E_y,H_{j,y},A_{j,y},\mathcal K_y \bigr).
\]
Since $K=2^t$, $w_t=\Lambda_q(2^t)^{-1}$ and $H_{j,y}/A_{j,y}=2^y$, we have
\[
        \frac{H_{j,y}}{A_{j,y}K\Lambda_q(K)} = \frac{2^y}{2^t\Lambda_q(2^t)} =\frac{2^yw_t}{2^t}.
\]
Hence the preceding large-value event is precisely $\mathcal G(\mathfrak s_{j,y})$, as defined in Definition \ref{def:finite-stage-hypotheses}.  Therefore
\[
        \nu_y\bigl(\mathcal G(\mathfrak s_{j,y})\bigr) \ge 1-\mathrm e^{-1}.
\]

Put $s_{j,y} = 2^jH_{j,y} = 2^jj2^y$. By Lemma~\ref{lem:log-calc}\ref{lem:log-calc:3}, for every fixed $j$ and $q$,
\[
        1+L_q(s_{j,y})  \lesssim_{j,q}  1+L_q(2^y).
\]
Consequently, we have
\[
\begin{aligned}
        2^{-j}\nu_y(E_y)\Phi(2^jH_{j,y})
        &\lesssim_{k,q} 2^{-j}\frac{2^{-y}}{1+L_q(2^y)}\Phi(s_{j,y})      \\
        &= j\,\rho_\Phi(s_{j,y}) \frac{1+L_q(s_{j,y})}{1+L_q(2^y)}      \lesssim_{j,k,q} \rho_\Phi(s_{j,y}).
\end{aligned}
\]
For fixed $j$, we have $s_{j,y}\to \infty$ as $y\to \infty$ along $\mathcal Y$.  Hence
\[
        2^{-j}\nu_y(E_y)\Phi(2^jH_{j,y})\to 0 \quad  \text{as }y\to \infty,\quad y\in\mathcal Y.
\]
We may therefore choose $y_j\in\mathcal Y$ inductively so that for  every $j\ge 1$, $\min U_{y_j}\ge j$ and
\[
        2^{-j}\nu_{y_j}(E_{y_j})\Phi(2^jH_{j,y_j}) \le 2^{-2j},
\]
while $ y_j>y_{j-1}$ for $j\ge 2$. Write
\[
\begin{gathered}
          \Omega_j=\Omega_{y_j}, \quad  \B_{\Omega_j}=\B_{y_j}, \quad \nu_j=\nu_{y_j},  \quad S_j=S_{y_j},      \\
          E_j=E_{y_j},\quad    H_j=j2^{y_j}, \quad     A_j=j, \quad   \mathcal K_j=\mathcal K_{y_j},
\end{gathered}
\]
and define
\[
        \mathfrak s_j  = \bigl(\Omega_j,\B_{\Omega_j},\nu_j,S_j,E_j,H_j,A_j,\mathcal K_j\bigr).
\]
Since $U_{y_j}$ is finite and nonempty, so is $\mathcal K_j$. Moreover, $E_j\in\B_{\Omega_j}$,
$0<\nu_j(E_j)\le1$  and $H_j,A_j>0$. Thus each $\mathfrak s_j$ is a finite stage.

We now verify the three $(\mathbf a,q,\Phi)$-admissibility conditions.  First,
\[
        A_j=j\to  \infty, \ \text{ and }\  \min\mathcal K_j = 2^{\min U_{y_j}} \ge 2^j \to  \infty.
\]
Thus \eqref{eq:finite-stage-escape} holds.  Next,
\[
        \nu_j\bigl(\mathcal G(\mathfrak s_j)\bigr) \ge 1-\mathrm e^{-1},  \quad j\in\N,
\]
so \eqref{eq:finite-stage-success} holds with $c_0=1-\mathrm e^{-1}$. Finally,
\[
        \sum_{j=1}^{\infty} 2^{-j}\nu_j(E_j)\Phi(2^jH_j)\le \sum_{j=1}^{\infty}2^{-2j} < \infty.
\]
This proves \eqref{eq:finite-stage-summability}.  Hence $(\mathfrak s_j)_{j\ge 1}$ is an $(\mathbf a,q,\Phi)$-admissible sequence of finite stages.
\end{proof}

Now we are ready to give the proof of Theorem~\ref{thm:universal-sparse}\ref{thm:universal-sparse:2}.

\begin{proof}[Proof of Theorem~\ref{thm:universal-sparse}\ref{thm:universal-sparse:2}]
Suppose first that $\mathbf a=(a_n)$ has polynomial ratio separation. Proposition~\ref{prop:poly-ratio-stages} shows that, for every
 Young function $\Phi$ satisfying
\[
        \Phi(t)=o\bigl(tL_q(t)\bigr) \quad (t\to \infty),
\]
there exists an $(\mathbf a,q,\Phi)$-admissible sequence of finite stages.  Corollary~\ref{cor:finite-stage-sharp} therefore implies that
$ L\log_{q}L$ is the sharp endpoint for the local observation problem associated with $(\mathbf a,\mathbf D^{(q)})$.  
This proves \ref{thm:universal-sparse}\ref{thm:universal-sparse:2}\ref{thm:universal-sparse:2:1}.

Suppose now that $a_n=k^n$ for some $k\ge 2$. Proposition~\ref{prop:exponential-stages} verifies the same hypothesis of Corollary~\ref{cor:finite-stage-sharp}.  The asserted endpoint sharpness follows, proving \ref{thm:universal-sparse}\ref{thm:universal-sparse:2}\ref{thm:universal-sparse:2:2}.
\end{proof}

\end{document}